\documentclass[11pt]{article}

\usepackage{amsmath, amsfonts, amsthm, amssymb}
\usepackage{authblk, color, bm, graphicx}
\usepackage[margin = 1.5cm, font = small, labelfont = bf, labelsep = period]{caption}
\usepackage{xspace}
\usepackage[utf8]{inputenc}
\usepackage{multirow}
\usepackage[round]{natbib}
\usepackage{float}
\usepackage{bbm}
\usepackage{setspace}
\usepackage{url}

\usepackage{subfig}

\usepackage[algoruled]{algorithm2e}
\SetKwRepeat{Do}{do}{while}

\usepackage{xcolor}
\usepackage[framemethod=TikZ]{mdframed}

\usepackage{fullpage}[2cm]
\newtheorem{thm}{Theorem}
\newtheorem{prop}{Proposition}
\newtheorem{lem}{Lemma}
\newtheorem{obs}{Observation}

\newtheorem{rem}{Remark}
\newtheorem{assm}{Assumption}

\title{Efficient Algorithms for Robust Markov Decision Processes \\ with $s$-Rectangular Ambiguity Sets}

\author[1]{Chin Pang Ho}
\author[2]{Marek Petrik}
\author[3]{Wolfram Wiesemann}

\affil[1]{\small \textit{Department of Data Science, City University of Hong Kong,} \texttt{clint.ho@cityu.edu.hk}}
\affil[2]{\small \textit{Department of Computer Science, University of New Hampshire,} \texttt{mpetrik@cs.unh.edu}}
\affil[3]{\small \textit{Imperial College Business School, Imperial College London,} \texttt{ww@imperial.ac.uk}}

\begin{document}
	
	\maketitle
	
	\begin{abstract}
		Robust Markov decision processes (MDPs) have attracted significant interest due to their ability to protect MDPs from poor out-of-sample performance in the presence of ambiguity. In contrast to classical MDPs, which account for stochasticity by modeling the dynamics through a stochastic process with a known transition kernel, a robust MDP additionally accounts for ambiguity by optimizing against the most adverse transition kernel from an ambiguity set constructed via historical data. In this paper, we develop a unified solution framework for a broad class of robust MDPs with $s$-rectangular ambiguity sets, where the most adverse transition probabilities are considered independently for each state. Using our algorithms, we show that $s$-rectangular robust MDPs with $1$- and $2$-norm as well as $\phi$-divergence ambiguity sets can be solved several orders of magnitude faster than with state-of-the-art commercial solvers, and often only a logarithmic factor slower than classical MDPs. We demonstrate the favorable scaling properties of our algorithms on a range of synthetically generated as well as standard benchmark instances. \\
		
	\end{abstract}
	
	\newpage
	
	\section{Introduction}

	We study robust Markov decision processes (MDPs) with a finite state space $\mathcal{S} = \{ 1, \ldots, S \}$, a finite action space $\mathcal{A} = \{ 1, \ldots, A \}$ and a discrete but infinite planning horizon $t = 0, 1, \ldots, \infty$. All of our results immediately extend to robust MDPs with a finite time horizon. Without loss of generality, we assume that every action $a \in \mathcal{A}$ is admissible in every state $s \in \mathcal{S}$. The MDP starts in a random initial state $\tilde{s}_0$ that follows a known probability distribution $\bm{p}^0$ from the probability simplex $\Delta_S$ in $\mathbb{R}^S$. If action $a \in \mathcal{A}$ is taken in state $s \in \mathcal{S}$, then the MDP transitions randomly to the next state according to the conditional probability distribution $\bm{p}_{sa} \in \Delta_S$, where $p_{sas'}$ describes the probability of each possible follow-up state $s' \in \mathcal{S}$. We collect these transition probabilities into the tensor $\bm{p} \in (\Delta_S)^{S \times A}$, and we assume that the transition probabilities are only known to reside in a non-empty, compact ambiguity set $\mathcal{P} \subseteq (\Delta_S)^{S \times A}$. Imprecise knowledge of the transition probabilities is common in practice, where the unknown true dynamics of a system are approximated by simplified theoretical models or estimated from limited empirical observations. For a transition from state $s \in \mathcal{S}$ to state $s' \in \mathcal{S}$ under action $a \in \mathcal{A}$, the decision maker receives an expected reward of $r_{sas'} \in \mathbb{R}_+$. As with the transition probabilities, we collect these rewards into the tensor $\bm{r} \in \mathbb{R}_+^{S \times A \times S}$. Without loss of generality, we assume that all rewards are non-negative.

	We denote by $\Pi = (\Delta_A)^S$ the set of all stationary (\emph{i.e.}, time- and history-independent) but randomized policies. A policy $\bm{\pi} \in \Pi$ takes action $a \in \mathcal{A}$ in state $s \in \mathcal{S}$ with probability $\pi_{sa}$. The transition probabilities $\bm{p} \in \mathcal{P}$ and the policy $\bm{\pi} \in \Pi$ induce a stochastic process $\{ (\tilde{s}_t, \tilde{a}_t) \}_{t = 0}^\infty$ on the space $(\mathcal{S} \times \mathcal{A})^\infty$ of sample paths. We refer to expectations with respect to this process as $\mathbb{E}^{\bm{p}, \bm{\pi}}$. We assume that the decision maker is risk neutral but ambiguity averse. As such, they wish to maximize the worst-case expected total reward under a discount factor $\lambda \in (0, 1)$,
	\begin{equation}\label{eq:robust_mdp}
	\max_{\bm{\pi} \in \Pi} \; \min_{\bm{p} \in \mathcal{P}} \; \mathbb{E}^{\bm{p}, \bm{\pi}} \left[ \sum_{t = 0}^\infty \lambda^t \cdot r (\tilde{s}_t, \tilde{a}_t, \tilde{s}_{t+1}) \; \Big| \; s_0 \sim \bm{p}^0 \right],
	\end{equation}
	where we use the notation $r (s, a, s') = r_{sas'}$ for better readability. Note that the maximum and minimum in this problem are both attained since $\Pi$ and $\mathcal{P}$ are non-empty and compact, while the objective function is finite for every $\bm{\pi} \in \Pi$ and $\bm{p} \in \mathcal{P}$ since $\lambda < 1$.

    While robust MDPs have been studied since the seventies \citep{SL73:uncertain_mdps}, they have witnessed renewed interest in recent years due to their adoption in applications ranging from assortment optimization \citep{RT12:rmdp_ass_opt}, medical decision-making \citep{ZSD17:rmdp_medical_decisions, GBZSM18:medical_innovs} and hospital operations management \citep{GrandClementChanGoyalEscobar2023}, production planning \citep{XG18:rmdp_inventory_control} and energy systems \citep{HJG18:rmdp_energy} to interdicting nuclear weapons development projects \citep{GKW15:rmdp_interdiction} and the robustification against approximation errors in aggregated MDP models \citep{PS14:rmdp_raam}.
    
    For general ambiguity sets $\mathcal{P}$, evaluating the inner minimization in~\eqref{eq:robust_mdp} is NP-hard even if the policy $\bm{\pi} \in \Pi$ is fixed, and all policies that maximize the worst-case expected total reward may be randomized and history-dependent \citep{WKR13:rmdps}. For these reasons, much of the theoretical research on robust MDPs as well as their applications have focused on $(s, a)$-rectangular ambiguity sets of the form
	\begin{equation}\label{eq:sa_rect}
	    \mathcal{P} = \left\{ \bm{p} \in (\Delta_S)^{S \times A} \, : \, \bm{p}_{sa} \in \mathcal{P}_{sa} \;\; \forall (s, a) \in \mathcal{S} \times \mathcal{A} \right\}, \quad \text{where} \quad \mathcal{P}_{sa} \subseteq \Delta_S, \, (s, a) \in \mathcal{S} \times \mathcal{A}.
	\end{equation}
	Intuitively speaking, $(s, a)$-rectangular robust MDPs hedge against an adversary that can choose the worst transition probabilities $\bm{p}_{sa}$ for each state $s$ and action $a$ separately, irrespective of the transition probabilities $\bm{p}_{s' a'}$ chosen for other state-action pairs $(s', a')$. For $(s, a)$-rectangular ambiguity sets $\mathcal{P}$, there are always optimal policies that are deterministic and stationary \citep{I05:rdp, NG05:rdp}. Moreover, Bellman's optimality principle remains satisfied, which implies that the optimal policy can be determined by a robust variant of the classical value iteration.
	 
	More recently, it has been observed that many of the favorable theoretical properties of $(s, a)$-rectangular ambiguity sets extend to the broader class of $s$-rectangular ambiguity sets $\mathcal{P}$ satisfying
	\begin{equation}\label{eq:s_rect}
        \mathcal{P} = \left\{ \bm{p} \in (\Delta_S)^{S \times A} \, : \, \bm{p}_s \in \mathcal{P}_s \;\; \forall s \in \mathcal{S} \right\}, \quad \text{where} \quad \mathcal{P}_s \subseteq (\Delta_S)^A, \, s \in \mathcal{S},
	\end{equation}
	see \cite{LT07:rmdp_thesis}, \cite{XM12:dr_mdps}, \cite{WKR13:rmdps} and \cite{YX16:dro_mdp}. In contrast to~\eqref{eq:sa_rect}, $s$-rectangular ambiguity sets no longer allow the adversary to freely choose the transition probabilities $\bm{p}_{s1}, \ldots, \bm{p}_{sA}$ that correspond to different actions $a$ applied in the same state $s$. This restricts the conservatism of the resulting robust MDP~\eqref{eq:robust_mdp} and typically leads to a better performance of the optimal policy, for example if the ambiguity set is estimated from data. Although Bellman's optimality principle extends to $s$-rectangular robust MDPs and there is always an optimal stationary policy, all optimal policies of an $s$-rectangular robust MDP may be randomized.\footnote{The necessity to randomize among discrete decisions in the presence of ambiguity has also been recognized in the wider distributionally robust optimization literature, see \cite{DS18:randomized} and \cite{DKW19:dicesions}.} While some studies have explored specific tractable formulations beyond classical rectangularity \citep{GGC18:beyond_rect, LKS25:nonrect}, \cite{GCSW24:tractable_rmdps} recently established that $(s,a)$- and $s$-rectangularity are essentially the only uncertainty structures that admit a general dynamic programming decomposition without further restrictions on the reward structure.

    In this paper, we show that for a broad class of robust MDPs with $s$-rectangular ambiguity sets, the robust Bellman operator, which underlies all efficient solution schemes for robust MDPs (such as robust value and policy iteration), despite requiring the solution of seemingly unstructured min-max problems, can be reduced to the solution of a small number of highly structured projection problems. We use this insight to develop tailored solution schemes for the projection problems corresponding to several popular ambiguity sets, which in turn give rise to efficient solution methods for the respective robust MDPs. We analyze the complexity of our algorithms, and we provide an empirical comparison with state-of-the-art commercial solvers as well as a previously proposed tailored solution scheme. Our contributions may be summarized as follows.
 	\begin{enumerate}
 		\item For a broad class of $s$-rectangular ambiguity sets, we relate the robust Bellman operator to the solution of much simpler projection problems, which preserve the structure of the ambiguity set and, therefore, are well suited for the development of tailored solution schemes.
 		\item We offer a complexity analysis that shows how the exact and approximate solution of the aforementioned projection problems translates into solutions of the robust Bellman operator.
 		\item We develop exact solution schemes for two norm-based ambiguity sets and approximate solution schemes for two $\phi$-divergence ambiguity sets, respectively, and we show that our algorithms compare favorably with a state-of-the-art commercial solver as well as an existing special-purpose solution scheme.
	\end{enumerate}

    An abridged version of this work has appeared as a conference paper \citep{HPW22:phi_divergence_mdps}. While the conference paper focused on robust MDPs with \(\phi\)-divergence ambiguity sets, the present work formulates and analyzes a unified algorithmic framework for robust MDPs with $s$-rectangular ambiguity sets of the form \(\sum_{a} d_a(p_{sa},\bar p_{sa}) \le \kappa\), under which a broad class of divergence- and norm-based models can be treated in a common way. Within this framework, we derive new exact solution schemes for robust MDPs with general weighted \(1\)- and \(2\)-norm ambiguity sets, and we show how these models subsume other divergence-based ambiguity sets, such as the \(\chi^2\)-distance, considered previously. The resulting algorithms—specifically, a geometric breakpoint search method for the weighted \(1\)-norm and a root-finding procedure for a system of nonlinear equations for the weighted \(2\)-norm—are structurally distinct from the solution approaches developed for standard \(\phi\)-divergences. We further provide a refined complexity analysis that characterizes how exact and approximate solutions of the underlying projection subproblems translate into accuracy and runtime guarantees for the robust Bellman iteration. The numerical study is expanded accordingly to cover a broader range of ambiguity sets, benchmark algorithms, and problem instances. Finally, the journal version includes an expanded discussion of the related literature, along with a number of technical generalizations and clarifications that improve the overall exposition.
 
    Efficient implementations of the robust Bellman operator have been first proposed by \cite{I05:rdp} and \cite{NG05:rdp} in the context of robust MDPs with $(s,a)$-rectangular ambiguity sets. The authors study ambiguity sets that bound the distance of the transition probabilities to some nominal distribution in terms of finite scenarios, interval matrix bounds, ellipsoids, the relative entropy, the Kullback-Leibler divergence and maximum a posteriori models. Subsequently, similar methods have been developed by \cite{XYW13:rmdp_radios} for interval matrix bounds as well as likelihood uncertainty models, by \cite{PS14:rmdp_raam} for $1$-norm ambiguity sets as well as by \cite{ZSD17:rmdp_medical_decisions} for interval matrix bounds intersected with a budget constraint. All of these contributions have in common that they focus on $(s,a)$-rectangular ambiguity sets where the existence of optimal deterministic policies is guaranteed, and it is not clear how they could be extended to $s$-rectangular ambiguity sets where all optimal policies may be randomized.
 	
    \cite{WKR13:rmdps} compute the robust Bellman operator of an $s$-rectangular robust MDP as a linear or conic optimization problem using commercial off-the-shelf solvers, which renders their solution scheme suitable primarily for small problem instances. More efficient tailored solution methods for $s$-rectangular robust MDPs have subsequently been developed by \cite{HPW18:fast_bellman}, \cite{BPH21:infty_norm} and \cite{HPW21:ppi}. \cite{HPW18:fast_bellman} develop a homotopy continuation method for robust MDPs with $(s,a)$-rectangular and $s$-rectangular weighted $1$-norm ambiguity sets, while \cite{BPH21:infty_norm} adapt the algorithm of \cite{HPW18:fast_bellman} to unweighted $\infty$-norm ambiguity sets. In contrast to our solution framework, the methods of \cite{HPW18:fast_bellman} and \cite{BPH21:infty_norm} exploit specific properties of $1$- and $\infty$-norm ambiguity sets (in particular the fact that the robust Bellman operator amounts to a linear program whose basic feasible solutions can be enumerated), and they do not appear to extend to other ambiguity sets. \cite{HPW21:ppi} embed the algorithms of \cite{HPW18:fast_bellman} in a partial policy iteration, which generalizes the robust modified policy iteration proposed by \cite{KS13:rob_mod_pi} for $(s,a)$-rectangular robust MDPs to $s$-rectangular robust MDPs. While the present paper focuses on the robust value iteration for ease of exposition, we note that the algorithms presented here can also be combined with partial policy iteration to obtain further speedups. \cite{DGM21:twice} establish an equivalence between $s$-rectangular robust MDPs and twice regularized MDPs, which they subsequently use to propose efficient Bellman updates for a modified policy iteration. While their approach can solve robust MDPs in almost the same time as classical non-robust MDPs, the obtained policies can be conservative as the worst-case transition probabilities are only restricted to belong to a $q$-norm ball and may neither be non-negative nor add up to $1$. Also, the size of the ambiguity set has to be sufficiently small (in particular, inversely proportional to $\sqrt{S}$). \cite{GCK21:first_order_RO}, finally, propose a first-order framework for robust MDPs with $s$-rectangular unweighted $2$-norm and Kullback-Leibler ambiguity sets that interleaves primal-dual first-order updates with approximate value iteration steps. The authors show that their algorithms outperform a robust value iteration that solves the emerging subproblems using state-of-the-art commercial solvers. We show that our solution method for weighted $1$-norm ambiguity sets achieves a time complexity that is one order of magnitude better than the method proposed by \cite{HPW18:fast_bellman}.

    Other than solving the robust Bellman equation, policy gradient methods have been proposed in recent years. \cite{WZ22:robust_policy_gradient} develop a policy gradient method for solving $(s,a)$-rectangular robust MDPs with $R$-contamination uncertainty sets. \cite{LLZ22:robust_policy_gradient} propose a policy mirror descent for $(s,a)$-rectangular robust MDPs that relies on an oracle to compute the robust $Q$ function. For $s$-rectangular robust MDPs, \cite{KDGLM23:robust_policy_gradient}  
    propose a policy gradient method that solves the relaxation of norm-constrained robust MDPs, and \cite{WHP22:robust_policy_gradient} and \cite{LKS25:nonrect} consider the exact model under generic ambiguity sets with higher computational complexity.

    While this paper exclusively studies $s$-rectangular uncertainty sets, we emphasize that alternative generalizations of $(s,a)$-rectangular ambiguity sets have been proposed. For example, \cite{MMX16:k_rect} consider $k$-rectangular ambiguity sets where the transition probabilities of different states can be coupled, \cite{GGC18:beyond_rect} study factor model ambiguity sets where the transition probabilities depend on a small number of underlying factors, and \cite{TCPZ18:policy_conditioned} construct ambiguity sets that bound marginal moments of state-action features defined over entire MDP trajectories. We also note the papers of \cite{XM12:dr_mdps}, \cite{CYH19:dr_mdps} and \cite{GCK21:first_order_Wasserstein}, which study the related problem of \emph{distributionally} robust MDPs whose transition probabilities are themselves regarded as random objects that are drawn from distributions that are only partially known. The first convex optimization formulation for robust MDPs over $(s,a)$-rectangular and $s$-rectangular ambiguity sets has been recently proposed by \cite{GCP23:convex_formulation}. While this result is remarkable in that it opens up the possibility to eventually prove polynomial-time solvability of robust MDPs, it is not currently computationally efficient due to the reliance on exponentially large coefficients. The connections between robust MDPs and multi-stage stochastic programs as well as distributionally robust problems are explored further by \cite{R10:risk_averse} and \cite{S16:rect,S21:rect}.

 	The remainder of this paper proceeds as follows. Section~\ref{sec:fundamentals} relates the robust Bellman operator for our ambiguity sets to a projection subproblem, and it shows how exact and approximate solutions to the latter translate into solutions of the former.  We then present exact solution schemes for the projection subproblems of $1$- and $2$-norm ambiguity sets in Sections~\ref{sec:1-norm} and~\ref{sec:2-norm}, respectively, and we present approximate solution schemes for the projection subproblems of two popular $\phi$-divergence ambiguity sets in Section~\ref{sec:phi_div}. We conclude with numerical experiments in Section~\ref{sec:numericals}. For ease of exposition, some auxiliary proofs and numerical results are relegated to the appendix.
 	 	
 	\textbf{Notation.} $\;$ We denote by $\mathbf{e}$ the vector of all ones, whose dimension will be clear from the context. We refer to the probability simplex in $\mathbb{R}^n$ by $\Delta_n = \{ \bm{p} \in \mathbb{R}^n_+ \, : \, \mathbf{e}^\top \bm{p} = 1 \}$. For a vector $\bm{x} \in \mathbb{R}^n$, we denote by $\text{diag} (\bm{x})$ the matrix with $\bm{x}$ on the diagonal and zeros elsewhere, we let $\min \{ \bm{x} \} = \min \{ x_i \, : \, i = 1, \ldots, n \}$ and similar for the maximum operator, and we define $[ \bm{x} ]_+ \in \mathbb{R}^n_+$ component-wise as $([ \bm{x} ]_+)_i = \max \{ x_i, \, 0 \}$, $i = 1, \ldots, n$. Finally, we refer to the conjugate of a real-valued function $f : \mathbb{R}^n \mapsto \mathbb{R}$ by $f^\star (\bm{y}) = \sup \{ \bm{y}^\top \bm{x} - f (\bm{x}) \, : \, \bm{x} \in \mathbb{R}^n \}$.
 	 	
	\section{Bellman Updates for $s$-Rectangular Robust MDPs}\label{sec:fundamentals}
	
	Throughout this paper, we study robust MDPs with $s$-rectangular ambiguity sets of the form
	\begin{equation}\label{eq:additive_s_rect}
	\mathcal{P}_s = \left\{ \bm{p}_s \in (\Delta_S)^A \, : \, \sum_{a \in \mathcal{A}} d_a (\bm{p}_{sa}, \overline{\bm{p}}_{sa}) \leq \kappa \right\},
	\end{equation}
	where $d_a : \Delta_S \times \Delta_S \mapsto \mathbb{R}$ and $\kappa \in \mathbb{R}$. Contrary to the generic $s$-rectangular ambiguity set~\eqref{eq:s_rect}, we assume in~\eqref{eq:additive_s_rect} that the state-wise ambiguity sets $\mathcal{P}_s$ impose a budget $\kappa$ on the deviations of the action-wise transition probabilities $\bm{p}_{sa}$ from some nominal transition probabilities $\overline{\bm{p}}_{sa} \in \Delta_S$.
	
	We impose the following regularity assumptions on the ambiguity set~\eqref{eq:additive_s_rect}:
	\begin{assm}[Regularity Assumptions on~\eqref{eq:additive_s_rect}]
	~
	\begin{itemize}
		\item[\emph{(\textbf{C})}] The deviation functions $d_a$ are convex in their first argument for fixed second argument.
		\item[\emph{(\textbf{D})}] The deviation functions $d_a$ are non-negative and satisfy $d_a (\bm{p}, \bm{p}') = 0$ if and only if $\bm{p} = \bm{p}'$.
		\item[\emph{(\textbf{K})}] The budget $\kappa$ is strictly positive.
	\end{itemize}
	\end{assm}
	These assumptions are needed both for the tractability of our solution scheme and the analysis of how approximation errors propagate within our algorithm. We emphasize that each of these assumptions is mild and usually satisfied in applications. To simplify the exposition of our complexity results, we also assume that the computation of the deviation functions $d_a$ takes time $\mathcal{O} (1)$ if one of the contributing distributions is a Dirac distribution and time $\mathcal{O} (S)$ otherwise. This is typically the case, and the assumption is satisfied for all ambiguity sets studied in this paper.
	
	
	Our interest in ambiguity sets of the form~\eqref{eq:additive_s_rect} is motivated by the following two examples.
	
	\noindent \textbf{Norm-based ambiguity sets.} $\;$
	Consider the scaled, axis-parallel $q$-norm ball
	\begin{equation}\label{eq:norm_based_amb_set}
	\mathcal{P}_s = \left\{ \bm{p}_s \in (\Delta_S)^A \, : \, \lVert \text{diag} (\bm{\sigma}_s) (\bm{p}_s - \overline{\bm{p}}_s) \rVert_q \leq \rho \right\},
	\end{equation}
	where $\bm{\sigma}_s = (\bm{\sigma}_{s1}, \ldots, \bm{\sigma}_{sA}) \in \mathbb{R}_+^{AS}$, $\overline{\bm{p}}_s \in (\Delta_S)^A$, $q, \rho \in \mathbb{R}$ with $q \geq 1$ and $ \rho > 0$. Of particular interest are the cases $q = 1$, which corresponds to the weighted $1$-norm ambiguity set, and $q = 2$, which corresponds to the weighted $2$-norm ambiguity sets. The norm-based ambiguity set~\eqref{eq:norm_based_amb_set} can be brought into the form~\eqref{eq:additive_s_rect} if we set $d_a (\bm{p}_{sa}, \overline{\bm{p}}_{sa}) = \lVert \text{diag} (\bm{\sigma}_{sa}) (\bm{p}_{sa} - \overline{\bm{p}}_{sa}) \rVert_q^q$ and $\kappa = \rho^q$. We will study $1$-norm and $2$-norm ambiguity sets in Sections~\ref{sec:1-norm} and~\ref{sec:2-norm}, respectively.
	
\noindent \textbf{$\phi$-Divergence ambiguity sets.} $\;$
Consider the $\phi$-divergence ambiguity set
\begin{equation}\label{eq:phi_div_amb_set}
	\mathcal{P}_s = \left\{ \bm{p}_s \in (\Delta_S)^A \, : \, D_\phi (\bm{p}_s \, \| \, \overline{\bm{p}}_s) \leq \rho \right\}
	\quad \text{with} \quad
	D_\phi (\bm{p}_s \, \| \, \overline{\bm{p}}_s) = \sum_{a \in \mathcal{A}} \sum_{s' \in \mathcal{S}} \overline{p}_{sas'} \cdot \phi \left( \frac{p_{sas'}}{\overline{p}_{sas'}} \right),
\end{equation}
where $\phi : \mathbb{R}_+ \mapsto \mathbb{R}_+$ is convex and satisfies $\phi (t) = 0$ if and only if $t=1$. Popular choices include the Kullback-Leibler divergence $\phi (t) = t \log t - t + 1$ and the Burg entropy $\phi (t) = - \log t + t - 1$. A $\phi$-divergence ambiguity set is of the form~\eqref{eq:additive_s_rect} if we set $d_a (\bm{p}_{sa}, \overline{\bm{p}}_{sa}) = \sum_{s' \in \mathcal{S}} \overline{p}_{sas'} \cdot \phi (p_{sas'} / \overline{p}_{sas'})$ and $\kappa = \rho$. We will study $\phi$-divergence ambiguity sets in Section~\ref{sec:phi_div}.
	
We note that there are other classes of $s$-rectangular ambiguity sets~\eqref{eq:s_rect} that \emph{cannot} be expressed in the form~\eqref{eq:additive_s_rect}. This holds true, for example, for ellipsoidal ambiguity sets that are not additively separable across actions, as well as for ambiguity sets where the transition probabilities are modeled as linear combinations of unobserved factors. Our restriction to ambiguity sets of the form~\eqref{eq:additive_s_rect} is motivated by the popularity of the norm and divergence-based ambiguity sets as well as the fact that the associated robust MDPs are amenable to a very efficient solution scheme, as we shall present in the remainder of the paper.
	
A standard approach for determining the optimal value and the optimal policy of the robust MDP~\eqref{eq:robust_mdp} is the robust value iteration \citep{I05:rdp, NG05:rdp, LT07:rmdp_thesis, WKR13:rmdps}: Starting with an initial estimate $\bm{v}^0 \in \mathbb{R}^S$ of the state-wise optimal value to-go, we conduct robust Bellman iterations of the form $\bm{v}^{t+1} \leftarrow \mathfrak{B} (\bm{v}^t)$, $t = 0, 1, \ldots$, where the robust Bellman operator $\mathfrak{B}$ is defined component-wise as
\begin{equation}\label{eq:rob_value_it}
	[\mathfrak{B} (\bm{v})]_s \;\; = \;\;
	\max_{\bm{\pi}_s \in \Delta_A} \; \min_{\bm{p}_s \in \mathcal{P}_s} \; \sum_{a \in \mathcal{A}} \pi_{sa} \cdot \bm{p}_{sa}{}^\top (\bm{r}_{sa} + \lambda \bm{v})
	\qquad \forall s \in \mathcal{S}.
\end{equation}
This yields the optimal value $\bm{p}^0{}^\top \bm{v}^\star$, where the limit $\bm{v}^\star = \lim_{t \rightarrow \infty} \bm{v}^t$ is approached component-wise at a geometric rate. The optimal policy $\bm{\pi}^\star \in \Pi$, finally, is recovered state-wise via
\begin{equation*}
	\bm{\pi}_s^\star \in \mathop{\arg \max}_{\bm{\pi}_s \in \Delta_A} \; \min_{\bm{p}_s \in \mathcal{P}_s} \; \sum_{a \in \mathcal{A}} \pi_{sa} \cdot \bm{p}_{sa}{}^\top (\bm{r}_{sa} + \lambda \bm{v}^\star)
	\qquad \forall s \in \mathcal{S}.
\end{equation*}
At the core of both steps is the solution of a max-min problem. We will show that for $s$-rectangular ambiguity sets of the form~\eqref{eq:additive_s_rect}, this problem can be solved efficiently whenever the following generalized $d_a$-projection of the nominal transition probabilities $\overline{\bm{p}}_{sa}$ can be computed efficiently:
\begin{equation}\label{eq:gen_projection}
	\mathfrak{P} (\overline{\bm{p}}_{sa}; \bm{b}, \beta) =
	\left[
	\begin{array}{l@{\quad}l}
	\text{minimize} & \displaystyle d_a (\bm{p}_{sa}, \overline{\bm{p}}_{sa}) \\
	\text{subject to} & \displaystyle \bm{b}^\top \bm{p}_{sa} \leq \beta \\
	& \displaystyle \bm{p}_{sa} \in \Delta_S
	\end{array}
	\right].
\end{equation}
Here, $\bm{p}_{sa} \in \Delta_S$ are the decision variables and $\overline{\bm{p}}_{sa} \in \Delta_S$, $\bm{b} \in \mathbb{R}^S_+$ and $\beta \in \mathbb{R}_+$ are parameters. Note that problem~\eqref{eq:gen_projection} is feasible if and only if $\min \{ \bm{b} \} \leq \beta$. Moreover, to avoid trivial cases, we assume throughout the paper that $\bm{b}^\top \overline{\bm{p}}_{sa} > \beta$, for otherwise problem~\eqref{eq:gen_projection} is trivially solved by $\overline{\bm{p}}_{sa}$ with an optimal objective value of $0$. Problem~\eqref{eq:gen_projection} is illustrated in Figure~\ref{fig:projection_problem}.
	
\begin{figure}[tb]
	\centering \includegraphics[width = 0.9 \textwidth]{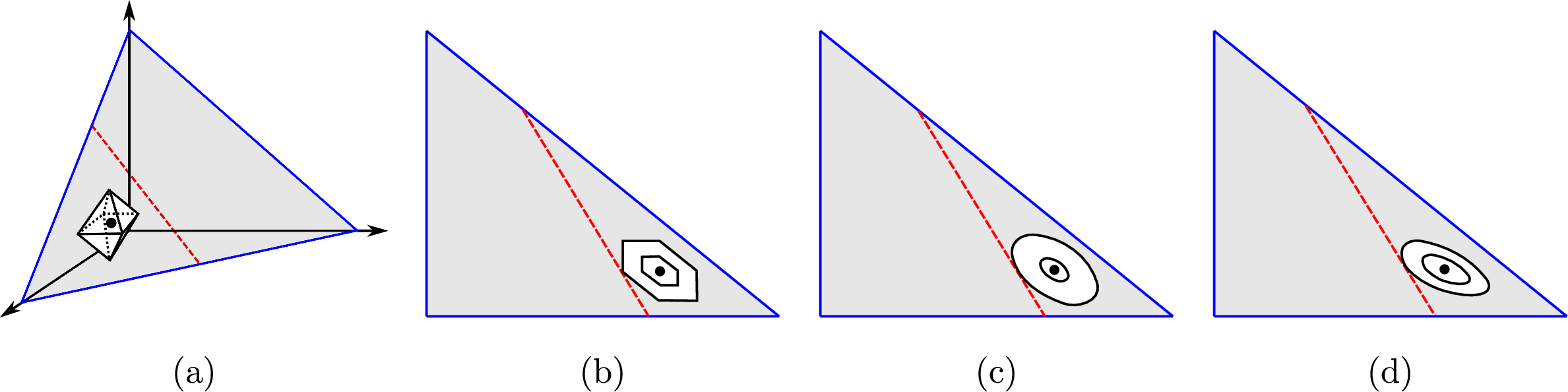}
	\caption{Problem~\eqref{eq:gen_projection} in $S = 3$ dimensions (a) and two-dimensional projections for the 1-norm (b), 2-norm (c) and the Kullback-Leibler divergence (d). The gray shaded areas represent the probability simplex $\Delta_S$, the red dashed lines show the boundaries of the intersections of the halfspaces $\bm{b}^\top \bm{p}_{sa} \leq \beta$ with the probability simplices, and the white shapes illustrate contour lines centered at the nominal transition probabilities $\overline{\bm{p}}_{sa}$. \label{fig:projection_problem}}
\end{figure}
    
Our generalized $d_a$-projection~\eqref{eq:gen_projection} relates to the rich literature on projections onto simplices. One of the most basic variants of this problem, the unweighted $2$-norm projection of a point onto the $S$-dimensional probability simplex, has been used in numerous applications in image processing, finance, optimization and machine learning \citep{C16:projection, AM20:projections}. \cite{M86:projection} proposes one of the earliest algorithms that computes this projection in time $\mathcal{O} (S^2)$ by iteratively reducing the dimension of the problem using Lagrange multipliers. The minimum complexity of $\mathcal{O} (S)$ is achieved, among others, by \cite{MdP89:projection} through a linear-time median-finding algorithm and by \cite{PBFR20:projection} through a filtered bucket-clustering method. The unweighted $2$-norm projection of a point onto the intersection of the $S$-dimensional probability simplex with an axis-parallel hypercube is computed by \cite{WL15:capped} through a sorting-based method and by \cite{AMLHW21:capped} through Newton's method, respectively. \cite{PdMK18:dr_sddp} optimize a linear function over the intersection of a probability simplex with an unweighted $2$-norm constraint through an iterative dimension reduction scheme. \cite{RBHdM19:effective_scenarios} study a variant of this problem where the $2$-norm constraint is replaced with an unweighted $1$-norm constraint, and they identify structural properties of the optimal solutions. \cite{AM20:projections}, finally, study algorithms that optimize linear functions over the intersection of a probability simplex with a constraint that bounds the distance to a nominal distribution. The authors consider as distance measures the unweighted $2$-norm as well as several $\phi$-divergences (including the Kullback-Leibler divergence and the Burg entropy). The algorithms of \cite{PdMK18:dr_sddp}, \cite{RBHdM19:effective_scenarios} and \cite{AM20:projections} can be used to compute the robust Bellman operator for $(s,a)$-rectangular robust MDPs. They do not apply to $s$-rectangular robust MDPs where the uncertainties are coupled across different actions, however, since the ambiguity set $\mathcal{P}_s$ emerges from the intersection of the cross-product of $A$ probability simplices with an additional constraint. Likewise, although the algorithms of \cite{PdMK18:dr_sddp}, \cite{RBHdM19:effective_scenarios} and \cite{AM20:projections} solve problems similar to our generalized $d_a$-projection~\eqref{eq:gen_projection}, they rely on the objective function being linear and the deviation functions participating in the constraints, and it is not clear how to adapt them to problem~\eqref{eq:gen_projection} while preserving their favorable runtimes.
 
 	

Define $\underline{R}_s (\bm{v}) = \max_{a \in \mathcal{A}} \; \min_{s' \in \mathcal{S}} \, \{ r_{sas'} + \lambda v_{s'} \}$ as a lower bound on $[ \mathfrak{B} (\bm{v})]_s$, $s \in \mathcal{S}$, and $\overline{R} = [1 - \lambda]^{-1} \cdot \max \{ r_{sas'} \, : \, s, s' \in \mathcal{S}, \, a \in \mathcal{A} \}$ as a simultaneous upper bound on all $[ \mathfrak{B} (\bm{v}) ]_s$, $\bm{v} \leq \bm{v}^\star$ and $s \in \mathcal{S}$. In the following, we say that for a given approximation $\bm{v}^t \in [0, \overline{R}]^S$ of the optimal value to-go, the robust Bellman iteration~\eqref{eq:rob_value_it} is solved to $\epsilon$-accuracy by any $\bm{v}^{t+1} \in \mathbb{R}^S$ satisfying $\left \lVert \bm{v}^{t+1} - \mathfrak{B} (\bm{v}^t) \right \rVert_\infty \leq \epsilon$. Our interest in $\epsilon$-optimal solutions is motivated by the fact that most of our ambiguity sets are nonlinear, and hence the exact Bellman iterate $\mathfrak{B} (\bm{v}^t)$ may be irrational and thus impossible to compute exactly with a numerical algorithm.

	\begin{thm}\label{thm:overall_complexity:exact_subproblem}
	Assume that the generalized $d_a$-projection~\eqref{eq:gen_projection} can be computed exactly in time $\mathcal{O} (h (S))$. Then the robust Bellman iteration~\eqref{eq:rob_value_it} can be computed for all states $s \in \mathcal{S}$ to accuracy $\epsilon > 0$ in time $\mathcal{O} (A S \cdot h (S) \cdot \log [\overline{R} / \epsilon])$.
	\end{thm}

Recall that for a non-robust MDP, the nominal Bellman operator can be computed for all states $s \in \mathcal{S}$ in time $\mathcal{O} (AS^2)$. Apart from the logarithmic error term $\mathcal{O} (\log [\overline{R} / \epsilon])$, the overhead of computing~\eqref{eq:rob_value_it} thus amounts to $\mathcal{O} (h(S) / S)$.

As we will see in Sections~\ref{sec:1-norm} and~\ref{sec:2-norm}, the projection problems~\eqref{eq:gen_projection} for the $1$-norm and the $2$-norm ambiguity sets can be solved exactly, and their corresponding robust Bellman iterations can thus be computed along the lines of Theorem~\ref{thm:overall_complexity:exact_subproblem}. For divergence-based ambiguity sets, the projection problem is generically nonlinear and can hence not be expected to be solved to exact optimality. To account for this additional complication, we say that for given $\overline{\bm{p}}_{sa} \in \Delta_S$, $\bm{b} \in \mathbb{R}^S_+$ and $\beta \in \mathbb{R}_+$, the generalized $d_a$-projection $\mathfrak{P} (\overline{\bm{p}}_{sa}; \bm{b}, \beta)$ is solved to $\delta$-accuracy by any pair $(\underline{d}, \overline{d}) \in \mathbb{R}^2$ satisfying $\mathfrak{P} (\overline{\bm{p}}_{sa}; \bm{b}, \beta) \in [\underline{d}, \overline{d}]$ and $\overline{d} - \underline{d} \leq \delta$.
	
\begin{thm}\label{thm:overall_complexity:inexact_subproblem}
	Assume that the generalized $d_a$-projection~\eqref{eq:gen_projection} can be computed to any accuracy $\delta > 0$ in time $\mathcal{O} (h (S, \delta))$. Then the robust Bellman iteration~\eqref{eq:rob_value_it} can be computed for all states $s \in \mathcal{S}$ to accuracy $\epsilon > 0$ in time $\mathcal{O} (A S \cdot h (S, \epsilon \kappa / [2A\overline{R} + A \epsilon]) \cdot \log [\overline{R} / \epsilon])$.
	\end{thm}

	The complexity bound in this statement differs from the one of Theorem~\ref{thm:overall_complexity:exact_subproblem} only in the revised complexity estimate $h$ for the projection problem.

\begin{table}
    \centering
    \begin{tabular}{l|cccc}
        \multicolumn{1}{l}{ambiguity set} & $d_a (\bm{p}_{sa}, \overline{\bm{p}}_{sa})$ & complexity $\mathfrak{P}$ & complexity $\mathfrak{B}$ & previous best \\ \hline
        weighted $1$-norm & $\left \lVert \text{diag} (\bm{\sigma}_{sa}) (\bm{p}_{sa} - \overline{\bm{p}}_{sa}) \right \rVert_1$ & $\mathcal{O} (S \log S)$ & $\mathcal{O} (AS^2 \log S)$ & $\mathcal{O} (AS^3 \log S)$ \\
        weighted $2$-norm & $\left \lVert \text{diag} (\bm{\sigma}_{sa}) (\bm{p}_{sa} - \overline{\bm{p}}_{sa}) \right \rVert_2^2$ & $\mathcal{O} (S^2)$ & $\mathcal{O} (AS^3)$ & n/a \\
        Kullback-Leibler & $\sum_{s'} p_{sas'} \log \left( \frac{p_{sas'}}{\overline{p}_{sas'}} \right)$ & $\mathcal{O} (S \log A)$ & $\mathcal{O} (AS^2 \log A)$ & $\mathcal{O} (\ell^2 A S^2)$ \\
        Burg entropy & $\sum_{s'} \overline{p}_{sas'} \log \left( \frac{\overline{p}_{sas'}}{p_{sas'}} \right)$ & $\mathcal{O} (S \log A)$ & $\mathcal{O} (AS^2 \log A)$ & n/a \\ \hline
    \end{tabular}
    \caption{Computational complexity of the generalized $d_a$-projections $\mathfrak{P}$ and the robust Bellman operators $\mathfrak{B}$ for different ambiguity sets, together with the previously best-known results from the literature. The complexity estimates disregard logarithmic error terms that are independent of $S$ and that are characterized in the subsequent sections. \label{tab:complexity_overview}}
\end{table}

Section~\ref{sec:phi_div} discusses the approximate solution of~\eqref{eq:gen_projection} for several divergence-based ambiguity sets. Their corresponding robust Bellman iterations can be computed along the lines of Theorem~\ref{thm:overall_complexity:inexact_subproblem}.

Table~\ref{tab:complexity_overview} summarizes the ambiguity sets studied in this paper, as well as the costs of computing the associated generalized $d_a$-projections~\eqref{eq:gen_projection} and robust Bellman iterations~\eqref{eq:rob_value_it}. We also list the previously best known complexity results for the robust Bellman iterations from the literature as reported by \cite{HPW18:fast_bellman} and \cite{HPW21:ppi} for the weighted $1$-norm and by \cite{GCK21:first_order_RO} for the Kullback-Leibler divergence. Recall that the nominal Bellman operator for non-robust MDPs can be computed in time $\mathcal{O} (AS^2)$, and that the robust Bellman operator for all of the ambiguity sets in Table~\ref{tab:complexity_overview} can be computed in a practical complexity of $\mathcal{O} (A^3 S^3)$ using interior-point methods \citep{BV04:CVXBook}. We conclude that apart from logarithmic error terms which we elaborate upon in the subsequent sections, computing the robust Bellman operator in our framework adds a logarithmic overhead (in the weighted $1$-norm, the Kullback-Leibler divergence and the Burg entropy) or a linear overhead (in the weighted $2$-norm) over the computation of the nominal Bellman operator.

\section{$1$-Norm Ambiguity Sets}\label{sec:1-norm}

We first assume that the deviation measure $d_a$ in the $s$-rectangular ambiguity set~\eqref{eq:additive_s_rect} satisfies $d_a (\bm{p}_{sa}, \overline{\bm{p}}_{sa}) = \left \lVert \text{diag} (\bm{\sigma}_{sa}) (\bm{p}_{sa} - \overline{\bm{p}}_{sa}) \right \rVert_1$ with $\bm{\sigma}_{sa} > \bm{0}$ component-wise. For ease of notation, we define the diagonal matrix $\bm{\Sigma}_{sa} = \text{diag} (\bm{\sigma}_{sa}) \succ \bm{0}$. $1$-norm ambiguity sets model variation distances \citep{BL15:phi_div, HPW22:phi_divergence_mdps} and admit a Bayesian interpretation \citep{RP19:bayesian}, and they can be calibrated to historical observations of the MDP's transitions via Hoeffding-style bounds \citep{WOSVW03:l1} or their relationship to Kullback-Leibler ambiguity sets \citep{I05:rdp}. In this paper, we assume that the generalized projection problem is feasible, that is, $\min \{ \bm{b} \} \leq \beta$. If this condition does not hold, then the outer bisection method fails, and one must search for a larger value of $\beta$ as part of the bisection procedure.

\begin{prop}
	Assume that the projection problem~\eqref{eq:gen_projection} is feasible, that is, that $\min \{ \bm{b} \} \leq \beta$. Then, for the deviation measure $d_a (\bm{p}_{sa}, \overline{\bm{p}}_{sa}) = \left \lVert \bm{\Sigma}_{sa} (\bm{p}_{sa} - \overline{\bm{p}}_{sa}) \right \rVert_1$, the optimal value of~\eqref{eq:gen_projection} equals the optimal value of the univariate convex optimization problem
	\begin{equation}\label{eq:norm_1:projection_problem}
		\begin{array}{l@{\quad}l}
			\text{\emph{maximize}} & \displaystyle \displaystyle (\overline{\bm{p}}_{sa}{}^\top \bm{b} - \beta) \cdot \alpha - \overline{\bm{p}}_{sa}{}^\top [ \bm{b} \alpha - \min \{ \bm{b} \alpha + \bm{\sigma}_{sa} \} \cdot \mathbf{e} - \bm{\sigma}_{sa} ]_+ \\
			\text{\emph{subject to}} & \displaystyle \alpha \in \mathbb{R}_+.
		\end{array}
	\end{equation}
\end{prop}

\begin{proof}
	Strong linear programming duality, which holds since problem~\eqref{eq:gen_projection} is feasible and bounded, implies that the optimal value of~\eqref{eq:gen_projection} coincides with the optimal value of its dual. After some basic algebraic manipulations, the dual of problem~\eqref{eq:gen_projection} becomes
	\begin{equation*}
		\begin{array}{l@{\quad}l}
			\text{maximize} & \displaystyle (\overline{\bm{p}}_{sa}{}^\top \bm{b} - \beta) \cdot \alpha - \overline{\bm{p}}_{sa}{}^\top \bm{\gamma} \\
			\text{subject to} & \displaystyle \left \lVert \bm{\Sigma}_{sa}^{-1} (-\bm{b} \alpha + \zeta \mathbf{e} + \bm{\gamma}) \right \rVert_\infty \leq 1 \\
			& \displaystyle \alpha \in \mathbb{R}_+, \;\; \zeta \in \mathbb{R}, \;\; \bm{\gamma} \in \mathbb{R}^S_+.
		\end{array}
	\end{equation*}
	For given $\alpha$ and $\zeta$, the first constraint is satisfied by any $\bm{\gamma} \in \mathbb{R}^S_+ \cap [\bm{b} \alpha - \zeta \mathbf{e} - \bm{\sigma}_{sa}, \, \bm{b} \alpha - \zeta \mathbf{e} + \bm{\sigma}_{sa}]$. In particular, such a $\bm{\gamma}$ exists if and only if $\bm{b} \alpha - \zeta \mathbf{e} + \bm{\sigma}_{sa} \geq \bm{0}$. Whenever this is the case, an optimal choice of $\bm{\gamma}$ is given by $\bm{\gamma}^\star = [ \bm{b} \alpha - \zeta \mathbf{e} - \bm{\sigma}_{sa} ]_+$ since the term $-\overline{\bm{p}}_{sa}{}^\top \bm{\gamma}$ in the objective function is non-increasing in $\bm{\gamma}$. We can thus remove $\bm{\gamma}$ from the problem to obtain the equivalent formulation
	\begin{equation*}
		\begin{array}{l@{\quad}l}
			\text{maximize} & \displaystyle (\overline{\bm{p}}_{sa}{}^\top \bm{b} - \beta) \cdot \alpha - \overline{\bm{p}}_{sa}{}^\top [ \bm{b} \alpha - \zeta \mathbf{e} - \bm{\sigma}_{sa} ]_+ \\
			\text{subject to} & \displaystyle \bm{b} \alpha + \bm{\sigma}_{sa} \geq \zeta \mathbf{e} \\
			& \displaystyle \alpha \in \mathbb{R}_+, \;\; \zeta \in \mathbb{R}.
		\end{array}
	\end{equation*}
	Note that the objective function above is non-decreasing in $\zeta$ since $\overline{\bm{p}}_{sa}$ is component-wise non-negative. For a given $\alpha$, the largest $\zeta$ that satisfies the first constraint in the problem is $\zeta^\star = \min \{ \bm{b} \alpha + \bm{\sigma}_{sa} \}$. Removing $\zeta$ then yields the formulation~\eqref{eq:norm_1:projection_problem} in the statement of the proposition. This is a convex optimization problem as it maximizes a piecewise linear concave objective function over $\mathbb{R}_+$.
\end{proof}

\begin{obs}
	For $S \geq 2$, the objective function in problem~\eqref{eq:norm_1:projection_problem} is piecewise affine with at most $2 S - 3$ breakpoints.
\end{obs}

\begin{proof}
	We define the vector-valued function $\bm{f} : \mathbb{R}_+ \mapsto \mathbb{R}^S$ through 
    $\bm{f} (\alpha) = \bm{b} \alpha - \min \{ \bm{b} \alpha + \bm{\sigma}_{sa} \} \mathbf{e} - \bm{\sigma}_{sa}$, which corresponds to the expression inside the $[ \cdot ]_+$-operator in the objective function of problem~\eqref{eq:norm_1:projection_problem}. Note that the components $f_i$, $i \in \mathcal{S}$, of this function are piecewise affine, they share the same breakpoints, and they satisfy $f_i (0) < 0$. Since each component $f_i$ can have at most $S - 1$ different strictly positive slopes, there can be at most $S - 2$ breakpoints $\alpha$ of $\bm{f}$ where at least one component $f_i (\alpha)$ is strictly positive. Moreover, every component $f_i$ that corresponds to the slope $b_i \in \arg \min \{ b_j \, : \, j \in \mathcal{S} \}$ remains negative for all $\alpha \in \mathbb{R}_+$. Thus, applying the $[ \cdot ]_+$-operator on $\bm{f}$, as done in the objective function of~\eqref{eq:norm_1:projection_problem}, introduces at most one additional breakpoint for at most $S - 1$ components of $\bm{f}$. In summary, the objective function in~\eqref{eq:norm_1:projection_problem} thus has at most $(S - 2) + (S - 1) = 2 S - 3$ breakpoints.
\end{proof}

Next, we present an algorithm that determines all breakpoints as well as the optimal objective value of~\eqref{eq:norm_1:projection_problem} in time $\mathcal{O} (S \log S)$. Our algorithm proceeds in four steps. Steps~1 and~2 compute the line segments and breakpoints of the expression $\min \{ \bm{b} \alpha + \bm{\sigma}_{sa} \}$ inside the objective function of~\eqref{eq:norm_1:projection_problem}. Step~3 determines the additional breakpoints that are introduced by the $[\cdot]_+$-operator in~\eqref{eq:norm_1:projection_problem}, and Step~4 traverses all breakpoints to compute the optimal value of~\eqref{eq:norm_1:projection_problem}. 

\begin{figure}[tb]
	\includegraphics[width = \textwidth]{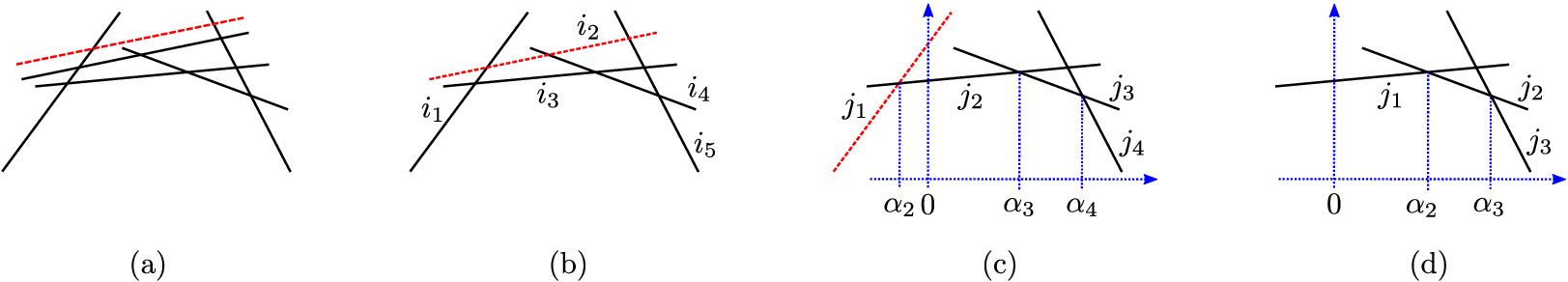}
	\caption{From left to right: lines $b_i \alpha + \sigma_{sai}$, $i = 1, \ldots, S$, before Step~1 (a); lines $b_{i_j} \alpha + \sigma_{sa i_j}$, $i_1, \ldots, i_m$, after Step~1 (b); lines $b_{j_k} \alpha + \sigma_{sa j_k}$, $j_1, \ldots, j_n$, after Step~2 (c); line segments $b_{j_k} \alpha + \sigma_{sa j_k}$, $j_1, \ldots, j_n$ (possibly relabeled), after removing negative breakpoints $\alpha_j$ (d). Note that $\bm{b} \geq \bm{0}$ and hence all line segments are non-decreasing in this section; we chose to include negative slopes in this figure to aid visual clarity. \label{fig:norm1_algo}}
\end{figure}

Step~1 sorts the components $b_1, \ldots, b_S$ of $\bm{b}$ in descending order. Moreover, whenever two components of $\bm{b}$ are equal, we remove the component whose associated entry in $\bm{\sigma}_{sa}$ is larger (ties are broken arbitrarily). In summary, the first step results in an index list $i_1, \ldots, i_m$, $m \leq S$, such that $b_{i_1} > \ldots > b_{i_m}$. Step~1 is illustrated in Figure~\ref{fig:norm1_algo}~(a) and~(b). Among the lines $b_i \alpha + \sigma_{sai}$, $i = 1, \ldots, S$, in Figure~\ref{fig:norm1_algo}~(a), the red dashed line is redundant and hence removed. Moreover, the remaining lines in Figure~\ref{fig:norm1_algo}~(b) satisfy that those participating in the concave envelope of $\min \{ \bm{b} \alpha + \bm{\sigma}_{sa} \}$, that is, the lines $i_1$, $i_3$, $i_4$ and $i_5$, are ordered from left to right. Using a heapsort or merge sort algorithm, Step~1 can be completed in time $\mathcal{O} (S \log S)$. If after Step~1 we have $m = 1$, that is, only a single line is left, then the expression inside the $[\cdot]_+$-operator in problem~\eqref{eq:norm_1:projection_problem} is linear and the problem can be solved directly. Thus, we assume in the following that $m > 1$.


\begin{algorithm}[tb]
	\KwData{Sorted index list $i_1, \ldots, i_m$ of the components of $\bm{b}$ such that $b_{i_1} > \ldots > b_{i_m}$}
	Push $(i_1, \alpha_1)$ onto $H$, where $\alpha_1 = - \infty$; ~terminate if $m = 1$\;
	Push $(i_2, \alpha_2)$ onto $H$, where $\alpha_2$ is the intersection of $b_{i_1} \alpha + \sigma_{sai_1}$ and $b_{i_2} \alpha + \sigma_{sai_2}$\;
	\For {$k = 3, \dots, m$}
	{
		\Repeat{$|H| = 1$ or $H$ remained unchanged}
		{
			Let $(i', \alpha')$ be the element on top of $H$\;
			Let $\alpha_k$ be the intersection of $b_{i_k} \alpha + \sigma_{sai_k}$ and $b_{i'} \alpha + \sigma_{sai'}$\;
			\textbf{if} $\alpha_k  \leq \alpha'$ \textbf{then} pop the top-most element $(i', \alpha')$ from $H$; \textbf{end} \\
		}
		Push $(i_k, \alpha_k)$ onto $H$\;
	}
	\KwResult{Sorted lists of supporting line indices $j_1, \ldots, j_n$ and endpoints $\alpha_1 < \ldots < \alpha_n$ in $H$}
	\caption{Breakpoints of $\min \{ \bm{b} \alpha + \bm{\sigma}_{sa} \}$ in problem~\eqref{eq:norm_1:projection_problem} \label{alg:1norm:dual_graham_scan}}
\end{algorithm}

In Step~2, we use Algorithm~\ref{alg:1norm:dual_graham_scan} to construct the concave envelope of the mapping $\alpha \mapsto \min \{ \bm{b} \alpha + \bm{\sigma}_{sa}\}$ in problem~\eqref{eq:norm_1:projection_problem}. The algorithm uses a stack $H$ to store the indices $j_k$ of the supporting line segments $b_{j_k} \alpha + \sigma_{s a j_k}$ of this mapping, as well as the breakpoints $\alpha_k$, $k = 2, \ldots, n$, between each successive pair of line segments $b_{j_{k-1}} \alpha + \sigma_{s a j_{k-1}}$ and $b_{j_k} \alpha + \sigma_{s a j_k}$. Algorithm~\ref{alg:1norm:dual_graham_scan} is illustrated in Figure~\ref{fig:norm1_algo}~(b) and~(c): taking the indices $i_1, \ldots, i_m$ with $m = 5$ as input, the algorithm removes the redundant line segment $i_2$ and returns the indices $j_1, \ldots, j_n$ with $n = 4$ as output.

\begin{lem}\label{lem:step_2}
	Algorithm~\ref{alg:1norm:dual_graham_scan} terminates in time $\mathcal{O} (S)$ and returns the rays $b_{j_1} \alpha + \sigma_{saj_1}$, $\alpha \in (- \infty, \alpha_2]$, and $b_{j_n} \alpha + \sigma_{saj_n}$, $\alpha \in [\alpha_n, +\infty)$, as well as the line segments $b_{j_k} \alpha + \sigma_{saj_k}$, $\alpha \in[\alpha_k, \alpha_{k+1}]$ and $k = 2, \ldots, n - 1$, that form the concave envelope of $\min \{ \bm{b} \alpha + \bm{\sigma}_{sa} \}$.
\end{lem}

\begin{proof}
	We show that the rays and line segments from the statement of the lemma form the concave envelope of $\min \{ b_{i_k} \alpha + \sigma_{sai_k} \,: \, k = 1, \ldots, m \}$. The statement then follows from the fact that the concave envelope of $\min \{ \bm{b} \alpha + \bm{\sigma}_{sa} \}$ does not change if we omit the indices $i \in \mathcal{S} \setminus \{ i_1, \ldots, i_m \}$ from the minimum that were removed in Step~1 of our overall procedure.
	
	We show the statement via an induction on the iteration counter $k$ in the for-loop. Assume that at the beginning of iteration $k$, $H$ contains the elements $(j_1, \alpha_1), \ldots, (j_l, \alpha_l)$ and that the concave envelope of $\min \{ b_{j_\ell} \alpha + \sigma_{saj_\ell} \,: \, \ell = 1, \ldots, k - 1 \}$ is given by the rays $b_{j_1} \alpha + \sigma_{saj_1}$, $\alpha \in (- \infty, \alpha_2]$, and $b_{j_l} \alpha + \sigma_{saj_l}$, $\alpha \in [\alpha_l, +\infty)$, as well as the line segments $b_{j_\ell} \alpha + \sigma_{saj_\ell}$, $\alpha \in[\alpha_\ell, \alpha_{\ell+1}]$ and $\ell = 2, \ldots, l - 1$. Note that this indeed holds true initially when $k = 3$, in which case $l = 2$. We claim that at the end of iteration $k$, the elements $(j'_1, \alpha'_1), \ldots, (j'_{l'}, \alpha'_{l'})$ in $H$ form the concave envelope of $\min \{ b_{j_\ell} \alpha + \sigma_{saj_\ell} \,: \, \ell = 1, \ldots, k \}$. Indeed, let $j_{l''}$, $1 \leq l'' \leq l$, be the index of the line $b_{i_{l''}} \alpha + \sigma_{sai_{l''}}$ that intersects the line $b_{i_k} \alpha + \sigma_{sai_k}$ last (\emph{i.e.}, at the largest value of $\alpha$) among $j_1, \ldots, j_l$, and let $\alpha'$ be the point of intersection. The line $b_{i_k} \alpha + \sigma_{sai_k}$ supports the concave envelope of $\min \{ b_{j_\ell} \alpha + \sigma_{saj_\ell}\,: \, \ell = 1, \ldots, k \}$ for $\alpha \geq \alpha'$ since $b_{i_k} < b_{j_\ell}$ for $\ell = 1, \ldots, l$. Thus, the new concave envelope consists of the lines $b_{j_\ell} \alpha + \sigma_{saj_\ell}$, $\ell = 1, \ldots, l''$, as well as $b_{i_k} \alpha + \sigma_{sai_k}$. The inner loop in Algorithm~\ref{alg:1norm:dual_graham_scan} removes all redundant lines $b_{j_\ell} \alpha + \sigma_{saj_\ell}$, $\ell = l''+1, \ldots, l$ since their breakpoints satisfy $\alpha_{l''} < \alpha_{l''+1} < \ldots < \alpha_l$ due to the induction hypothesis. Since each index $i_k$, $k = 1, \ldots, m$, is pushed onto and popped from the stack at most once, Algorithm~\ref{alg:1norm:dual_graham_scan} runs in linear time $\mathcal{O} (S)$.
\end{proof}

Figures~\ref{fig:norm1_algo}~(b) and~(c) illustrate the induction step in the proof of Lemma~\ref{lem:step_2}: In iteration $k = 3$, the algorithm confirms that the intersection $\alpha_k$ of the line segments $i_2$ and $i_3$ (denoted by $i'$ and $i_k$ in the algorithm description, respectively) lies to the left of the intersection $\alpha'$ of the line segments $i_1$ and $i_2$. Thus, the line segment $i_2$ is removed from the stack. We remark that the point-line duality of projective geometry allows us to interpret Algorithm~\ref{alg:1norm:dual_graham_scan} as a dual version of Graham's scan for finding the convex hull of a finite set of points~\citep{BCKO08:comp_geometry}.

Before we commence with Step~3 of our algorithm, we remove any line segment-endpoint pairs $(j_k, \alpha_k)$, $k = 1, \ldots, n - 1$, whose associated right-sided endpoints $\alpha_{k+1}$ are negative. This intermediate step, which is admissible since any feasible decision $\alpha$ in problem~\eqref{eq:norm_1:projection_problem} is non-negative, leads to the removal of the line segment $j_1$ in Figure~\ref{fig:norm1_algo}~(c) and~(d). To reduce the notational burden, we refer to the remaining supporting line indices and endpoints as $j_1, \ldots, j_n$ and $\alpha_1, \ldots, \alpha_n$, even though some pairs may have been removed and $n$ may have been decreased accordingly. As before, we assume in the following that $n > 1$ as otherwise problem~\eqref{eq:norm_1:projection_problem} can be solved directly.

\begin{algorithm}[tb]
	\KwData{Sorted lists of supporting line indices $j_1, \ldots, j_n$ and breakpoints $\alpha_2, \ldots, \alpha_n$}
	\For{$s' = 1, 2, \ldots, S$}{
		\uIf{$b_{s'} = b_{j_n}$}
		{
			set $\alpha^0_{s'} = +\infty$\;
		}
		\uElseIf{$(b_{s'} - b_{j_n}) \, \alpha_n \leq \sigma_{s a s'} + \sigma_{saj_n}$}
		{
			let $\alpha^0_{s'}$ be the solution to $(b_{s'} - b_{j_n}) \, \alpha = \sigma_{s a s'} + \sigma_{sa j_n}$\;
		}
		\uElseIf{$(b_{s'} - b_{j_1}) \, \alpha_2 \geq \sigma_{s a s'} + \sigma_{saj_1}$}
		{
			let $\alpha^0_{s'}$ be the solution to $(b_{s'} - b_{j_1}) \, \alpha = \sigma_{s a s'} + \sigma_{sa j_1}$\;
		}
		\Else
		{			
			bisect on $\{ j_k \}_{k=2}^{n-1}$ to find the line segment $j_l$ that satisfies: \\
			$\qquad$ \emph{(i)} $(b_{s'} - b_{j_l}) \, \alpha_l \leq \sigma_{sas'} + \sigma_{sa j_l}$ $\quad$ and $\quad$ \emph{(ii)} $(b_{s'} - b_{j_l}) \, \alpha_{l+1} \geq \sigma_{sas'} + \sigma_{sa j_l}$\;
			let $\alpha^0_{s'}$ be the solution to $(b_{s'} - b_{j_l}) \alpha = \sigma_{sas'} + \sigma_{sa j_l}$\;
		}
	}
	\KwResult{Roots $\alpha^0_{s'}$ (or $\alpha^0_{s'} = +\infty$ if no intersection) for all components $s' = 1, \ldots, S$}
	\caption{Additional breakpoints introduced by the $[\cdot]_+$-operator in problem~\eqref{eq:norm_1:projection_problem} \label{alg:1norm:plus_BkPts}}
\end{algorithm}

\begin{figure}[tb]
	\centering \includegraphics[width = \textwidth]{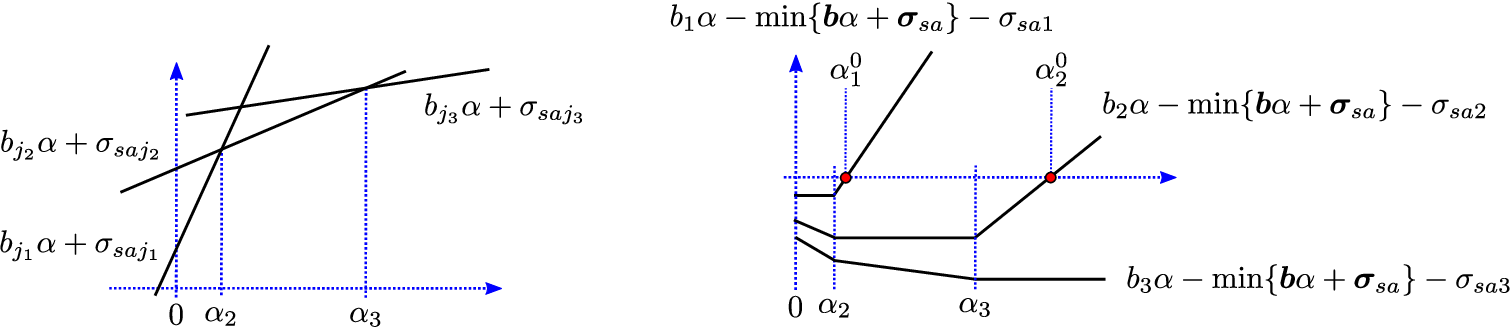}
	\caption{For the expression $\min \{ \bm{b} \alpha + \bm{\sigma}_{sa} \}$ plotted in the left graph, Algorithm~\ref{alg:1norm:plus_BkPts} determines the two additional breakpoints $\alpha^0_1$ and $\alpha^0_2$ in the right graph. \label{fig:norm1_algo-2}}
\end{figure}

In Step~3, we use Algorithm~\ref{alg:1norm:plus_BkPts} to determine for each component $s' \in \mathcal{S}$ of the vector-valued expression $\bm{b} \alpha - \min \{ \bm{b} \alpha + \bm{\sigma}_{sa} \} \cdot \mathbf{e} - \bm{\sigma}_{sa}$ in the objective function of problem~\eqref{eq:norm_1:projection_problem} the additional breakpoint $\alpha^0_{s'}$ introduced by the $[ \cdot ]_+$-operator (or set $\alpha^0_{s'} = + \infty$ if no such breakpoint exists). Algorithm~\ref{alg:1norm:plus_BkPts} is illustrated in Figure~\ref{fig:norm1_algo-2}: For the expression $\min \{ \bm{b} \alpha + \bm{\sigma}_{sa} \}$ plotted in the left graph, the right graph plots the three components of $\bm{b} \alpha - \min \{ \bm{b} \alpha + \bm{\sigma}_{sa} \} \cdot \mathbf{e} - \bm{\sigma}_{sa}$, together with the two additional breakpoints $\alpha^0_{s'}$ determined by the algorithm (the third breakpoint $\alpha^0_3$ is set to $+ \infty$).

\begin{lem}
	Algorithm~\ref{alg:1norm:plus_BkPts} terminates in time $\mathcal{O} (S \log S)$ and returns all roots of the components of $\bm{b} \alpha - \min \{ \bm{b} \alpha + \bm{\sigma}_{sa} \} \cdot \mathbf{e} - \bm{\sigma}_{sa}$ over $\alpha \in \mathbb{R}_+$.
\end{lem}

\begin{proof}
	We show that for each state $s' \in \mathcal{S}$, the algorithm correctly identifies the root $\alpha^0_{s'} \geq 0$ satisfying $b_{s'} \alpha^0_{s'} - \min \{ \bm{b} \alpha^0_{s'} + \bm{\sigma}_{sa} \} - \sigma_{sas'} = 0$, or it sets $\alpha^0_{s'} = + \infty$ if no such root exists. To this end, we consider each of the four if-cases inside the loop in Algorithm~\ref{alg:1norm:plus_BkPts} separately.
	
	Consider first the case when $b_{s'} = b_{j_n}$. Note that $b_{i_1} > b_{i_2} > \ldots > b_{i_n}$ after Step~1, and the subsequent steps keep the same order for the indices $j_k$ (although some indices may be removed). Thus, if $b_{s'} = b_{j_n}$, then $b_{s'} \leq b_i$ for all $i \in \mathcal{S}$, and hence the expression $b_{s'} \alpha - \min \{ \bm{b} \alpha + \bm{\sigma}_{sa} \} - \sigma_{sas'}$ is non-positive throughout $\alpha \in \mathbb{R}_+$, implying that indeed $\alpha^0_{s'} = +\infty$.
	
	In the second case, if $(b_{s'} - b_{j_n}) \, \alpha_n \leq \sigma_{s a s'} + \sigma_{saj_n}$, then $b_{s'} \alpha_n - \min \{ \bm{b} \alpha_n + \bm{\sigma}_{sa} \} - \sigma_{s a s'} \leq 0$ since $\min \{ \bm{b} \alpha_n + \bm{\sigma}_{sa} \} = b_{j_n} \alpha_n + \sigma_{sa j_n}$ (\emph{cf.}~Lemma~\ref{lem:step_2}). In other words, the expression $b_{s'} \alpha - \min \{ \bm{b} \alpha + \bm{\sigma}_{sa} \} - \sigma_{s a s'}$ is non-positive for all $\alpha \in [0, \alpha_n]$, and thus the intersection of this expression with $0$ must be located at the ray $b_{s'} \alpha - (b_{j_n} \alpha + \sigma_{sa j_n}) - \sigma_{s a s'}$, $\alpha \in [\alpha_n, +\infty)$ since the slope $b_{s'} - b_{j_n}$ of this ray is positive (for otherwise, $s'$ would have been subsumed by the first case). 
    
    The third if-case treats the situation where the intersection with $0$ is located at the ray $b_{s'} \alpha - (b_{j_1} \alpha + \sigma_{sa j_1}) - \sigma_{s a s'}$, $\alpha \in (-\infty, \alpha_2]$. We skip the argument for this case as it is analogous to the second one.
	
	In the fourth case, finally, we know that the function $b_{s'} \alpha - \min \{ \bm{b} \alpha + \bm{\sigma}_{sa} \} - \sigma_{sas'}$ must intersect $0$ at one of the line segments $b_{s'} \alpha - (b_{j_k} \alpha + \sigma_{sa j_k}) - \sigma_{sas'}$, $\alpha \in [\alpha_k, \alpha_{k+1}]$ and $k = 2, \ldots, n - 1$. We can thus bisect on these segments to find the segment $j_l$ whose left-sided breakpoint $b_{s'} \alpha_l - (b_{j_l} \alpha_l + \sigma_{sa j_l}) - \sigma_{sas'}$ is non-positive and whose right-sided breakpoint $b_{s'} \alpha_{l+1} - (b_{j_l} \alpha_{l+1} + \sigma_{sa j_l}) - \sigma_{sas'}$ is non-negative. Through the exclusion of the previous cases, such a segment $j_l$ must exist.
	
	In the worst case, Algorithm~\ref{alg:1norm:plus_BkPts} performs a bisection over up to $\mathcal{O} (S)$ line segments for each of the $S$ states, which results in the stated complexity bound of $\mathcal{O} (S \log S)$.
\end{proof}

\begin{algorithm}[tb]
	\KwData{Merged list $\alpha'_1, \ldots, \alpha'_k$ of breakpoints and supporting line indices $j_1, \ldots, j_n$}
	\textbf{Initialization:} Set
	$\alpha = 0$, $\;$
	$f = 0$, $\;$
	$\nabla f_1 = \overline{\bm{p}}_{sa}{}^\top \bm{b} - \beta$, $\;$
	$\nabla f_2 = 0$, $\;$
	$\nabla f_{\min} = b_{j_1}$, $\;$
    $\ell = 1$, $\;$
	$\overline{p}_\Sigma = 0$\;
	\For{$i = 1, 2, \ldots, k$}
	{
		~\\[-3mm]
		\textbf{if} $\nabla f_1 - \nabla f_2 \leq 0$ \textbf{then} terminate; \textbf{end} \\
		set $f = f + (\nabla f_1 - \nabla f_2) (\alpha'_i - \alpha)$ and $\alpha = \alpha'_i$\;
		\uIf{breakpoint $\alpha'_i$ is a root $\alpha^0_{s'}$}
		{
			set $\nabla f_2 = \nabla f_2 + \overline{p}_{s a s'} \cdot(b_{s'} - \nabla f_{\min})$ and $\overline{p}_\Sigma = \overline{p}_\Sigma + \overline{p}_{s a s'}$\;
		}
		\ElseIf{breakpoint $\alpha'_i$ is a breakpoint $\alpha_j$}
		{
			set $\ell = \ell + 1$, $\nabla f_2 = \nabla f_2 - (b_{j_{\ell}} - \nabla f_{\min}) \cdot \overline{p}_\Sigma$ and $\nabla f_{\min} = b_{j_\ell}$\;
		}
	}
	\KwResult{Optimal value $f$ of problem~\eqref{eq:norm_1:projection_problem}}
	\caption{Optimal value of problem~\eqref{eq:norm_1:projection_problem} \label{alg:1norm:findopt}}
\end{algorithm}

After Step~3, we sort the roots $\alpha^0_{s'}$, $s' \in \mathcal{S}$ with $\alpha^0_{s'} < +\infty$, and merge the list of breakpoints $\alpha_2, \ldots, \alpha_n$ with the list of roots $\alpha^0_{s'}$ in non-descending order to a new list $\alpha'_1, \ldots, \alpha'_k$. When a breakpoint $\alpha_i$ coincides with a root $\alpha^0_{s'}$, we keep both elements in the list. Both operations can be achieved in $\mathcal{O} (S \log S)$ time. Algorithm~\ref{alg:1norm:findopt} then traverses all breakpoints and roots and determines the optimal objective value of problem~\eqref{eq:norm_1:projection_problem}. In this algorithm, $\alpha$ denotes the current breakpoint or root, $f$ is the objective value of~\eqref{eq:norm_1:projection_problem} evaluated at $\alpha$, $\nabla f_1$ and $\nabla f_2$ denote the slope of the first and second summand in the objective function of~\eqref{eq:norm_1:projection_problem}, respectively, $\nabla f_{\min}$ is the current slope of $\min \{ \bm{b} \alpha + \bm{\sigma}_{sa} \}$, and $\overline{p}_\Sigma = 0$ stores the sum of all $\overline{p}_{sas'}$, $s' \in \mathcal{S}$, that multiply positive components of $[ \bm{b} \alpha - \min \{ \bm{b} \alpha + \bm{\sigma}_{sa} \} \cdot \mathbf{e} - \bm{\sigma}_{sa} ]_+$. Through updates at each breakpoint $\alpha'_i$ and termination once the objective function no longer increases, the algorithm solves problem~\eqref{eq:norm_1:projection_problem} as desired.

\begin{obs}
	Algorithm~\ref{alg:1norm:findopt} terminates in time $\mathcal{O} (S)$ and returns the optimal value of problem~\eqref{eq:norm_1:projection_problem}.
\end{obs}

In conclusion, we thus arrive at the following main result of this section.

\begin{thm}
	For the deviation measure $d_a (\bm{p}_{sa}, \overline{\bm{p}}_{sa}) = \left \lVert \bm{\Sigma}_{sa} (\bm{p}_{sa} - \overline{\bm{p}}_{sa}) \right \rVert_1$, the projection problem~\eqref{eq:gen_projection} can be solved in time $\mathcal{O} (S \log S)$. 
\end{thm}

Thus, the overall time complexity of the associated value iteration is $\mathcal{O} (S^2 A \cdot \log S \cdot \log [\overline{R} / \epsilon])$. \cite{HPW18:fast_bellman} and \cite{HPW21:ppi} develop a homotopy continuation method for robust MDPs with $(s,a)$-rectangular weighted $1$-norm ambiguity sets, and they extend their method to $s$-rectangular ambiguity sets with a complexity of $\mathcal{O} (S^3 A \cdot \log [S A])$. Our method is faster by a factor of $\mathcal{O} (S)$, as our additional multiplicative factor is independent of the MDP size. This improvement over the homotopy method comes from reducing the complexity of identifying the linear segments of the function $q(\xi)$ in equation~(5.1) of \cite{HPW21:ppi} by a factor of $\mathcal{O} (S)$.

\section{$2$-Norm Ambiguity Sets}\label{sec:2-norm}

We now assume that the deviation measure $d_a$ in the $s$-rectangular ambiguity set~\eqref{eq:additive_s_rect} satisfies $d_a (\bm{p}_{sa}, \overline{\bm{p}}_{sa}) = \left \lVert \bm{\Sigma}_{sa} (\bm{p}_{sa} - \overline{\bm{p}}_{sa}) \right \rVert_2^2$ with $\bm{\Sigma}_{sa} = \text{diag} (\bm{\sigma}_{sa})$ and $\bm{\sigma}_{sa} > \bm{0}$ component-wise. Weighted $2$-norm ambiguity sets generalize ambiguity sets with modified $\chi^2$-distance \citep{HPW22:phi_divergence_mdps}, and they can be calibrated to historical observations of the MDP's transitions via statistical information theory \citep{I05:rdp, NG05:rdp} as well as classical results from Markov chains \citep{WKR13:rmdps}. As in the previous section, we assume that the generalized projection problem is feasible, that is, that $\min \{ \bm{b} \} \leq \beta$.

\begin{prop}\label{prop:2norm}
	For the deviation measure $d_a (\bm{p}_{sa}, \overline{\bm{p}}_{sa}) = \left \lVert \bm{\Sigma}_{sa} (\bm{p}_{sa} - \overline{\bm{p}}_{sa}) \right \rVert_2^2$, the optimal value of~\eqref{eq:gen_projection} equals $0$ if $\bm{b}^\top \overline{\bm{p}}_{sa} \leq \beta$ and
	\begin{equation}\label{eq:q:2norm:resulting_value}
	- \beta \alpha + \gamma + \overline{\bm{p}}_{sa}{}^\top \min \left\{ \bm{b} \alpha - \gamma \mathbf{e}, \, 2 \bm{\Sigma}_{sa}^2 \overline{\bm{p}}_{sa} \right\} - \frac{1}{4} \left \lVert \bm{\Sigma}_{sa}^{-1} \min \left\{ \bm{b} \alpha - \gamma \mathbf{e}, \, 2 \bm{\Sigma}_{sa}^2 \overline{\bm{p}}_{sa} \right\} \right \rVert_2^2
	\end{equation}
	otherwise, where $(\alpha, \gamma) \in \mathbb{R}^2$ is any solution to the (solvable) system of nonlinear equations
	\begin{equation}\label{eq:2norm:nonlinear_eqs}
	\begin{array}{r@{}l}
    \displaystyle \mathbf{e}^\top \bm{\Sigma}_{sa}^{-2} \left[ - \bm{b} \alpha + \gamma \mathbf{e}+ 2 \bm{\Sigma}_{sa}^2 \overline{\bm{p}}_{sa} \right]_+ &= 2 \\
	\displaystyle \bm{b}^\top \bm{\Sigma}_{sa}^{-2} \left[ - \bm{b} \alpha + \gamma \mathbf{e} + 2 \bm{\Sigma}_{sa}^2 \overline{\bm{p}}_{sa} \right]_+ &= 2 \beta .
	\end{array}
	\end{equation}
\end{prop}

The equation system~\eqref{eq:2norm:nonlinear_eqs} may have multiple solutions. If $S = 2$, $\bm{b} = (3/4, 1/4)^\top$, $\beta = 1/4$, $\bm{\Sigma} = \bm{I}$ and $\overline{\bm{p}}_{sa} = \mathbf{e} / 2 \in \Delta_2$, for example, any $\alpha \geq 4$ and $\gamma = \alpha / 4 + 1$ solve~\eqref{eq:2norm:nonlinear_eqs}. In such cases, however, the expression~\eqref{eq:q:2norm:resulting_value} attains the same value for all admissible combinations of $\alpha$ and $\gamma$.

\begin{proof}[Proof of Proposition~\ref{prop:2norm}]
	For the stated deviation measure, the projection problem~\eqref{eq:gen_projection} becomes
	\begin{equation}\label{eq:2norm:intermediate_prob}
	\begin{array}{l@{\quad}l}
	\text{minimize} & \displaystyle \left \lVert \bm{\Sigma}_{sa} (\bm{p}_{sa} - \overline{\bm{p}}_{sa}) \right \rVert_2^2 \\
	\text{subject to} & \displaystyle \bm{b}^\top \bm{p}_{sa} \leq \beta \\
	& \displaystyle \bm{p}_{sa} \in \Delta_S.
	\end{array}
	\end{equation}
	If $\bm{b}^\top \overline{\bm{p}}_{sa} \leq \beta$, then problem~\eqref{eq:2norm:intermediate_prob} is solved by $\bm{p}^\star_{sa} = \overline{\bm{p}}_{sa}$ and its optimal value is $0$. We now assume that $\bm{b}^\top \overline{\bm{p}}_{sa} > \beta$, in which case any optimal solution $\bm{p}^\star_{sa}$ to~\eqref{eq:2norm:intermediate_prob} must satisfy $\bm{b}^\top \bm{p}^\star_{sa} = \beta$ as otherwise convex combinations of $\bm{p}^\star_{sa}$ and $\overline{\bm{p}}_{sa}$ would allow us to further reduce the objective value of~\eqref{eq:2norm:intermediate_prob}.
	
	Our previous reasoning allows us to study the variant of problem~\eqref{eq:2norm:intermediate_prob} that replaces the inequality $\bm{b}^\top \bm{p}_{sa} \leq \beta$ with an equality. The Lagrange dual function associated with this problem is
	\begin{equation*}
	g (\alpha, \gamma, \bm{\zeta}) = \inf \left\{ \left \lVert \bm{\Sigma}_{sa} (\bm{p}_{sa} - \overline{\bm{p}}_{sa}) \right \rVert_2^2 + \alpha (\bm{b}^\top \bm{p}_{sa} - \beta) + \gamma (1 - \mathbf{e}^\top \bm{p}_{sa}) - \bm{\zeta}^\top \bm{p}_{sa} \, : \, \bm{p}_{sa} \in \mathbb{R}^S \right\},
	\end{equation*}
	where $\alpha, \gamma \in \mathbb{R}$ and $\bm{\zeta} \in \mathbb{R}^S_+$. The first-order unconstrained optimality condition implies that the infimum in this expression is attained by
	\begin{equation}\label{eq:2norm:optimal_p}
	\bm{p}_{sa} = \overline{\bm{p}}_{sa} + \frac{1}{2} \bm{\Sigma}_{sa}^{-2} (- \alpha \bm{b} + \gamma \mathbf{e} + \bm{\zeta}).
	\end{equation}
	Note that the dual variables $\bm{\zeta} \in \mathbb{R}^S_+$ associated with the non-negativity constraints of $\bm{p}_{sa} \in \Delta_S$ could be avoided by taking the infimum over all non-negative vectors $\bm{p}_{sa} \in \mathbb{R}^S_+$ in the Lagrange dual; the associated first-order \emph{constrained} optimality condition would lead to a non-smooth dependence of $\bm{p}_{sa}$ on $\alpha$ and $\gamma$, however, which would complicate our argument below.

	Substituting the expression~\eqref{eq:2norm:optimal_p} into the Lagrange dual function, we obtain the dual to~\eqref{eq:2norm:intermediate_prob}:
	\begin{equation}\label{eq:2norm:dual_prob}
	\begin{array}{l@{\quad}l}
	\text{maximize} & \displaystyle - \beta \alpha + \gamma + \overline{\bm{p}}_{sa}{}^\top (\alpha \bm{b} - \gamma \mathbf{e} - \bm{\zeta}) - \frac{1}{4} \left \lVert \bm{\Sigma}_{sa}^{-1} (\alpha \bm{b} - \gamma \mathbf{e} - \bm{\zeta}) \right \rVert_2^2 \\
	\text{subject to} & \displaystyle \alpha, \gamma \in \mathbb{R}, \;\; \bm{\zeta} \in \mathbb{R}^S_+
	\end{array}
	\end{equation}
	Strong duality holds since the dual problem is strictly feasible by construction. In the remainder, we argue that any optimal solution $(\alpha^\star, \gamma^\star, \bm{\zeta}^\star)$ to this problem is defined through the nonlinear equation system~\eqref{eq:2norm:nonlinear_eqs}, which will complete the proof.
	
	The first-order \emph{unconstrained} optimality conditions for $\alpha$, $\gamma$ and $\bm{\zeta}$ in~\eqref{eq:2norm:dual_prob} are, respectively,
	\begin{subequations}
	\begin{align}
	- \beta + \bm{b}^\top \overline{\bm{p}}_{sa} - \frac{1}{2} \bm{b}^\top \bm{\Sigma}_{sa}^{-2} (\bm{b} \alpha - \gamma \mathbf{e} - \bm{\zeta}) &= 0 \label{eq:2norm:unconstr_opt_conds_1} \\
	1 - \mathbf{e}^\top \overline{\bm{p}}_{sa} - \frac{1}{2} \mathbf{e}^\top \bm{\Sigma}_{sa}^{-2} (- \bm{b} \alpha + \gamma \mathbf{e} + \bm{\zeta}) &= 0 \label{eq:2norm:unconstr_opt_conds_2} \\
	- \overline{\bm{p}}_{sa} + \frac{1}{2} \bm{\Sigma}_{sa}^{-2} (\bm{b} \alpha - \gamma \mathbf{e} - \bm{\zeta}) &= 0.
	\end{align}
	\end{subequations}
	Note that the objective function is additively separable in the components of $\bm{\zeta}$. We thus conclude that the optimal value of $\bm{\zeta}$, which is restricted to the non-negative orthant, must satisfy $\bm{\zeta}^\star = [\bm{b} \alpha^\star - \gamma^\star \mathbf{e} - 2 \bm{\Sigma}_{sa}^2 \overline{\bm{p}}_{sa}]_+$. Substituting this expression into the first two optimality conditions~\eqref{eq:2norm:unconstr_opt_conds_1} and~\eqref{eq:2norm:unconstr_opt_conds_2} as well as the objective function of problem~\eqref{eq:2norm:dual_prob}, respectively, yields the nonlinear equation system~\eqref{eq:2norm:nonlinear_eqs} as well as the expression~\eqref{eq:q:2norm:resulting_value}.
\end{proof}


We next develop algorithms that jointly solve the system of nonlinear equations~\eqref{eq:2norm:nonlinear_eqs}. We first determine the complete solution set for each equation in~\eqref{eq:2norm:nonlinear_eqs} individually (Section~\ref{sec:2norm:indiv_algorithm}). To this end, note that each of the two equations in~\eqref{eq:2norm:nonlinear_eqs} can be expressed as
\begin{equation}\label{eq:2norm:nonlinear_genericform}
    \bm{a}^\top [ -\bm{b} \alpha + \gamma \mathbf{e} + \bm{c} ]_+
    \; = \;
    \rho,
\end{equation}
where $\bm{a}, \bm{b}, \bm{c} \geq \bm{0}$ and $\rho > 0$. We furthermore assume that $\bm{a} > \mathbf{0}$ component-wise as indices $s \in \mathcal{S}$ with $a_s = 0$ have no impact on the solution set of equation~\eqref{eq:2norm:nonlinear_genericform}. To compute a solution $(\alpha^\star, \gamma^\star)$ that satisfies both equations in~\eqref{eq:2norm:nonlinear_eqs} simultaneously, finally, we compute the intersection of the two individual solution sets in an efficient way (Section~\ref{sec:2norm:joint_algorithm}).

\subsection{Solution Set of the Nonlinear Equation~\eqref{eq:2norm:nonlinear_genericform}}\label{sec:2norm:indiv_algorithm}

We simplify the exposition of our algorithm by making the following regularity assumption. 

\begin{assm}\label{assm:2norm:simple_algo}
    The parameters $\bm{a}$, $\bm{b}$, $\bm{c}$ and $\rho$ in~\eqref{eq:2norm:nonlinear_genericform} satisfy the following conditions.
    \begin{itemize}
        \item[\emph{(i)}] For all $s, t \in \mathcal{S}$ with $s \neq t$, we have $b_s \neq b_t$.
        \item[\emph{(ii)}] For every set $\mathcal{S}' \subseteq \mathcal{S}$ and $s \in \mathcal{S}$, we have $\rho + \sum_{s' \in \mathcal{S}'} a_{s'} (-c_{s'} + c_s) \neq 0$.
        \item[\emph{(iii)}] For every set $\mathcal{S}' \subseteq \mathcal{S}$ and $s, t \in \mathcal{S}$, $\rho + \sum_{s' \in \mathcal{S}'} a_{s'}  (b_{s'} - b_s) \left( \alpha_{s,t} - \alpha_{s,s'} \right) \neq 0$, where $\alpha_{ij} = (c_i - c_j) / (b_i - b_j)$ denotes the intersection of the lines $b_i \alpha - c_i$ and $b_j \alpha - c_j$, $i, j \in \mathcal{S}$.
    \end{itemize}
\end{assm}

As we will see below, the first condition in Assumption~\ref{assm:2norm:simple_algo} implies that for $\alpha$ sufficiently large, equation~\eqref{eq:2norm:nonlinear_genericform} reduces to $a_s [ -b_s \alpha + \gamma + c_s ]_+ = \rho$, where $s \in \mathcal{S}$ is the index whose associated component $b_s$ of $\bm{b}$ is the smallest. This allows us to readily identify the optimal solutions $(\alpha^\star, \gamma^\star)$ when $\alpha^\star$ is sufficiently large. The second condition in Assumption~\ref{assm:2norm:simple_algo} ensures that except for finitely many breakpoints $\alpha \in \mathbb{R}$, each component $-b_s \alpha + \gamma + c_s$ in~\eqref{eq:2norm:nonlinear_genericform} is either positive or negative. The third condition in Assumption~\ref{assm:2norm:simple_algo}, finally, ensures that no two different components $s, t \in \mathcal{S}$, $s \neq t$, in~\eqref{eq:2norm:nonlinear_genericform} vanish at the same value of $\alpha \in \mathbb{R}$. Together, these three conditions simplify the bookkeeping of our algorithm.

We emphasize that Assumption~\ref{assm:2norm:simple_algo} is not needed \emph{per se}, but it simplifies the exposition and analysis of our algorithm by eliminating tedious but otherwise straightforward corner cases. Moreover, one can readily show that when the conditions in Assumption~\ref{assm:2norm:simple_algo} are violated, they can thus always be satisfied by applying an arbitrarily small perturbation to the parameters $\bm{a}$, $\bm{b}$, $\bm{c}$ and $\rho$ that in turn has an arbitrarily small impact on the accuracy of the solution set satisfying equation~\eqref{eq:2norm:nonlinear_genericform}.

%

\begin{prop}\label{prop:gamma_star_alpha}
    For every $\alpha \in \mathbb{R}$, there is a unique $\gamma^\star (\alpha) \in \mathbb{R}$ such that $\bm{a}^\top \left[ -\bm{b} \alpha + \gamma^\star (\alpha) \mathbf{e} + \bm{c} \right]_+ = \rho$. Moreover, the function $\gamma^\star$ is piecewise affine, monotonically non-decreasing and concave. 
\end{prop}

\begin{proof}
    To see the existence of $\gamma^\star (\alpha)$, define the function $f (\gamma) = \bm{a}^\top [ -\bm{b} \alpha + \gamma \mathbf{e} + \bm{c} ]_+$ for any fixed $\alpha \in \mathbb{R}$. One readily observes that $f (\gamma)$ vanishes for $\gamma \in \mathbb{R}$ sufficiently small, that $f (\gamma) > \rho$ for $\gamma \in \mathbb{R}$ sufficiently large, and that $f (\gamma)$ is continuous. As for the uniqueness of $\gamma^\star (\alpha)$, we first note that any $\gamma \in \mathbb{R}$ satisfying $f (\gamma) = \rho$ must imply that $f (\gamma) > 0$ since $\rho > 0$. One readily observes, however, that $f (\gamma)$ is strictly monotonically increasing over $\{ \gamma \in \mathbb{R} \, : \, f (\gamma) > 0 \}$.
    
    To see that $\gamma^\star$ is piecewise affine, monotonically non-decreasing and concave, we observe that $\gamma^\star$ can be interpreted as the optimal value function of the parametric linear program
    \begin{equation*}
        \alpha \;\; \mapsto \;\;
        \left[
	    \begin{array}{l@{\quad}l}
	        \text{maximize} & \displaystyle \gamma \\
	        \text{subject to} & \displaystyle \bm{a}^\top \bm{\theta} \leq \rho \\
	        & \displaystyle \bm{\theta} \geq -\bm{b} \alpha + \gamma \mathbf{e}  + \bm{c} \\
	        & \displaystyle \gamma \in \mathbb{R}, \;\; \bm{\theta} \in \mathbb{R}_+^S.
	    \end{array}
	    \right],
	\end{equation*}
	and the stated properties then directly follow from standard results in quantitative stability analysis of linear programs, see, \emph{e.g.}, \citet[\S 29]{R97:convex_analysis}.
\end{proof}

\begin{figure}[tb]
	\centering \includegraphics[width = 0.5 \textwidth]{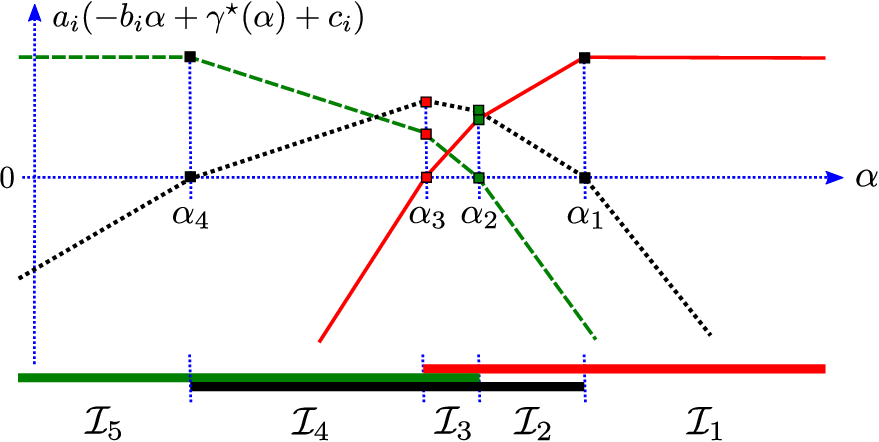}
	\caption{For the three component functions $a_i (- b_i \alpha + \gamma^\star (\alpha) + c_i)$ in solid red, dotted black and dashed green, Algorithm~\ref{alg:2norm:gamma_form} computes the solution set of equation~\eqref{eq:2norm:nonlinear_genericform} in $5$ iterations with the breakpoints $\alpha_1, \ldots, \alpha_4$. The colored bars below the graph indicate which component function indices $i$ are contained in each set $\mathcal{I}_t$. \label{fig:norm2_algo-1}}
\end{figure}

Algorithm~\ref{alg:2norm:gamma_form} iteratively computes the complete solution set of equation~\eqref{eq:2norm:nonlinear_genericform} by tracing the real line from $\alpha \rightarrow +\infty$ to $\alpha \rightarrow -\infty$. In each iteration $t$, the set $\mathcal{I}_t$ comprises the indices of the component functions $a_i (- b_i \alpha + \gamma^\star (\alpha) + c_i)$ that are positive on the interval $(\alpha_t, \alpha_{t - 1})$. The knowledge of this set renders the solution of the nonlinear equation~\eqref{eq:2norm:nonlinear_genericform} straightforward. Breakpoints occur between subsequent iterations when one of the component functions switches sign and thus the set $\mathcal{I}_t$ needs to be updated. Since each component function $a_i (- b_i \alpha + \gamma^\star (\alpha) + c_i)$ is concave and---by the proof of Theorem~\ref{thm:2norm_simple_alg_works}---never vanishes on an interval of positive width, the number of sign switches (and thus the number of algorithm iterations) is limited by $2S$. Figure~\ref{fig:norm2_algo-1} illustrates the algorithm on an instance of equation~\eqref{eq:2norm:nonlinear_genericform} with $S = 3$ component functions.

\begin{algorithm}[tb]
	\KwData{$\bm{a} > \bm{0}$, $\bm{b}, \bm{c} \geq \bm{0}$ and $\rho > 0$.}
	\textbf{Initialization:} Sort $\{ b_s \}_{s = 1}^S$ in ascending order and update $\bm{a}$ and $\bm{c}$ accordingly. \\ $\mspace{122mu}$ Set $\mathcal{I}_1 = \{ 1 \}$ and $\alpha_0 = \infty$. \\
	\For{$t = 1, 2, \ldots$}
	{
		~\\[-3mm]
		Define the affine function
		\begin{equation*}
		    \gamma_t (\alpha)
		    \; = \;
		    \frac{\rho + \sum_{s \in \mathcal{I}_t} a_s (b_s \alpha - c_s) }{\sum_{s \in \mathcal{I}_t} a_s}
		    \; =: \;
		    m_t \alpha + v_t.
		\end{equation*}		
		~\\[-3mm]		
		Compute the breakpoint
		\begin{equation*}
		    \alpha_t = \max \Big\{ \alpha \in [-\infty, \, \alpha_{t - 1}) \, : \, -b_s \alpha + \gamma_t (\alpha) + c_s = 0 \text{ for some $s \in \mathcal{S}$} \Big\}.
		\end{equation*}
		~$\mspace{-16mu}$\uIf{$\alpha_t > -\infty$}
		{
		    Let $s_t \in \mathcal{S}$ be the index satisfying $-b_{s_t} \alpha_t + \gamma_t (\alpha_t) + c_{s_t} = 0$\;
			\textbf{if} $s_t  \in \mathcal{I}_t$ \textbf{then} set $\mathcal{I}_{t + 1} = \mathcal{I}_t \setminus \{ s_t \}$ \textbf{else} set $\mathcal{I}_{t + 1} = \mathcal{I}_t \cup \{ s_t \}$ \textbf{end}\;
		}
		\Else
		{
			Set $T = t$ and terminate\;
		}
	}
	\KwResult{List of $(\alpha_t, m_t, v_t)$, $t = 1, \ldots, T$, such that $\bm{a}^\top \left[ -\bm{b} \alpha + [m_t \alpha + v_t] \mathbf{e} + \bm{c} \right]_+ = \rho$ for all $\alpha \in [\alpha_t, \alpha_{t-1}]$ and all $t = 1, \ldots, T$.}
	\caption{Complete solution set for equation~\eqref{eq:2norm:nonlinear_genericform} \label{alg:2norm:gamma_form}}
\end{algorithm}

\begin{thm}\label{thm:2norm_simple_alg_works}
    Algorithm~\ref{alg:2norm:gamma_form} terminates in time $\mathcal{O} (S^2)$ with an output satisfying $(\alpha_0, \alpha_T) = (\infty, -\infty)$ and $\bm{a}^\top \left[ -\bm{b} \alpha + [m_t \alpha + v_t] \mathbf{e} + \bm{c} \right]_+ = \rho$ for all $\alpha \in [\alpha_t, \alpha_{t-1}]$ and all $t = 1, \ldots, T$, where $T \leq 2S$.
\end{thm}


\begin{proof}
    We first show that in each iteration $t$ of the algorithm, the index set $\mathcal{I}_t$ satisfies
    \begin{align*}
        \mathcal{I}_t = \mathcal{I}^+ (\alpha)
        \quad \text{with} \quad
        &\mathcal{I}^+ (\alpha) = \left\{ s \in \mathcal{S} \, : \, -b_s \alpha + \gamma_t (\alpha) + c_s > 0 \right\} \\
        \mathcal{S} \setminus \mathcal{I}_t = \mathcal{I}^- (\alpha)
        \quad \text{with} \quad
        &\mathcal{I}^- (\alpha) = \left\{ s \in \mathcal{S} \, : \, -b_s \alpha + \gamma_t (\alpha) + c_s < 0 \right\}
    \end{align*}
    for all $\alpha \in (\alpha_t, \alpha_{t-1})$. In other words, the index set $\mathcal{I}_t$ contains the components of $-\bm{b} \alpha + \gamma_t (\alpha) \mathbf{e} + \bm{c}$ that are strictly positive on the open interval $\alpha \in (\alpha_t, \alpha_{t-1})$, and all the components in the complement $\mathcal{S} \setminus \mathcal{I}_t$ are strictly negative over $\alpha \in (\alpha_t, \alpha_{t-1})$. From this we then conclude that $\gamma_t (\alpha) = \gamma^\star (\alpha)$ over $\alpha \in [\alpha_t, \alpha_{t-1}]$ for all $t$. Since $\alpha_0 = \infty$ and $\alpha_T = -\infty$ by construction, the latter implies that the output of the algorithm is correct upon termination. We finally show that Algorithm~\ref{alg:2norm:gamma_form} terminates in at most $2S$ iterations. Since the per-iteration computations are bounded by $\mathcal{O} (S)$ and the initial sorting takes time $\mathcal{O} (S \log S)$, \mbox{the complexity statement then follows as well.}
    
    In view of the first step, we first observe that
    \begin{equation*}
        -b_1 \alpha + \gamma_1 (\alpha) + c_1
        \;\; = \;\;
        -b_1 \alpha + \frac{\rho + a_1 (b_1 \alpha - c_1)}{a_1} + c_1
        \;\; = \;\;
        \frac{\rho}{a_1}
        \;\; > \;\;
        0,
    \end{equation*}
    that is, $\mathcal{I}_1 \subseteq \mathcal{I}^+ (\alpha)$, while at the same time
    \begin{equation*}
        -b_s \alpha + \gamma_1 (\alpha) + c_s
        \;\; = \;\;
        -b_s \alpha + \frac{\rho + a_1 (b_1 \alpha - c_1)}{a_1} + c_s
        \;\; = \;\;
        (b_1 - b_s) \alpha + \frac{\rho}{a_1} + (c_s - c_1)
        \;\; < \;\;
        0
    \end{equation*}
    for all $s \in \mathcal{S} \setminus \mathcal{I}_1$, that is, $\mathcal{S} \setminus \mathcal{I}_1 \subseteq \mathcal{I}^- (\alpha)$, whenever $\alpha > \alpha_1$. Note that Assumption~\ref{assm:2norm:simple_algo}~\emph{(i)} ensures that $b_s \neq b_1$ for all $s \neq 1$. We thus conclude that $\mathcal{I}_1 = \mathcal{I}^+ (\alpha)$ and $\mathcal{S} \setminus \mathcal{I}_1 = \mathcal{I}^- (\alpha)$ for $\alpha > \alpha_1$.
    
    Assume now that $\mathcal{I}_t = \mathcal{I}^+ (\alpha)$ and $\mathcal{S} \setminus \mathcal{I}_t = \mathcal{I}^- (\alpha)$ for all $\alpha \in (\alpha_t, \alpha_{t-1})$, and consider the construction of the index set $\mathcal{I}_{t + 1}$ in iteration $t$ for the case when $\alpha_t > -\infty$. By Assumption~\ref{assm:2norm:simple_algo}~\emph{(ii)}, there is no $s \in \mathcal{S}$ such that $-b_s \alpha + \gamma_t (\alpha) + c_s = 0$ on an interval of positive width (\emph{cf.}~Lemma~\ref{lem:consequences_of_assm_2norm} in the appendix), which implies that the breakpoint $\alpha_{t+1}$ is strictly smaller than $\alpha_t$. Moreover, Assumption~\ref{assm:2norm:simple_algo}~\emph{(iii)} ensures that there are no $s, s' \in \mathcal{S}$, $s \neq s'$, such that $-b_s \alpha + \gamma_t (\alpha) + c_s = 0$ and $-b_{s'} \alpha + \gamma_t (\alpha) + c_{s'} = 0$ simultaneously at any $\alpha \in \mathbb{R}$ (\emph{cf.}~Lemma~\ref{lem:consequences_of_assm_2norm}), implying that the algorithm does not skip any components $s \in \mathcal{S}$ whose expressions $-b_s \alpha + \gamma_t (\alpha) + c_s$ vanish. To show that $\mathcal{I}_{t+1} = \mathcal{I}^+ (\alpha)$ and $\mathcal{S} \setminus \mathcal{I}_{t+1} = \mathcal{I}^- (\alpha)$ for all $\alpha \in (\alpha_{t+1}, \alpha_t)$, we proceed in two steps. We first show that $\gamma_{t+1} (\alpha_t) = \gamma_t (\alpha_t)$. The construction of $\alpha_{t+1}$ then implies that none of the expressions $-b_s \alpha + \gamma_t (\alpha) + c_s$, $s \in \mathcal{S} \setminus \{ s_t \}$ vanish over the interval $\alpha \in (\alpha_{t+1}, \alpha_t)$, that is, $\mathcal{I}_{t+1} \setminus \{ s_t \} \subseteq \mathcal{I}^+ (\alpha)$ and $\mathcal{S} \setminus [\mathcal{I}_{t+1} \cup \{ s_t \}] \subseteq \mathcal{I}^- (\alpha)$ over $\alpha \in (\alpha_{t+1}, \alpha_t)$. We then show that $-b_s \alpha + \gamma_t (\alpha) + c_s > 0$ across $\alpha \in (\alpha_{t+1}, \alpha_t)$ if $s_t \in \mathcal{I}_{t+1}$ and $-b_s \alpha + \gamma_t (\alpha) + c_s < 0$ over $\alpha \in (\alpha_{t+1}, \alpha_t)$ if $s_t \notin \mathcal{I}_{t+1}$. Together with our earlier finding, this implies that $\mathcal{I}_{t+1} = \mathcal{I}^+ (\alpha)$ and $\mathcal{S} \setminus \mathcal{I}_{t+1} = \mathcal{I}^- (\alpha)$ over $\alpha \in (\alpha_{t+1}, \alpha_t)$.
    
    To see that $\gamma_{t+1} (\alpha_t) = \gamma_t (\alpha_t)$, we note that
    \begin{equation*}
        \sum_{s \in \mathcal{I}_{t+1}} a_s \left[ -b_s \alpha_t + \gamma_{t+1} (\alpha_t) + c_s \right]
        \;\; = \;\;
        \rho
        \;\; = \;\;
        \sum_{s \in \mathcal{I}_t} a_s \left[ -b_s \alpha_t + \gamma_t (\alpha_t) + c_s \right]
    \end{equation*}
    by construction of $\gamma_t (\alpha)$ and $\gamma_{t+1} (\alpha)$. If $\mathcal{I}_{t+1} = \mathcal{I}_t \cup \{s_t\}$, then the above equation implies that
    \begin{equation*}
        \begin{array}{rl}
            0
            & \displaystyle = \;\; \mspace{10mu} \sum_{s \in \mathcal{I}_{t+1}} a_s \left[ -b_s \alpha_t + \gamma_{t+1} (\alpha_t) + c_s \right] -  \sum_{s \in \mathcal{I}_t} a_s \left[ -b_s \alpha_t + \gamma_t (\alpha_t) + c_s \right] \\
            & \displaystyle = \;\;
            \sum_{s \in \mathcal{I}_t \cup \{ s_t \}} a_s \left[ -b_s \alpha_t + \gamma_{t+1} (\alpha_t) + c_s \right] -  \sum_{s \in \mathcal{I}_t} a_s \left[ -b_s \alpha_t + \gamma_t (\alpha_t) + c_s \right] \\
            & \displaystyle = \;\;
            a_{s_t} \left[ -b_{s_t} \alpha_t + c_{s_t} \right] + \Bigg( \sum_{s \in \mathcal{I}_t \cup \{ s_t \}} a_s \Bigg) \gamma_{t+1} (\alpha_t) - \Bigg( \sum_{s \in \mathcal{I}_t} a_s \Bigg) \gamma_t (\alpha_t) + a_{s_t} \gamma_t (\alpha_t) - a_{s_t} \gamma_t (\alpha_t) \\
            & \displaystyle = \;\;
            a_{s_t} \left[ -b_{s_t} \alpha_t + \gamma_t (\alpha_t) + c_{s_t} \right] + \Bigg( \sum_{s \in \mathcal{I}_t \cup \{ s_t \}} a_s \Bigg) \left[ \gamma_{t+1} (\alpha_t) - \gamma_t (\alpha_t) \right].
        \end{array}
    \end{equation*}
    The above equality implies that $\gamma_{t+1} (\alpha_t) = \gamma_t (\alpha_t)$ since $\bm{a} > \mathbf{0}$ and $-b_{s_t} \alpha_t +  \gamma_t (\alpha_t) + c_{s_t} = 0$ by the definition of $\alpha_t$. A similar argument shows that $\gamma_{t+1} (\alpha_t) = \gamma_t (\alpha_t)$ also when $\mathcal{I}_{t+1} = \mathcal{I}_t \backslash \{ s_t \}$.
    
    Assume now that $\mathcal{I}_{t+1} = \mathcal{I}_t \cup \{ s_t \}$, in which case $-b_{s_t} \alpha + \gamma_t (\alpha) + c_{s_t} < 0$ over $\alpha \in (\alpha_t, \alpha_{t-1})$. This implies that the slope $m_t$ of $\gamma_t(\alpha)$ is strictly less than $b_{s_t}$, and we thus have
    \begin{equation*}
        \mspace{-5mu}
        \frac{\sum_{s \in \mathcal{I}_t} a_s b_s}{\sum_{s \in \mathcal{I}_t} a_s} < b_{s_t}
        \; \Longleftrightarrow \;
        \sum_{s \in \mathcal{I}_t} a_s \left[ b_s - b_{s_t} \right] < 0 
        \; \Longleftrightarrow \;
        \sum_{s \in \mathcal{I}_t \cup \{ s_t \}} a_s \left[ b_s - b_{s_t} \right] < 0
        \; \Longleftrightarrow \;
        \frac{\sum_{s \in \mathcal{I}_{t+1}} a_s b_s}{\sum_{s \in \mathcal{I}_{t+1}} a_s} < b_{s_t},
    \end{equation*}
    that is, the slope of $\gamma_{t+1}(\alpha)$ is also strictly less than $b_{s_t}$. Since $-b_{s_t}\alpha_t + \gamma_{t+1}(\alpha_t) + c_{s_t} = 0$, we must have $-b_{s_t}\alpha + \gamma_{t+1}(\alpha) + c_{s_t} > 0$ for $\alpha \in (\alpha_{t+1}, \alpha_t)$, that is, the inclusion of $s_t$ into $\mathcal{I}_{t+1}$ is correct. A similar argument shows that $-b_s \alpha + \gamma_{t+1} (\alpha) + c_s < 0$ over $\alpha \in (\alpha_{t+1}, \alpha_t)$ if $\mathcal{I}_{t+1} = \mathcal{I}_t \setminus \{ s_t \}$.
    
    We now observe that for all $t = 1, \ldots$ and all $\alpha \in [\alpha_t, \alpha_{t-1}]$, we have
    \begin{equation*}
        \bm{a}^\top \left[ -\bm{b} \alpha + \gamma_t (\alpha) \mathbf{e} + \bm{c} \right]_+
        \;\; = \;\;
        \sum_{s \in \mathcal{I}_t} \left[ - a_s b_s \alpha + a_s \gamma_t (\alpha) + a_s c_s \right]
        \;\; = \;\;
        \rho,
    \end{equation*}
    where the first identity follows from the properties of $\mathcal{I}_t$ and $\mathcal{S} \setminus \mathcal{I}_t$, and the second identity is due to the definition of $\gamma_t$. We thus conclude that $\gamma_t (\alpha) = \gamma^\star (\alpha)$ over $[\alpha_t, \alpha_{t-1}]$ as desired.

    We finally consider the runtime of the algorithm. To this end, recall that $\gamma^\star$ is a concave function by Proposition~\ref{prop:gamma_star_alpha}, and Assumption~\ref{assm:2norm:simple_algo}~\emph{(ii)} ensures that there is no $s \in \mathcal{S}$ such that $-b_s \alpha + \gamma^\star (\alpha)+ c_s = 0$ on an interval of positive width  (\emph{cf.}~Lemma~\ref{lem:consequences_of_assm_2norm}). Thus, for each component $s \in \mathcal{S}$ we have $-b_s \alpha + \gamma^\star (\alpha) + c_s = 0$ at no more than two of the breakpoints $\alpha_1, \alpha_2, \ldots$. The result now follows from the fact that the number of iterations is bounded by the number of times that any of the component functions vanishes.
\end{proof}

\subsection{Efficient Solution for Equation System~\eqref{eq:2norm:nonlinear_eqs}}\label{sec:2norm:joint_algorithm}

A direct approach to solving the system~\eqref{eq:2norm:nonlinear_eqs} is to compute the solution paths
$\{(\alpha^1_\tau, m_\tau^1, v_\tau^1)\}_{\tau=1}^{T_1}$ and
$\{(\alpha^2_\tau, m_\tau^2, v_\tau^2)\}_{\tau=1}^{T_2}$ for the two equations separately using Algorithm~\ref{alg:2norm:gamma_form}, and then determine the intersection of the resulting piecewise affine curves $\gamma^1(\alpha)$ and $\gamma^2(\alpha)$. While conceptually simple, this strategy is unnecessary because the first equation in~\eqref{eq:2norm:nonlinear_eqs},
\begin{equation}\label{eq:l2_first_eq}
\mathbf{e}^\top \bm{\Sigma}_{sa}^{-2}\Big[ -\bm{b}\alpha + \gamma \mathbf{e} + 2\bm{\Sigma}_{sa}^2\overline{\bm{p}}_{sa}\Big]_+ \;=\; 2,
\end{equation}
does not depend on the bisection parameter $\beta$ (whereas the second equation does). We therefore solve~\eqref{eq:l2_first_eq} once, obtain $\gamma^1(\alpha)$ on $\alpha\ge 0$, and then reuse this solution path to solve the system~\eqref{eq:2norm:nonlinear_eqs} for each value of $\beta$ encountered in the outer bisection (as in the proofs of Theorems~\ref{thm:overall_complexity:exact_subproblem} and \ref{thm:overall_complexity:inexact_subproblem}).

Given $\gamma^1(\alpha)$ satisfying~\eqref{eq:l2_first_eq}, the system~\eqref{eq:2norm:nonlinear_eqs} reduces to finding $\alpha^\star$ such that
\begin{equation}\label{eq:l2_second_eq_with_gamma1}
\bm{b}^\top \bm{\Sigma}_{sa}^{-2}\Big[ -\bm{b}\alpha^\star + \gamma^1(\alpha^\star)\mathbf{e} + 2\bm{\Sigma}_{sa}^2\overline{\bm{p}}_{sa}\Big]_+ \;=\; 2\beta,
\end{equation}
and then setting $\gamma^\star = \gamma^1(\alpha^\star)$. Note that $\alpha$ is the dual variable associated with the constraint $\bm{b}^\top\bm{p}_{sa}\le \beta$ in~\eqref{eq:gen_projection}, hence $\alpha^\star\ge 0$.

We first record two refinements of Algorithm~\ref{alg:2norm:gamma_form} 
when applied to~\eqref{eq:l2_first_eq}.

\begin{prop}\label{prop:L2_sysEq_reducedSet}
When applying Algorithm~\ref{alg:2norm:gamma_form} to~\eqref{eq:l2_first_eq}, the breakpoint update
\[
\alpha_t = \max \Big\{ \alpha \in (-\infty, \alpha_{t - 1}) \, : \, -b_s \alpha + \gamma_t (\alpha) + c_s = 0  \text{ for some $s \in \mathcal{S}$} \Big\}
\]
can be replaced by
\[
\alpha_t = \max \Big\{ \alpha \in [0 , \alpha_{t - 1}) \, : \, -b_s \alpha + \gamma_t (\alpha) + c_s = 0  \text{ for some $s \in \mathcal{S}_t$} \Big\},
\qquad
\mathcal{S}_t := \{ s\in\mathcal{S} : m_t < b_s \}.
\]
In particular, since $\gamma(\alpha)$ is a piecewise affine non-decreasing concave function (Proposition~\ref{prop:gamma_star_alpha}) and $m_t$ is non-decreasing in $t$, the set sizes $|\mathcal{S}_t|$ are non-increasing in $t$.
\end{prop}

\begin{proof}
Since $\alpha^\star\ge 0$, it suffices to search for breakpoints over $\alpha\in[0,\alpha_{t-1})$.
To justify restricting $\mathcal{S}$ to $\mathcal{S}_t$, recall that for~\eqref{eq:l2_first_eq} we have $\rho=2$ and $c_{s'} = 2\sigma_{sas'}^2 \overline{p}_{sas'} \ge 0$, and Algorithm~\ref{alg:2norm:gamma_form} yields
\[
v_t \;=\; \frac{\rho - \sum_{s \in \mathcal{I}_t} a_s c_s}{\sum_{s \in \mathcal{I}_t} a_s}
\;=\; \frac{2 - \sum_{s' \in \mathcal{I}_t} 2 \overline{p}_{sas'}}{\sum_{s \in \mathcal{I}_t} a_s}
\;=\; \frac{2\bigl(1 - \sum_{s' \in \mathcal{I}_t} \overline{p}_{sas'}\bigr)}{\sum_{s \in \mathcal{I}_t} a_s}
\;\ge\; 0,
\]
where the second equality follows by matching~\eqref{eq:2norm:nonlinear_genericform} to~\eqref{eq:l2_first_eq}.
Now
\[
-b_s \alpha + \gamma_t(\alpha) + c_s = (m_t-b_s)\alpha + v_t + c_s = 0
\quad\Longleftrightarrow\quad
\alpha = -\frac{v_t+c_s}{m_t-b_s}.
\]
Since $v_t+c_s\ge 0$, this implies $\alpha>0$ only if $m_t<b_s$. Thus indices with $m_t\ge b_s$ can only generate breakpoints at $\alpha\le 0$ and can be ignored on the domain $\alpha\ge 0$.
\end{proof}

\begin{prop}\label{prop:L2_sysEq_notrepeat}
Consider the setting of Proposition~\ref{prop:L2_sysEq_reducedSet}. The search set $\mathcal{S}_t$ can be further replaced by $\mathcal{S}_t \setminus \mathcal{I}_t$. In particular, in Algorithm~\ref{alg:2norm:gamma_form}, the update
\begin{center}\rm
\textbf{if} $s_t \in \mathcal{I}_t$ \textbf{then} set $\mathcal{I}_{t+1}=\mathcal{I}_t\setminus\{s_t\}$ \textbf{else} set $\mathcal{I}_{t+1}=\mathcal{I}_t\cup\{s_t\}$ \textbf{end}
\end{center}
can be simplified to
\begin{center}\rm
set $\mathcal{I}_{t+1}=\mathcal{I}_t\cup\{s_t\}$.
\end{center}
\end{prop}

\begin{proof}
Let $s'\in\mathcal{S}_t$. If $s'\in\mathcal{I}_t$, then by the proof of Theorem~\ref{thm:2norm_simple_alg_works} we have
$-b_{s'}\alpha+\gamma_t(\alpha)+c_{s'}>0$ for all $\alpha\in(\alpha_t,\alpha_{t-1})$.
Since $s'\in\mathcal{S}_t$ implies $m_t<b_{s'}$, the affine function
$(m_t-b_{s'})\alpha+v_t+c_{s'}$ is strictly decreasing in $\alpha$, hence it remains strictly positive on the entire half-line $(-\infty,\alpha_{t-1})$ once it is positive on $(\alpha_t,\alpha_{t-1})$.
Therefore such an index cannot attain the maximizing breakpoint in the definition of $\alpha_t$ on $[0,\alpha_{t-1})$, and may be removed from the search set without affecting $\alpha_t$.
\end{proof}

We next show how to solve the full system~\eqref{eq:2norm:nonlinear_eqs} 
using a single scan over the segments of $\gamma^1(\alpha)$. Let $\{(\alpha_t,m_t,v_t)\}_{t=1}^T$ describe the solution path $\gamma^1(\alpha)$ of~\eqref{eq:l2_first_eq} produced by Algorithm~\ref{alg:2norm:gamma_form}, so that
\[
\gamma^1(\alpha)=m_t\alpha+v_t \qquad \text{for all }\alpha\in[\alpha_t,\alpha_{t-1}],\quad t=1,\dots,T,
\]
with $\alpha_0=\infty$ and $\alpha_T=-\infty$ (Theorem~\ref{thm:2norm_simple_alg_works}).
Define, for $\alpha\ge 0$,
\[
F(\alpha)\;:=\;\bm{b}^\top \bm{\Sigma}_{sa}^{-2}\Big[ -\bm{b}\alpha + \gamma^1(\alpha)\mathbf{e} + 2\bm{\Sigma}_{sa}^2\overline{\bm{p}}_{sa}\Big]_+.
\]
Because both equations in~\eqref{eq:2norm:nonlinear_eqs} involve the same vector inside $[\cdot]_+$, the active set of indices does not change within any interval $(\alpha_t,\alpha_{t-1})$. Concretely, on such an interval we have
\[
-b_{s'}\alpha + \gamma^1(\alpha) + 2\sigma_{sas'}^2\overline{p}_{sas'}
=
(m_t-b_{s'})\alpha + \bigl(v_t+2\sigma_{sas'}^2\overline{p}_{sas'}\bigr),
\]
and therefore the index set
\begin{equation}\label{eq:active_set_hat}
\hat{\mathcal S}_t
:=
\Bigl\{s'\in\mathcal S:\;
(m_t-b_{s'})\alpha + (v_t+2\sigma_{sas'}^2\overline{p}_{sas'})>0
\ \text{for (equivalently) any }\alpha\in(\alpha_t,\alpha_{t-1})
\Bigr\}
\end{equation}
is well-defined and constant over $(\alpha_t,\alpha_{t-1})$.
Consequently, $F(\alpha)$ restricts to a single affine function on $[\alpha_t,\alpha_{t-1}]$ and the root condition $F(\alpha)=2\beta$ can be checked by scanning the segments sequentially. This procedure is given in Algorithm~\ref{alg:2norm:findopt}.

\begin{thm}\label{thm:l2_joint_correctness_complexity}
Let $\{(\alpha_t,m_t,v_t)\}_{t=1}^T$ describe $\gamma^1(\alpha)$ as above, and let $\beta$ be fixed.
Algorithm~\ref{alg:2norm:findopt} computes a value $\alpha^\star\ge 0$ such that $(\alpha^\star,\gamma^\star)$ with $\gamma^\star=\gamma^1(\alpha^\star)$ satisfies the system~\eqref{eq:2norm:nonlinear_eqs}.
Moreover, the algorithm runs in time $\mathcal{O}(ST)$; since $T\le 2S$ by Theorem~\ref{thm:2norm_simple_alg_works}, the overall time is $\mathcal{O}(S^2)$.
\end{thm}

\begin{proof}
Fix a segment $t$ and consider $\alpha\in(\alpha_t,\alpha_{t-1})$.
By definition of $\hat{\mathcal S}_t$ in~\eqref{eq:active_set_hat}, the sign pattern of the components
\(
(m_t-b_{s'})\alpha+(v_t+2\sigma_{sas'}^2\overline{p}_{sas'})
\)
does not change on $(\alpha_t,\alpha_{t-1})$, hence the positive-part operator is equivalent to restricting the sum to $\hat{\mathcal S}_t$ throughout the interval. It follows that $F(\alpha)=A_t\alpha+B_t$ is affine on $[\alpha_t,\alpha_{t-1}]$ with coefficients as in Algorithm~\ref{alg:2norm:findopt}.
Therefore, if $2\beta$ lies between the endpoint values $F(\alpha_t)$ and $F(\alpha_{t-1})$, then either \emph{(i)} $A_t\neq 0$ and the unique solution
$\alpha^\star=(2\beta-B_t)/A_t$ lies in $[\alpha_t,\alpha_{t-1}]$, or \emph{(ii)} $A_t=0$ and $F$ is constant on this segment and any $\alpha$ in the segment satisfies $F(\alpha)=2\beta$.
Since $\alpha^\star\ge 0$ and the segments cover all $\alpha\ge 0$ up to $\alpha_0=\infty$, scanning the segments finds such a segment and returns a valid $\alpha^\star$. Setting $\gamma^\star=\gamma^1(\alpha^\star)$ satisfies~\eqref{eq:l2_first_eq} by construction and~\eqref{eq:l2_second_eq_with_gamma1} by the choice of $\alpha^\star$, hence solves~\eqref{eq:2norm:nonlinear_eqs}.
Finally, each segment requires $\mathcal{O}(S)$ work to construct $\hat{\mathcal S}_t$ (e.g., by testing the sign at an interior point) and to compute $(A_t,B_t)$, hence the total time is $\mathcal{O}(ST)$.
\end{proof}

\begin{algorithm}[tb]
	\KwData{List $\{(\alpha_t,m_t,v_t)\}_{t=1}^T$ defining $\gamma^1(\alpha)$, and parameters $\bm{b}$, $\bm{\Sigma}_{sa}$, $\overline{\bm{p}}_{sa}$, and $\beta$.}
	\For{$t=1,\dots,T$}{
		Choose an interior point $\tilde\alpha_t\in(\alpha_t,\alpha_{t-1})\cap\mathbb{R}_+$, e.g.,
		$\tilde\alpha_t=\frac{1}{2}(\alpha_t+\alpha_{t-1})$ when $\alpha_{t-1}<\infty$, and $\tilde\alpha_t=\max\{0,\alpha_t\}+1$ when $\alpha_{t-1}=\infty$.

		Compute the active set on this segment:
		\[
		\hat{\mathcal S}_t
		=
		\Bigl\{s'\in\mathcal S:\;
		(m_t-b_{s'})\tilde\alpha_t + \bigl(v_t+2\sigma_{sas'}^2\overline{p}_{sas'}\bigr) > 0
		\Bigr\}.
		\]

		Compute the affine coefficients
		\[
		A_t=\sum_{s'\in\hat{\mathcal S}_t}\frac{b_{s'}}{\sigma_{sas'}^2}(m_t-b_{s'}),
		\qquad
		B_t=\sum_{s'\in\hat{\mathcal S}_t}\frac{b_{s'}}{\sigma_{sas'}^2}\bigl(v_t+2\sigma_{sas'}^2\overline{p}_{sas'}\bigr).
		\]
		Define $F_t(\alpha)=A_t\alpha+B_t$ on $[\alpha_t,\alpha_{t-1}]\cap\mathbb{R}_+$.

		Evaluate endpoints:
		$y_{\rm low}=F_t(\max\{0,\alpha_t\})$ and $y_{\rm high}=F_t(\alpha_{t-1})$ (interpreting $F_t(\infty)$ as the limit if $\alpha_{t-1}=\infty$).

		\uIf{$\min\{y_{\rm low},y_{\rm high}\}\le 2 \beta \le \max\{y_{\rm low},y_{\rm high}\}$}{
			\uIf{$A_t=0$}{
				Set $\alpha^\star=\max\{0,\alpha_t\}$ and $\gamma^\star=m_t\alpha^\star+v_t$, and \textbf{return} $(\alpha^\star,\gamma^\star)$.
			}
			\Else{
				Set $\alpha^\star=(2 \beta-B_t)/A_t$ and $\gamma^\star=m_t\alpha^\star+v_t$, and \textbf{return} $(\alpha^\star,\gamma^\star)$.
			}
		}
	}
	\KwResult{A solution $(\alpha^\star,\gamma^\star)$ to~\eqref{eq:2norm:nonlinear_eqs}.}
	\caption{Efficient solution of equation system~\eqref{eq:2norm:nonlinear_eqs} using a precomputed path $\gamma^1(\alpha)$.\label{alg:2norm:findopt}}
\end{algorithm}

\section{$\phi$-Divergence Ambiguity Sets}\label{sec:phi_div}

We now assume that the deviation measure $d_a$ in the $s$-rectangular ambiguity set~\eqref{eq:additive_s_rect} satisfies $d_a (\bm{p}_{sa}, \overline{\bm{p}}_{sa}) = \sum_{s' \in \mathcal{S}} \overline{p}_{sas'} \cdot \phi (p_{sas'} / \overline{p}_{sas'})$ for a convex function $\phi : \mathbb{R}_+ \mapsto \mathbb{R}_+$ with $\phi (1) = 0$. To keep the notational overhead small, we assume throughout this section that $\overline{\bm{p}}_{sa} > \mathbf{0}$ component-wise and the generalized projection problem is feasible, that is, that $\min \{ \bm{b} \} \leq \beta$. Note that if $\overline{p}_{sas'} = 0$ for some $s, s' \in \mathcal{S}$ and $a \in \mathcal{A}$, then $p_{sas'} = 0$ for all $\bm{p}_{sa} \in \Delta_S$ with $d_a (\bm{p}_{sa}, \overline{\bm{p}}_{sa}) < \infty$, and thus we can always ensure that $\overline{\bm{p}}_{sa} > \mathbf{0}$ by a removing redundant indices. $\phi$-divergence ambiguity sets have received significant attention in the literature on distributionally robust and data-driven optimization, and they can be readily calibrated to historical observations of the MDP's transitions \citep{BTdHWMR13:phi_div, BL15:phi_div}.

\begin{prop}\label{prop:phi_div_dual}
	 For the deviation measure $d_a (\bm{p}_{sa}, \overline{\bm{p}}_{sa}) = \sum_{s' \in \mathcal{S}} \overline{p}_{sas'} \cdot \phi \left( \frac{p_{sas'}}{\overline{p}_{sas'}} \right)$, the optimal value of the projection problem~\eqref{eq:gen_projection} equals the optimal value of the bivariate convex optimization problem
	\begin{equation}\label{eq:phi_div:projection_problem}
		\begin{array}{l@{\quad}l}
			\text{\emph{maximize}} & \displaystyle - \beta \alpha + \zeta - \sum_{s' \in \mathcal{S}} \overline{p}_{sas'} \phi^{\star} (-\alpha b_{s'} + \zeta ) \\
			\text{\emph{subject to}} & \displaystyle \alpha \in \mathbb{R}_+, \;\; \zeta \in \mathbb{R},
		\end{array}
	\end{equation}
	where $\phi^\star (y) = \sup \, \{ yt - \phi (t) \, : \, t \in \mathbb{R}_+ \}$ is the convex conjugate of $\phi$.
\end{prop}

\begin{proof}
	For the deviation measure from the statement of this proposition, problem~\eqref{eq:gen_projection} becomes
	\begin{equation}\label{eq:phi_div:projection_problem_org}
	\begin{array}{l@{\quad}l}
	\text{minimize} & \displaystyle \sum_{s' \in \mathcal{S}} \overline{p}_{sas'} \cdot \phi \left( \frac{p_{sas'}}{\overline{p}_{sas'}} \right) \\
	\text{subject to} & \displaystyle \bm{b}^\top \bm{p}_{sa} \leq \beta \\
	& \displaystyle \bm{p}_{sa} \in \Delta_S.
	\end{array}
	\end{equation}
	The Lagrange dual function associated with this problem is
	\begin{equation*}
	g (\alpha, \zeta) = \inf \left\{ \sum_{s' \in \mathcal{S}} \overline{p}_{sas'} \cdot \phi \left( \frac{p_{sas'}}{\overline{p}_{sas'}} \right) + \alpha (\bm{b}^\top \bm{p}_{sa} - \beta) + \zeta (1 - \mathbf{e}^\top \bm{p}_{sa}) \, : \, \bm{p}_{sa} \in \mathbb{R}^S_+ \right\},
	\end{equation*}
	where $\alpha \in \mathbb{R}_+$ and $\zeta \in \mathbb{R}$. Rearranging terms, we observe that
	\begin{equation*}
	g (\alpha, \zeta) = - \beta \alpha + \zeta - \sum_{s' \in \mathcal{S}} \overline{p}_{sas'} \cdot \sup \left\{ \frac{p_{sas'}}{\overline{p}_{sas'}} \cdot (-\alpha b_{s'} + \zeta) -   \phi \left( \frac{p_{sas'}}{\overline{p}_{sas'}} \right) \, : \, p_{sas'} \in \mathbb{R}_+ \right\},
	\end{equation*}
	and the suprema inside this expression coincide with the convex conjugates $\phi^\star (-\alpha b_{s'} + \zeta)$, $s' \in \mathcal{S}$. The resulting optimization problem~\eqref{eq:phi_div:projection_problem} is convex since the conjugates are convex. Moreover, since $\min \{ \bm{b} \} \leq \beta$ by assumption, problem~\eqref{eq:phi_div:projection_problem_org} affords a feasible solution, and the linearity of the constraints implies that strong duality holds between~\eqref{eq:phi_div:projection_problem} and~\eqref{eq:phi_div:projection_problem_org}, that is, their optimal objective values indeed coincide.
\end{proof}

In the remainder of the section, we study two popular $\phi$-divergences and show that for both of them, problem~\eqref{eq:phi_div:projection_problem} can be further simplified to a univariate convex optimization problem that can be solved efficiently via bisection.

\subsection{Kullback-Leibler Divergence}

We first show that for the Kullback-Leibler divergence $\phi (t) = t \log t - t + 1$, problem~\eqref{eq:phi_div:projection_problem} can be reduced to a univariate convex optimization problem.

\begin{prop}\label{prop:kl_univariate_problem}
	For the Kullback-Leibler divergence $\phi (t) = t \log t - t + 1$, the optimal value of the projection problem~\eqref{eq:gen_projection} equals the optimal value of the univariate convex optimization problem
	\begin{equation}\label{eq:kl_divergence_projection_problem}
		\begin{array}{l@{\quad}l}
			\text{\emph{maximize}} & \displaystyle - \beta \alpha - \log \left( \sum_{s' \in \mathcal{S}} \overline{p}_{sas'} \cdot \mathrm{e}^{-\alpha b_{s'} } \right)  \\
			\text{\emph{subject to}} & \displaystyle \alpha \in \mathbb{R}_+.
		\end{array}
	\end{equation}
\end{prop}

\begin{proof}
	Plugging the convex conjugate $\phi^{\star} (y) = e^y - 1$ of the Kullback-Leibler divergence into the bivariate optimization problem~\eqref{eq:phi_div:projection_problem}, we obtain
	\begin{equation*}
	\begin{array}{l@{\quad}l}
	\text{maximize} & \displaystyle - \beta \alpha + \zeta - \sum_{s' \in \mathcal{S}} \overline{p}_{sas'} \left( \mathrm{e}^{-\alpha b_{s'} + \zeta} - 1 \right) \\
	\text{subject to} & \displaystyle \alpha \in \mathbb{R}_+, \;\; \zeta \in \mathbb{R}.
	\end{array}
	\end{equation*}
	By rearranging terms, the objective function can be expressed as
	\begin{equation}\label{eq:kl_intermediate_obj_fct}
	1 - \beta \alpha + \zeta - e^{\zeta} \left( \sum_{s' \in \mathcal{S}} \overline{p}_{sas'} \cdot \mathrm{e}^{-\alpha b_{s'} } \right),
	\end{equation}
	and the first-order optimality condition shows that for fixed $\alpha \in \mathbb{R}_+$, the function is maximized by
	\begin{equation*}
	1- e^{\zeta^\star} \left( \sum_{s' \in \mathcal{S}} \overline{p}_{sas'} \cdot \mathrm{e}^{-\alpha b_{s'} } \right)  = 0 \;\; \Longleftrightarrow \;\; \zeta^\star = - \log \left( \sum_{s' \in \mathcal{S}} \overline{p}_{sas'} \cdot \mathrm{e}^{-\alpha b_{s'} } \right).
	\end{equation*}
	Substituting $\zeta^\star$ in~\eqref{eq:kl_intermediate_obj_fct}, we obtain problem~\eqref{eq:kl_divergence_projection_problem} as postulated.
\end{proof}

\begin{thm}\label{thm:kl_complexity}
If $\beta \geq \min \{ \bm{b} \} + \omega$ for some $\omega > 0$, then the projection problem~\eqref{eq:gen_projection} can be solved to $\delta$-accuracy in time $\mathcal{O} (S \cdot \log [\max \{ \bm{b} \} \cdot \log (\min \{ \overline{\bm{p}} \}^{-1})/ (\delta \omega)])$.
\end{thm}

Note that the projection problem~\eqref{eq:gen_projection} is infeasible if $\beta < \min \{ \bm{b} \}$. The condition in the statement of Theorem~\ref{thm:kl_complexity} can thus be interpreted as a strict feasibility requirement.

\begin{proof}[Proof of Theorem~\ref{thm:kl_complexity}]
We prove the statement in three steps. Step~1 shows that the optimal solution $\alpha^\star$ to~\eqref{eq:kl_divergence_projection_problem} is lower and upper bounded by $\underline{\alpha}^0 = 0$ and $\overline{\alpha}^0 = \log \left( \frac{1}{\min \{ \overline{\bm{p}} \}} \right) \cdot \frac{1}{\beta - \min \{ \bm{b} \}}$, respectively. Note that $\overline{\alpha}^0$ is finite due to the assumed strict positivity of $\min \{ \overline{\bm{p}} \}$ and $\beta - \min \{ \bm{b} \}$. Step~2 derives a global upper bound on the derivative of $f (\alpha)$, which we henceforth use to denote  of the objective function of problem~\eqref{eq:kl_divergence_projection_problem}. In conjunction with the concavity of $f$, this will allow us to bound the maximum objective function value over any interval $[ \underline{\alpha}, \overline{\alpha} ] \subseteq \mathbb{R}_+$. Step~3, finally, employs a bisection search to solve~\eqref{eq:kl_divergence_projection_problem} to $\delta$-accuracy in the stated complexity.

As for the first step, the validity of the lower bound $\underline{\alpha}^0$ follows directly from the non-negativity constraint in~\eqref{eq:kl_divergence_projection_problem}. In view of the upper bound $\overline{\alpha}^0$, we note that
\begin{align*}
    \overline{\alpha}^0 \; = \; \log \left(\frac{1}{\min\{ \overline{\bm{p}} \}} \right) \cdot \frac{1}{\beta - \min\{\bm{b}\}}
    \quad &\Longleftrightarrow \quad
    \min\{\overline{\bm{p}}\} \cdot \mathrm{e}^{\overline{\alpha}^0 (\beta - \min \{ \bm{b} \})} \; = \; 1 \\
    &\Longrightarrow \quad
    \sum_{s' \in \mathcal{S}} \overline{p}_{sas'} \cdot \mathrm{e}^{\overline{\alpha}^0 (\beta -b_{s'})} \; \geq \; 1 \\
    &\Longleftrightarrow \quad
    \sum_{s' \in \mathcal{S}} \overline{p}_{sas'} \cdot \mathrm{e}^{-\overline{\alpha}^0 b_{s'}} \; \geq \; \mathrm{e}^{- \beta \overline{\alpha}^0} \\
    &\Longleftrightarrow \quad
    \log \left( \sum_{s' \in \mathcal{S}} \overline{p}_{sas'} \cdot \mathrm{e}^{-\overline{\alpha}^0 b_{s'} } \right) \; \geq \; - \beta \overline{\alpha}^0 \\
    &\Longleftrightarrow \quad
    f (\overline{\alpha}^0) \; \leq \; 0.
\end{align*}
Since $f (0) = 0$ and $f (\overline{\alpha}^0) \leq 0$ while at the same time $\overline{\alpha}^0 > 0$, we conclude from the concavity of $f$ that $\overline{\alpha}^0$ is indeed a valid upper bound on the maximizer of problem~\eqref{eq:kl_divergence_projection_problem}.

In view of the second step, we observe that
\begin{equation*}
    f' (\alpha)
    \;\; \leq \;\;
    f' (0)
    \;\; = \;\;
    \frac{\sum_{s' \in \mathcal{S}} \overline{p}_{sas'} \cdot b_{s'}}{\sum_{s' \in \mathcal{S}} \overline{p}_{sas'}} - \beta
    \;\; \leq \;\;
    \overline{\bm{p}}_{sa}^\top \bm{b}
    \;\; \leq \;\;
    \max\{\bm{b}\},
    \qquad \forall \alpha \in \mathbb{R}_+,
\end{equation*}
where the first inequality follows from the concavity of $f$ and the other two inequalities hold since $\overline{\bm{p}}_{sa} \in \Delta_S$. The concavity of $f (\alpha)$ then implies that for any $\alpha \in [ \underline{\alpha}, \overline{\alpha} ] \subset \mathbb{R}_+$, we have
\begin{equation*}
    f (\underline{\alpha})
    \;\; \leq \;\;
    f (\alpha)
    \;\; \leq \;\;
    f (\underline{\alpha}) + f' (\underline{\alpha}) \cdot (\overline{\alpha} - \underline{\alpha})
    \;\; \leq \;\;
    f (\underline{\alpha}) + \max\{\bm{b}\} \cdot (\overline{\alpha} - \underline{\alpha}).
\end{equation*}
Thus, if we find $\underline{\alpha}$, $\overline{\alpha}$ sufficiently close such that $\alpha^\star \in [\underline{\alpha}, \overline{\alpha}]$, then we can closely bound the optimal objective value of~\eqref{eq:kl_divergence_projection_problem} from below and above by $f (\underline{\alpha})$ and $f (\underline{\alpha}) + \max\{\bm{b}\} \cdot (\overline{\alpha} - \underline{\alpha})$, respectively.

As for the third step, finally, we bisect on $\alpha$ by starting with the initial bounds $(\underline{\alpha}^0, \overline{\alpha}^0)$, halving the length of the interval $[\underline{\alpha}^i, \overline{\alpha}^i]$ in each iteration $i = 0, 1, \ldots$ by verifying whether $f' ([\underline{\alpha}^i + \overline{\alpha}^i] / 2)$ is positive and terminating once $\overline{\alpha}^i - \underline{\alpha}^i \leq \delta / \max\{\bm{b}\}$. Since $\beta - \min\{\bm{b}\} \geq \omega$, we have
\begin{equation*}
    \overline{\alpha}^0 - \underline{\alpha}^0
    \;\; = \;\;
    \log \left( \frac{1}{\min \{ \overline{\bm{p}} \} } \right) \cdot \frac{1}{\beta - \min \{ \bm{b} \}}
    \;\; \leq \;\;
    \frac{1}{\omega} \cdot \log \left( \frac{1}{\min \{ \overline{\bm{p}} \}} \right),
\end{equation*}
and thus the length of the interval no longer exceeds $\delta / \max\{\bm{b}\}$ once the iteration number $i$ satisfies
\begin{align*}
    2^{-i} \cdot (\overline{\alpha}^0 - \underline{\alpha}^0)
    \;\; \leq \;\;
    \frac{\delta }{\max\{\bm{b}\}}
    \quad \Longleftarrow& \quad
    2^{-i} \cdot \frac{1}{\omega} \cdot \log\left(\frac{1}{\min\{\overline{\bm{p}}\}}\right)
    \;\; \leq \;\;
    \frac{\delta }{\max\{\bm{b}\}} \\
    \quad \Longleftrightarrow& \quad
    i
    \;\; \geq \;\;
    \log_2 \left( \frac{\max\{\bm{b}\} \cdot \log (\min\{\overline{\bm{p}}\}^{-1})}{\delta \omega} \right),
\end{align*}
that is, after $\mathcal{O} (\log [\max \{ \bm{b} \} \cdot \log (\min\{\overline{\bm{p}}\}^{-1}) / (\delta \omega)])$ iterations. The interval $[f (\underline{\alpha}^i), f (\overline{\alpha}^i)]$ then provides the $\delta$-accurate solution to the projection problem~\eqref{eq:gen_projection}. The statement now follows from the fact that evaluating the derivative $f' ([\underline{\alpha}^i + \overline{\alpha}^i] / 2)$ in each bisection step takes $\mathcal{O} (S)$ time.
\end{proof}
The proof of Theorem~\ref{thm:overall_complexity:inexact_subproblem} employs an outer bisection over $\theta$ that requires for each $(s, a) \in \mathcal{S} \times \mathcal{A}$ the repeated solution of the projection problem~\eqref{eq:kl_divergence_projection_problem} with $\bm{b} = \bm{r}_{sa} + \lambda \bm{v}$ and $\beta = \theta \in [ \underline{R}_s (\bm{v}) + \frac{\epsilon}{2}, \overline{R} - \frac{\epsilon}{2}]$ (since the outer bisection is stopped when the interval length no longer exceeds $\epsilon$) to an accuracy of $\delta = \epsilon \kappa / [2 A \overline{R} + A \epsilon]$. In that case, for each $(s, a) \in \mathcal{S} \times \mathcal{A}$ we have
\begin{align*}
    \beta - \min \{ \bm{b} \}
    \;\; &\geq \;\;
    \underline{R}_s (\bm{v}) + \frac{\epsilon}{2} - \min \{ \bm{r}_{sa} + \lambda \bm{v} \} \\
    &\geq \;\;
    \max_{a \in \mathcal{A}} \min_{s' \in \mathcal{S}} \{ r_{sas'} + \lambda v_{s'} \} + \frac{\epsilon}{2} - \max_{a \in \mathcal{A}} \min_{s' \in \mathcal{S}} \{ r_{sas'} + \lambda v_{s'} \}
    \;\; = \;\;
    \frac{\epsilon}{2}
\end{align*}
and $\max \{ \bm{b} \} \leq \overline{R}$. Plugging those estimates into the statement of Theorem~\ref{thm:kl_complexity}, we see that the projection problem~\eqref{eq:kl_divergence_projection_problem} is solved in time $h (\epsilon \kappa / [2 A \overline{R} + A \epsilon]) = \mathcal{O} (S \cdot \log [A \overline{R}^2 \cdot \log (\min \{ \overline{\bm{p}} \}^{-1})/ (\epsilon^2 \kappa)])$.

\begin{rem}
The bounds $\underline{\alpha}^0 = 0$ and $\overline{\alpha}^0
= \log(\min\{\overline{\bm{p}}\}^{-1}) / (\beta - \min\{\bm{b}\})$ on $\alpha^\star$ established in the proof of Theorem~\ref{thm:kl_complexity} are tight up to constant factors. To see this, consider an instance with $S = 2$, $A = 1$, $\beta = 1.5$, $\bm b = (1, 2)^\top$, and $\overline{\bm p} = (P, 1-P)^\top$ for some $P \in (0, 0.5)$. This instance satisfies the strict feasibility assumption since $\min\{\bm b\} = 1 < \beta$. The dual objective \eqref{eq:kl_divergence_projection_problem} simplifies to
\[
    g(\alpha) = -1.5 \alpha - \log\left( P \mathrm{e}^{-\alpha} + (1-P) \mathrm{e}^{-2\alpha} \right).
\]
Since $g$ is strictly concave, its unique maximizer $\alpha^\star \in \mathbb{R}_+$ is characterized by the stationarity condition $g'(\alpha^\star) = 0$, which yields
\[
    -1.5 - \frac{ - P \mathrm{e}^{-\alpha} - 2 (1-P) \mathrm{e}^{-2\alpha} }{ P \mathrm{e}^{-\alpha} + (1-P) \mathrm{e}^{-2\alpha} } = 0
    \quad \Longleftrightarrow \quad
    \alpha^\star = \log\left( \frac{1-P}{P} \right).
\]
As $P \to 0$, we have $\min \{ \overline{\bm p} \} = P$ and $\alpha^\star = \Theta( \log(1/P) )$. This matches the asymptotic growth of the upper bound derived in Theorem~\ref{thm:kl_complexity}, which scales with $\log(\min \{ \overline{\bm p} \}^{-1})$. Conversely, as $P \to 0.5$, we have $\alpha^\star \to 0$, which recovers the lower bound of $0$.
\end{rem}

\subsection{Burg Entropy}

We next specialize the general $\phi$-divergence formulation to the Burg
entropy, which corresponds to the choice
$\phi(t) = -\log t + t - 1$.

\begin{prop}
	For the Burg entropy $\phi (t) = - \log t + t - 1$, if $\beta > \min \{ \bm{b} \}$, then the optimal value of the projection problem~\eqref{eq:gen_projection} equals the optimal value of the univariate convex optimization problem
	\begin{equation}\label{eq:burg_entropy_projection_problem}
	\begin{array}{l@{\quad}l}
	\text{\emph{maximize}} & \displaystyle \sum_{s' \in \mathcal{S}} \overline{p}_{sas'} \cdot \log \left( 1 + \alpha \frac{b_{s'} - \beta}{\beta - \min \{ \bm{b} \}} \right) \\
	\text{\emph{subject to}} & \displaystyle \alpha \leq 1 \\
	& \displaystyle \alpha \in \mathbb{R}_+.
	\end{array}
	\end{equation}
\end{prop}

\begin{proof}
	Plugging the convex conjugate $\phi^{\star} (y) = - \log (1 - y)$ of the Burg entropy into the bivariate optimization problem~\eqref{eq:phi_div:projection_problem}, we obtain
	\begin{equation}\label{eq:burg_intermediate_problem}
	\begin{array}{l@{\quad}l}
	\text{maximize} & \displaystyle - \beta \alpha + \zeta + \sum_{s' \in \mathcal{S}} \overline{p}_{sas'} \cdot \log (1 + \alpha b_{s'} - \zeta) \\
	\text{subject to} & \displaystyle 1 + \alpha \min \{ \bm{b} \} \geq \zeta \\
	& \displaystyle \alpha \in \mathbb{R}_+, \;\; \zeta \in \mathbb{R}.
	\end{array}
	\end{equation}
	Here, the first constraint ensures that the logarithms in the objective function are well-defined (as usual, we assume that $\log 0 = -\infty$). Unlike the proof of Proposition~\ref{prop:kl_univariate_problem}, the first-order optimality condition of this problem's objective function does not lend itself to extracting the optimal value of $\zeta$. Instead, we consider the Karush-Kuhn-Tucker conditions for problem~\eqref{eq:burg_intermediate_problem}, which are:
	\begin{equation*}
	\begin{array}{l@{\qquad}r}
	\displaystyle \sum_{s' \in \mathcal{S}} \overline{p}_{sas'} \cdot \frac{b_{s'}}{1 + \alpha b_{s'} - \zeta} = \beta - \eta \min \{ \bm{b} \} - \gamma & \text{(Stationarity)} \\
	\displaystyle \sum_{s' \in \mathcal{S}} \overline{p}_{sas'} \cdot \frac{1}{1 + \alpha b_{s'} - \zeta} = 1 - \eta & \text{(Stationarity)} \\
	1 + \alpha \min \{ \bm{b} \} - \zeta \geq 0, \;\; \alpha \in \mathbb{R}_+, \;\; \zeta \in \mathbb{R} & \text{(Primal Feasibility)} \\
	\eta, \gamma \in \mathbb{R}_+ & \text{(Dual Feasibility)} \\
	\eta (1 + \alpha \min \{ \bm{b} \} - \zeta) = 0, \;\; \alpha \gamma = 0 & \text{(Complementary Slackness)}
	\end{array}
	\end{equation*}
	The optimal value of problem~\eqref{eq:burg_intermediate_problem} is non-negative since $(\alpha, \zeta) = \bm{0}$ satisfies the constraints of~\eqref{eq:burg_intermediate_problem}. Hence, complementary slackness implies that $\eta^\star = 0$, as otherwise $1 + \alpha^\star \min \{ \bm{b} \} - \zeta^\star = 0$ would imply that the optimal objective value of problem~\eqref{eq:burg_intermediate_problem} was $- \infty$. Multiplying the first stationarity condition with $\alpha^\star$ and the second one with $1 - \zeta^\star$ and summing up then yields
	\begin{align*}
	& \alpha^\star \left( \sum_{s' \in \mathcal{S}} \overline{p}_{sas'} \cdot \frac{b_{s'}}{1+\alpha^\star b_{s'} - \zeta^\star} \right) + (1 - \zeta^\star) \left( \sum_{s' \in \mathcal{S}} \overline{p}_{sas'} \cdot \frac{1}{1+\alpha^\star b_{s'} - \zeta^\star} \right) \; = \; \alpha^\star (\beta - \gamma^\star) + (1 - \zeta^\star) \\
	~\mspace{-25mu} \Longleftrightarrow \quad
	& \sum_{s' \in \mathcal{S}} \overline{p}_{sas'} \cdot \frac{1+\alpha^\star b_{s'} - \zeta^\star}{1+\alpha^\star b_{s'} - \zeta^\star} \; = \; \alpha^\star (\beta - \gamma^\star) + (1 - \zeta^\star)
	\quad \Longleftrightarrow \quad
	\zeta^\star = \alpha^\star \beta,
	\end{align*}
	where the right-hand side of the first line exploits the fact that $\eta^\star = 0$ and the last equivalence uses complementary slackness to replace $\alpha^\star \gamma^\star$ with $0$. The result now follows from substituting $\zeta^\star$ with $\alpha^\star \beta$ in problem~\eqref{eq:burg_intermediate_problem} and rescaling $\alpha$ via $\alpha \leftarrow (\beta - \min \{ \bm{b} \}) \alpha$.
\end{proof}

\begin{thm}\label{thm:burg_complexity}
    If $\beta \geq \min \{ \bm{b} \} + \omega$ for some $\omega > 0$, then the projection problem~\eqref{eq:gen_projection} can be solved to $\delta$-accuracy in time $\mathcal{O} (S \cdot \log [\max \{ \bm{b} \} / (\delta \omega)])$.
\end{thm}

\begin{proof}
    Similar to the proof of Theorem~\ref{thm:kl_complexity}, we show the statement in three steps. Step~1 argues that $f (\alpha)$, which we henceforth use to denote the objective function of problem~\eqref{eq:burg_entropy_projection_problem}, is well-defined and continuously differentiable on the half-open interval $\alpha \in [0, 1)$ with a positive derivative at $0$ and a negative derivative close to $1$, respectively. This ensures that the optimum is attained on the open interval $\alpha \in (0, 1)$. Step~2 derives a global upper bound on $f' (\alpha)$, which will allow us to bound the maximum objective function value over any interval $[\underline{a}, \overline{\alpha}] \subseteq \mathbb{R}_+$ due to the concavity of $f$. Step~3, finally, employs a bisection search to solve~\eqref{eq:burg_entropy_projection_problem} to $\delta$-accuracy in the stated complexity.

    In view of the first step, we note that for $\alpha \in [0, 1)$ we have
    \begin{equation*}
        \begin{array}{r@{}l@{}l}
            \displaystyle (1 - \alpha) (\beta - \min \{ \bm{b} \}) > 0
            \quad \Longleftrightarrow& \quad
            \displaystyle \beta - \min \{ \bm{b} \} + \alpha (\min \{ \bm{b} \} - \beta) > 0 \\
            \Longrightarrow& \quad
            \displaystyle \beta - \min \{ \bm{b} \} + \alpha ( b_{s'} - \beta) > 0, & \forall s' \in \mathcal{S}, \\
            \Longleftrightarrow& \quad
            \displaystyle 1 + \alpha \frac{b_{s'} - \beta}{\beta - \min \{ \bm{b} \}} > 0, & \forall s' \in \mathcal{S},
        \end{array}
    \end{equation*}
    and thus the expression inside the logarithm of $f (\alpha)$ is strictly positive for all $s' \in \mathcal{S}$. Here, the first inequality holds by assumption, and the last equivalence follows from a division by $\beta - \min \{ \bm{b} \}$, which is strictly positive by assumption. We then observe that for $\alpha \in [0, 1)$, we have
    \begin{equation*}
        \mspace{-12mu}
        f' (\alpha)
        \;\; = \;\;
        \sum_{s' \in \mathcal{S}} \overline{p}_{sas'} \left[ \left( 1 + \alpha \frac{b_{s'} - \beta}{\beta - \min \{ \bm{b} \}} \right)^{-1} \cdot \frac{b_{s'} - \beta}{\beta - \min \{ \bm{b} \}} \right]
        \;\; = \;\;
        \sum_{s' \in \mathcal{S}} \overline{p}_{sas'} \left[ \frac{b_{s'} - \beta}{\beta - \min \{ \bm{b} \} + \alpha (b_{s'} - \beta)} \right].
    \end{equation*}
    In particular, we have $f' (0) = (\overline{\bm{p}}_{sa}{}^\top \bm{b} - \beta) / (\beta - \min \{ \bm{b} \})$, which is positive since $\beta \in \left( \min \{ \bm{b} \}, \; \overline{\bm{p}}_{sa}{}^\top \bm{b} \right)$ by assumption. (Recall that the projection problem is trivial if $\overline{\bm{p}}_{sa}{}^\top \bm{b} \leq \beta$.) For $\alpha \uparrow 1$, on the other hand, the fractions in $f' (\alpha)$ corresponding to the indices $s' \in \mathcal{S}$ with $b_{s'} = \min \{ \bm{b} \}$ evaluate to $1 / (\alpha - 1) \longrightarrow - \infty$, whereas the other fractions evaluate to
    \begin{equation*}
        \frac{b_{s'} - \beta}{(\beta - \min \{ \bm{b} \}) (1 - \alpha) + \alpha (b_{s'} - \min \{ \bm{b} \})}
        \; \longrightarrow \;
        \frac{b_{s'} - \beta}{\alpha (b_{s'} - \min \{ \bm{b} \})}
    \end{equation*}
    and thus remain finite. In conclusion, we have $f' (\alpha) < 0 $ for $\alpha$ near $1$.

    As for the second step, we observe that
    \begin{equation*}
        f' (\alpha)
        \;\; \leq \;\;
        f' (0)
        \;\; = \;\;
        \frac{\overline{\bm{p}}_{sa}{}^\top \bm{b} - \beta}{\beta - \min \{ \bm{b} \}}
        \;\; \leq \;\;
        \frac{\max \{ \bm{b} \}}{\beta - \min \{ \bm{b} \}}
        \;\; \leq \;\;
        \frac{\max \{ \bm{b} \}}{\omega},
    \end{equation*}
    where the inequalities follow from the concavity of $f$, the fact that $\overline{\bm{p}}_{sa} \in \Delta_S$ as well as $\beta \geq 0$, and because $\beta - \min \{ \bm{b} \} \geq \omega$, respectively. Similar arguments as in the proof of Theorem~\ref{thm:kl_complexity} then allow us to closely bound the optimal value of problem~\eqref{eq:burg_entropy_projection_problem} from below and above by $f (\underline{\alpha})$ and $f (\underline{\alpha}) + (\max \{ \bm{b} \} / \omega) \cdot (\overline{\alpha} - \underline{\alpha})$, respectively, whenever $\alpha^\star \in [\underline{\alpha}, \overline{\alpha}]$.

    In view of the third step, finally, we bisect on $\alpha$ by starting with the initial bounds $(\underline{\alpha}^0, \overline{\alpha}^0) = (0, 1)$, halving the length of the interval $[\underline{\alpha}^i, \overline{\alpha}^i]$ in each iteration $i = 0, 1, \ldots$ by verifying whether $f' ([\underline{\alpha}^i + \overline{\alpha}^i]/2)$ is positive and terminating once $\overline{\alpha}^i - \underline{\alpha}^i \leq \delta \omega / \max \{ \bm{b} \}$. Similar arguments as in the proof of Theorem~\ref{thm:kl_complexity} show that this is the case after $\mathcal{O} (\log [\max \{ \bm{b} \} / (\delta \omega)])$ iterations. The statement now follows since evaluating the derivative $f' ([\underline{\alpha}^i + \overline{\alpha}^i] / 2)$ in each bisection step takes time $\mathcal{O} (S)$.
\end{proof}

The proof of Theorem~\ref{thm:overall_complexity:inexact_subproblem} employs an outer bisection over $\theta$ that requires for each $(s, a) \in \mathcal{S} \times \mathcal{A}$ the repeated solution of the projection problem~\eqref{eq:kl_divergence_projection_problem} with $\bm{b} = \bm{r}_{sa} + \lambda \bm{v}$ and $\beta = \theta \in [ \underline{R}_s (\bm{v}) + \frac{\epsilon}{2}, \overline{R} - \frac{\epsilon}{2}]$ (since the outer bisection is stopped when the interval length no longer exceeds $\epsilon$) to an accuracy of $\delta = \epsilon \kappa / [2 A \overline{R} + A \epsilon]$. In that case, for each $(s, a) \in \mathcal{S} \times \mathcal{A}$ we have $\max \{ \bm{b} \} \leq \overline{R}$. Plugging this estimate into the statement of Theorem~\ref{thm:burg_complexity}, we see that the projection problem~\eqref{eq:burg_entropy_projection_problem} is solved in time $h (\epsilon \kappa / [2 A \overline{R} + A \epsilon]) = \mathcal{O} (S \cdot \log [A \overline{R}^2 / (\epsilon^2 \kappa)])$.

\section{Numerical Results}\label{sec:numericals}

We study the empirical performance of our unified solution framework for robust Markov decision processes with $s$-rectangular ambiguity sets and compare it against three state-of-the-art commercial solvers as well as a tailored homotopy-based solution scheme from the literature. Our experiments are conducted on both synthetically generated instances and standard benchmark instances. The synthetic instances enable us to investigate how runtimes scale with the number of states $S$ and actions $A$, whereas the benchmark instances capture structural features of real-world MDPs---such as heterogeneous transition dynamics, uneven reward distributions and sparsity patterns---that are difficult to reproduce in purely synthetic settings.

All experiments were run on AMD EPYC 7742 cluster nodes with 32GB of RAM using a single computational thread throughout. All algorithms were implemented in \texttt{C++} (GCC~8.5.0) and compiled with the \texttt{-O3} optimization flag. We used CPLEX~22.1.1, Gurobi~13.0 and MOSEK~11.1 as commercial solvers. Apart from restricting all solvers to single-threaded execution, we relied on their default parameter settings.

In the following, we first investigate the performance of our projection algorithms (Section~\ref{sec:numericals-projection}). We then utilize Theorems~\ref{thm:overall_complexity:exact_subproblem} and~\ref{thm:overall_complexity:inexact_subproblem} to obtain runtimes for the robust Bellman iteration from Section~\ref{sec:numericals-bellman}. We close with end-to-end solution times for a robust value iteration on benchmark instances in Section~\ref{sec:numericals-vi}. Additional numerical results are deferred to Appendix~B. All source codes, instances and results can be found on the GitHub repository accompanying this work.\footnote{GitHub repository: \url{https://github.com/wolframwi/fast-robust-mdps}.}
 
\subsection{Projection Problems}\label{sec:numericals-projection}

We begin by evaluating the performance of our algorithms for the generalized $d_a$-projection problems. To this end, we generate synthetic instances with $S \in \{10, 20, \ldots, 100\}$ states and $A \in \{10, S\}$ actions. For every state–action pair $(s,a)$ in an instance, the next-state support size is set to $k = \max\{2, \lceil 0.30\,S \rceil\}$. A support set of size $k$ is sampled uniformly without replacement, and the corresponding transition probabilities are drawn from a symmetric Dirichlet distribution with all concentration parameters equal to $\eta = 1$. All non-support next states are assigned a transition probability of zero. Rewards are sampled independently according to $r(s, a, s') \sim \mathcal{U}[0,1]$. The initial state distribution is uniform, $p^0(s)=1/S$, and the discount factor is fixed to $\lambda=0.99$.

Ambiguity is specified using uniform weights $\sigma(s,a,s')=1$ for both the $\ell_1$- and $\ell_2$-norm ambiguity sets. Uncertainty is calibrated via a target total-variation (TV) radius per transition row. Recall that the TV distance between two distributions $\bm{p}$ and $\bm{q}$ is defined as $\mathrm{TV}(\bm{p},\bm{q}) := \tfrac{1}{2}\lVert \bm{p}-\bm{q} \rVert_1$. We set the TV radius to $\tau=0.05$ and allow for all perturbations that satisfy $\mathrm{TV}(\bm{p},\bm{q}) \le \tau$, which corresponds to an $\ell_1$ ambiguity radius of $\rho = 2\tau$. The ambiguity radii for the $\ell_2$-norm, the KL divergence and the Burg entropy are chosen so as to induce uncertainty on the same TV scale. For the KL divergence and Burg entropy, this calibration is based on Pinsker's inequality, which implies $\mathrm{TV}(\bm{p},\bm{q}) \le \tfrac{1}{2}\sqrt{D_{\mathrm{KL}}(\bm{p}|\bm{q})}$; matching this upper bound to the target TV radius $\tau$ yields KL and Burg radii of $(2\tau)^2/2$. For the $\ell_2$-norm, we adopt the same quadratic scaling, resulting in an $\ell_2$ radius of $(2\tau)^2$, which provides a comparable uncertainty budget under uniform weights. This yields $\rho=0.10$ for the $\ell_1$-norm, $\rho=0.01$ for the $\ell_2$-norm, and $\rho=0.005$ for KL the divergence and the Burg entropy. For each instance size, we generate 10 independent instances.

Our projection algorithms (referred to as `fast' in the figures and tables henceforth) use tolerances of $10^{-12}$ for slope comparisons and $10^{-10}$ for bisection. For each instance, we solve one projection problem for every state–action pair $(s,a)$, resulting in $S\cdot A$ projections per instance. The value vector $\bm{v}$ is generated once per instance, with its components drawn uniformly at random from the interval $[0, \overline{R}]$. For each projection associated with $(s,a)$, the projection constraint is set to $\beta = \tfrac{1}{2}\bigl(\overline{\bm{p}}_{sa}{}^\top \bm{v} + \min_{s' \in \mathcal{S}} v_{s'}\bigr)$ if $\overline{\bm{p}}_{sa}{}^\top \bm{v} > \min_{s' \in \mathcal{S}} v_{s'}$ and to $\overline{\bm{p}}_{sa}{}^\top \bm{v}$ otherwise.

Figure~\ref{fig:projection} compares the runtimes of our projection
algorithms against those obtained by solving the corresponding
projection problems using commercial solvers. Since MOSEK is the only solver among those considered that natively supports the logarithmic expressions arising in the KL divergence and Burg entropy, comparisons in those settings are restricted to that solver. We report median runtimes across the 10 instances and $S \cdot A$ projections for each problem size. Speedups for the $\ell_1$-norm projection problem range from 240$\times$–650$\times$ (CPLEX), 270$\times$–450$\times$ (Gurobi), and 510$\times$–1{,}500$\times$ (MOSEK). For the $\ell_2$-norm, speedups are in the range
40$\times$–450$\times$ (CPLEX), 40$\times$–240$\times$ (Gurobi), and
110$\times$–1{,}200$\times$ (MOSEK). For KL divergence, speedups are
100$\times$–200$\times$ (MOSEK), while for Burg entropy they are
570$\times$–730$\times$ (MOSEK).

\begin{figure}[tb]
    \begin{tabular}{cc}
        \includegraphics[width = 0.5 \textwidth]{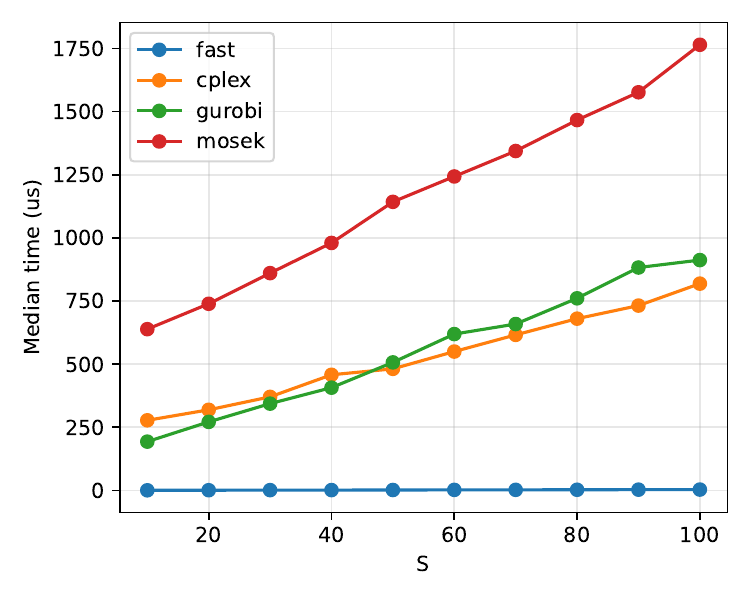} &
        \includegraphics[width = 0.5 \textwidth]{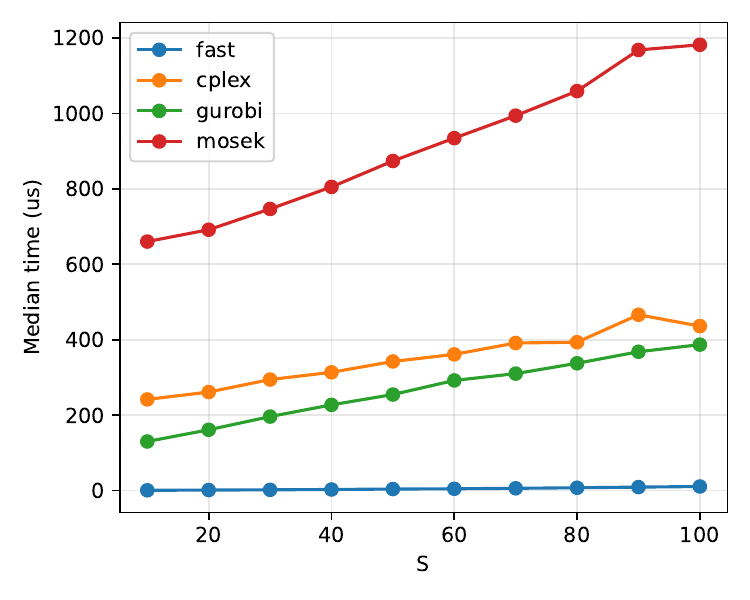} \\
        \includegraphics[width = 0.5 \textwidth]{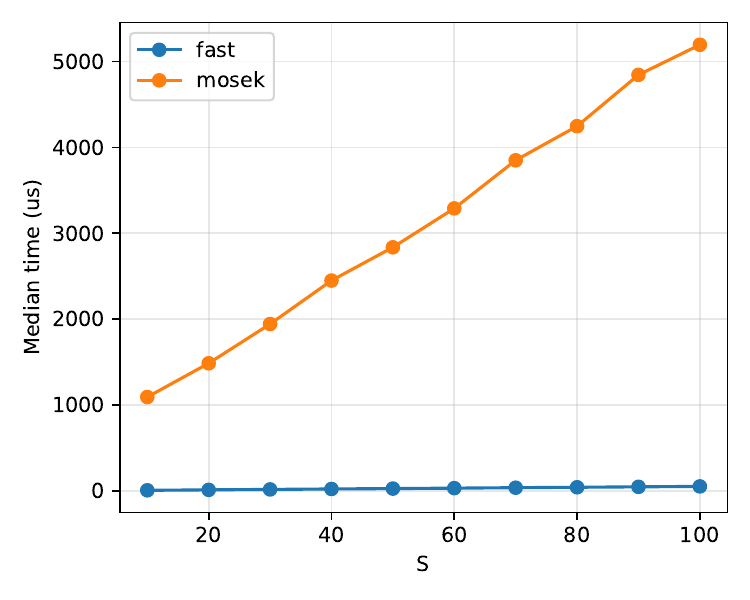} &
        \includegraphics[width = 0.5 \textwidth]{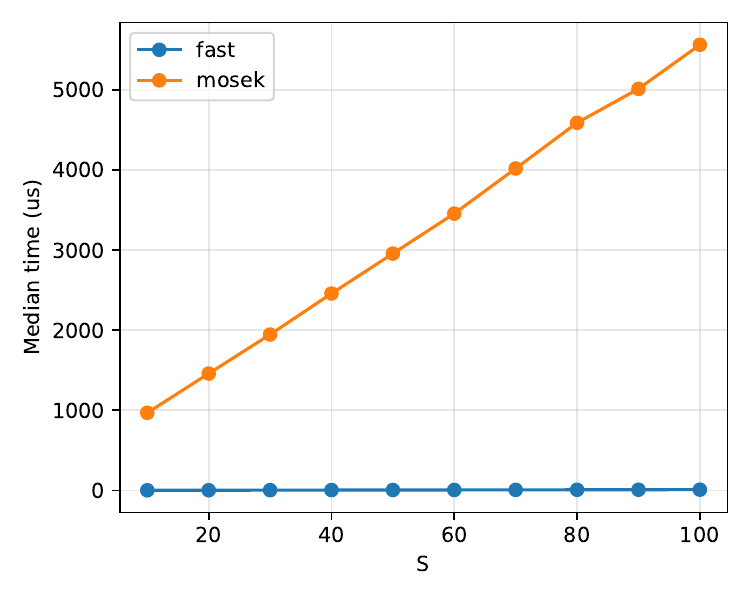}
    \end{tabular}
    \vspace{-0.75cm}
    \caption{Median runtimes (in $\mu\text{s}$) of the projection problems for the $\ell_1$-norm (top left), $\ell_2$-norm (top right), KL divergence (bottom left), and Burg entropy (bottom right). \label{fig:projection}}
\end{figure}

\subsection{Robust Bellman Operator}\label{sec:numericals-bellman}

We next turn to the evaluation of the robust Bellman operator. For each
instance, we sample 100 state–action pairs $(s,a)$ uniformly at random
and measure the runtime of a single robust Bellman update for each
sample. As before, the value vector $\bm{v}$ is generated once per instance, with its components drawn uniformly at random from the interval $[0, \overline{R}]$. For each problem size, we report the median runtime across the 100 samples and 10 instances. The results are shown in Figures~\ref{fig:bellman-l1}--\ref{fig:bellman-burg}.

For $\ell_1$-norm ambiguity, speedups of the robust Bellman operator are
of the order $10\times$--$430\times$ relative to CPLEX,
$10\times$--$280\times$ relative to Gurobi, and $40\times$--$250\times$
relative to MOSEK. Under $\ell_2$-norm ambiguity, the corresponding
speedups are of the order $10\times$--$260\times$ (CPLEX),
$6\times$--$70\times$ (Gurobi), and $4\times$--$120\times$ (MOSEK). For
KL divergence and Burg entropy, comparisons are restricted to MOSEK, with
speedups of the order $30\times$--$300\times$ and $70\times$--$600\times$,
respectively.

The speedups observed at the level of the robust Bellman operator closely track those obtained for the underlying projection problems, but they are systematically damped. Indeed, our solver-based implementations employ structured reformulations of the robust Bellman operator. Specifically, we invoke the minimax theorem to first exchange the order of the maximization over $\bm{\pi}_s$ and the minimization over $\bm{p}_s$ in~\eqref{eq:rob_value_it}, which allows us to subsequently replace the maximization over randomized policies $\bm{\pi}_s$ with a maximization over deterministic actions $a \in \mathcal{A}$. This reformulation admits an efficient epigraph representation and enables the solvers to better exploit problem structure. As a result, the relative performance gap at the Bellman operator level is smaller than for the isolated projection problems, while the speedups remain significant.

In addition to the three commercial solvers, we compare against the homotopy continuation method of \cite{HPW18:fast_bellman} in the $\ell_1$-norm setting, which is the only ambiguity class supported by that approach. While the homotopy method is competitive on small instances, our robust Bellman evaluation increasingly outperforms it as the problem size grows, achieving speedups of up to $30\times$ on the largest instances. This result is consistent with the theoretical complexity analysis of both algorithms.

\begin{figure}[tb]
    \begin{tabular}{cc}
        \includegraphics[width = 0.5 \textwidth]{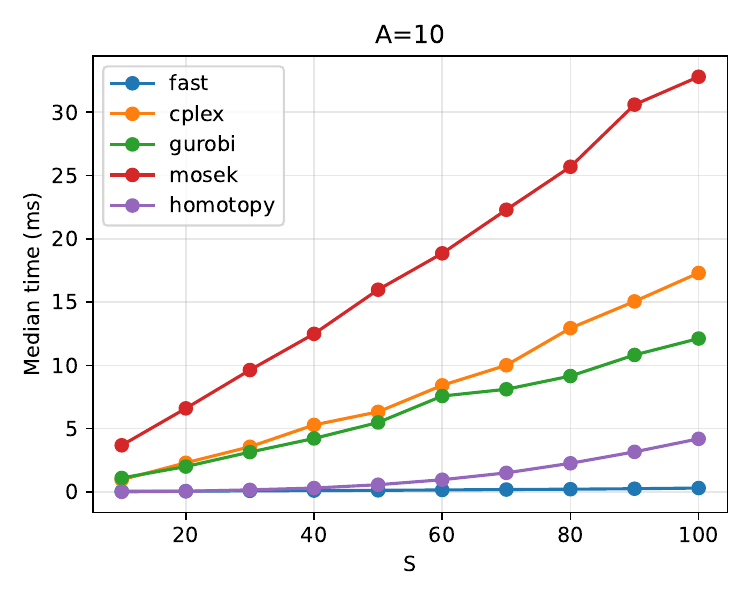} &
        \includegraphics[width = 0.5 \textwidth]{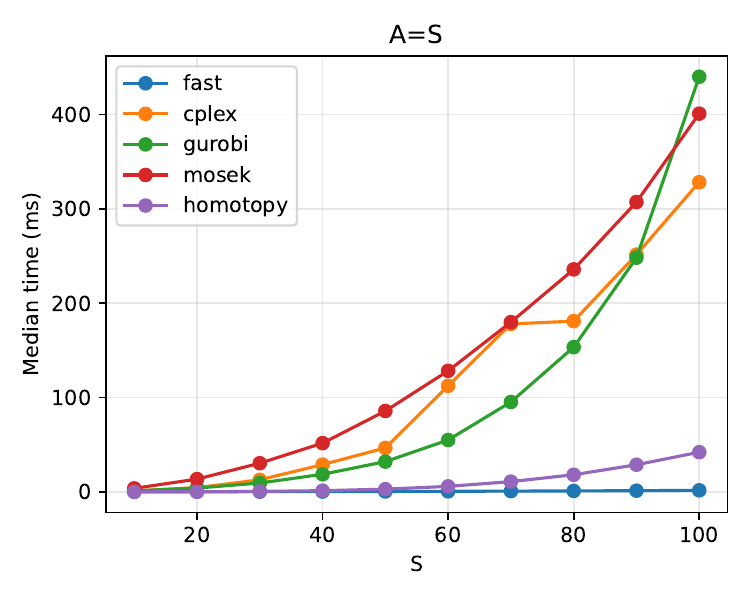}
    \end{tabular}
    \vspace{-0.75cm}
    \caption{Median runtimes (in ms) of the robust Bellman operator under $\ell_1$-norm ambiguity, with $A = 10$ actions (left) and $A = S$ actions (right).
\label{fig:bellman-l1}}
\end{figure}

\begin{figure}[tb]
    \begin{tabular}{cc}
        \includegraphics[width = 0.5 \textwidth]{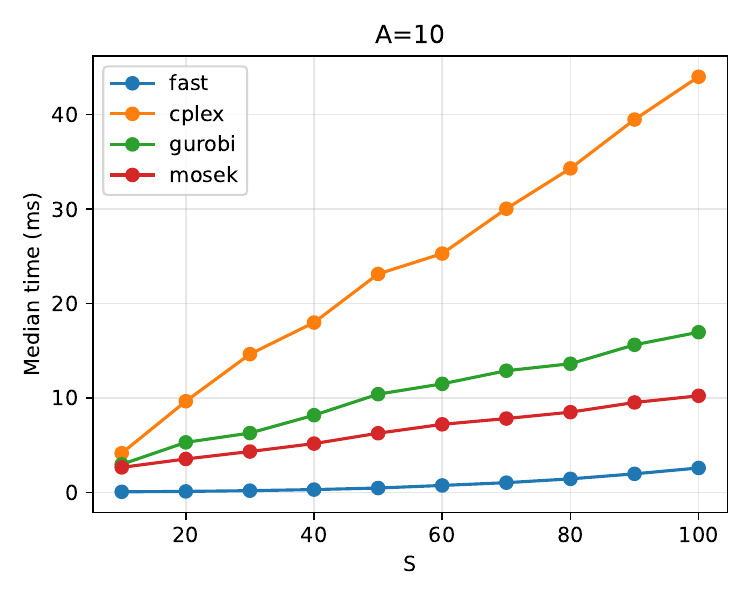} &
        \includegraphics[width = 0.5 \textwidth]{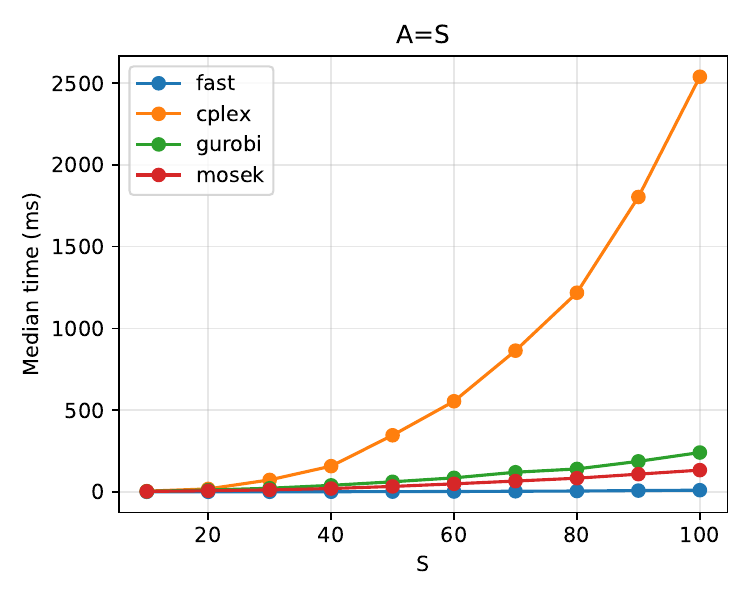}
    \end{tabular}
    \vspace{-0.75cm}
    \caption{Median runtimes (in ms) of the robust Bellman operator under $\ell_2$-norm ambiguity, with $A = 10$ actions (left) and $A = S$ actions (right).
\label{fig:bellman-l2}}
\end{figure}

\begin{figure}[tb]
    \begin{tabular}{cc}
        \includegraphics[width = 0.5 \textwidth]{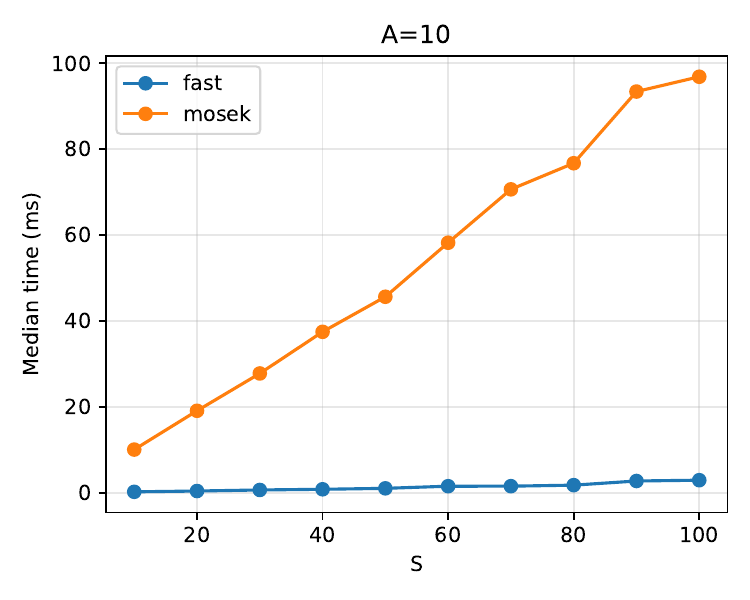} &
        \includegraphics[width = 0.5 \textwidth]{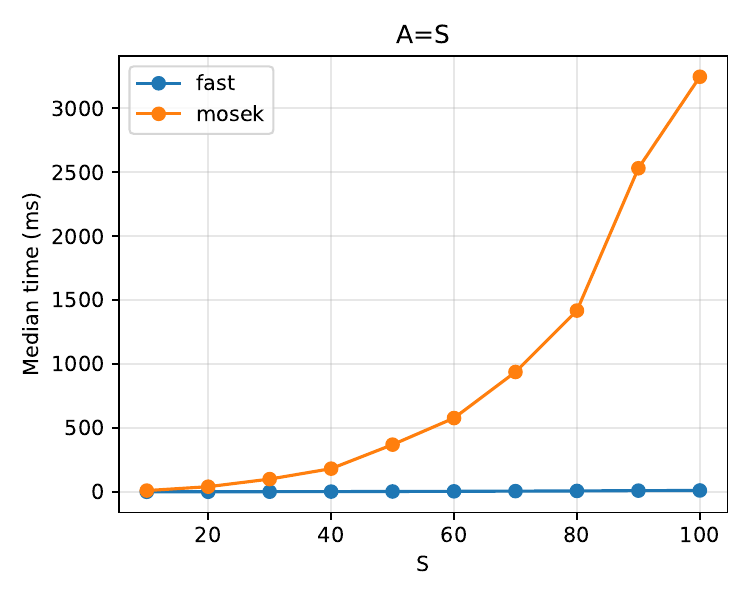}
    \end{tabular}
    \vspace{-0.75cm}
    \caption{Median runtimes (in ms) of the robust Bellman operator under KL divergence ambiguity, with $A = 10$ actions (left) and $A = S$ actions (right).
\label{fig:bellman-kl}}
\end{figure}

\begin{figure}[tb]
    \begin{tabular}{cc}
        \includegraphics[width = 0.5 \textwidth]{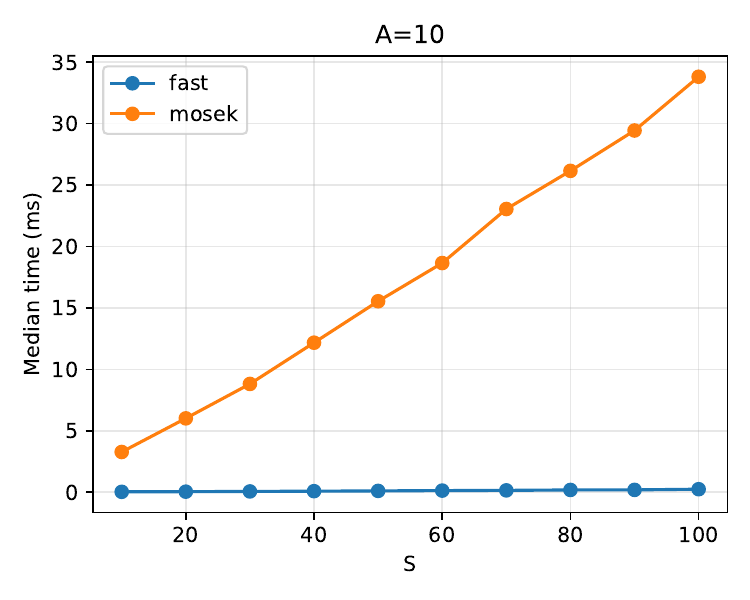} &
        \includegraphics[width = 0.5 \textwidth]{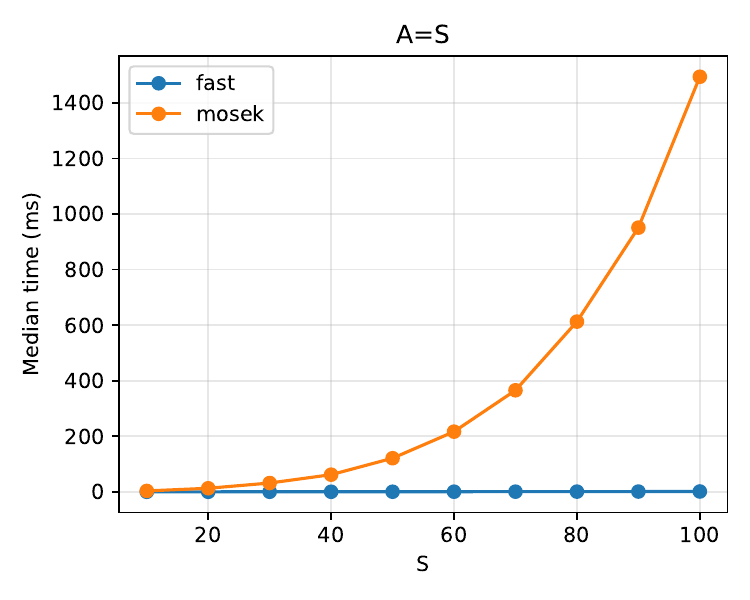}
    \end{tabular}
    \vspace{-0.75cm}
    \caption{Median runtimes (in ms) of the robust Bellman operator under Burg entropy ambiguity, with $A = 10$ actions (left) and $A = S$ actions (right). \label{fig:bellman-burg}}
\end{figure}

\subsection{Robust Value Iteration}\label{sec:numericals-vi}

We finally evaluate the end-to-end runtimes of a na\"ive robust value iteration (VI) that incorporates the robust Bellman operators from the preceding experiments. For each instance, we perform a single run of robust VI until convergence, defined by a termination threshold of $10^{-5}$ in $\ell_\infty$-norm. In addition to the synthetic instances from before, we include standard benchmark MDPs.

Our benchmark suite combines standard reinforcement learning environments and
classical textbook Markov decision processes. We include the \texttt{Blackjack}, \texttt{CliffWalking},
\texttt{FrozenLake} (4$\times$4 and 8$\times$8, with slipperiness), and \texttt{Taxi} domains from the
\texttt{toy\_text} collection in Gymnasium (\url{https://gymnasium.farama.org}).
For these environments, nominal transition probabilities are obtained by
exact enumeration of the underlying dynamics or by direct extraction from
the Gymnasium transition tables, using the default reward specifications
and initial state distributions given by \texttt{env.unwrapped.isd}.
We further include the \texttt{forest50} instance from
\texttt{pymdptoolbox.example.forest}, which implements the classical forest
management benchmark with $S=50$ states.
An additional grid-based instance, \texttt{openspiel\_grid16}, is derived
from the single-agent pathfinding rules in OpenSpiel (\url{https://github.com/google-deepmind/open_spiel}) on an empty $4\times4$ grid with standard OpenSpiel rewards.
Finally, we consider a collection of widely used textbook-style benchmarks,
including \texttt{chain10}, \texttt{riverswim6}, \texttt{riverswim20},
\texttt{gridworld25}, \texttt{capacity50}, \texttt{inventory50},
\texttt{perishable50}, and \texttt{machine20}, whose parameterizations
follow common constructions in the dynamic programming and robust MDP
literature \citep{P94:mdp, strehl2008analysis, }.
All benchmark MDPs use a discount factor of $\lambda = 0.99$.

Tables~\ref{tab:vi-l1-l2} and~\ref{tab:vi-kl-burg} report end-to-end
runtimes for robust value iteration under $\ell_1$- and $\ell_2$-norm
ambiguity, as well as under KL divergence and Burg entropy, respectively,
on both synthetic and benchmark instances. For each instance, we report
the total runtime of a single robust value iteration run until
convergence.

Across all settings, our implementation consistently outperforms the
commercial solvers, often by one to several orders of magnitude. The
largest gains are observed on synthetic instances with large state and
action spaces, as well as on benchmark problems with more complex
transition structures. In the $\ell_1$-norm setting, we also compare
against the homotopy continuation method of \cite{HPW21:ppi},
which is competitive on smaller instances but scales less favorably.
Overall, robust value iteration based on our specialized algorithms is
substantially faster and scales reliably to problem sizes at which both
generic solvers and the homotopy approach become prohibitively slow
and/or encounter numerical difficulties.

\begin{table}[t]
\caption{$\ell_1$- and $\ell_2$-norm robust value iteration runtimes on
synthetic and benchmark instances (in seconds). CPLEX exceeded the time
limit on the largest synthetic instances and encountered numerical
issues on the \texttt{taxi} benchmark. \label{tab:vi-l1-l2}}
\hspace{-2.5cm}
\footnotesize
\begin{tabular}{l|rrrrr|rrrr}
Instance & L1 Fast & L1 CPLEX & L1 Gurobi & L1 Mosek & L1 Homotopy & L2 Fast & L2 CPLEX & L2 Gurobi & L2 Mosek \\
\hline
synthetic ($S=10$, $A=10$) & 0.32 & 8.76 & 9.35 & 35.60 & 0.11 & 0.43 & 31.23 & 15.28 & 21.63 \\
synthetic ($S=20$, $A=10$) & 0.90 & 34.55 & 38.30 & 131.84 & 1.03 & 1.61 & 127.33 & 50.45 & 63.43 \\
synthetic ($S=30$, $A=10$) & 1.80 & 79.19 & 87.57 & 272.19 & 4.20 & 3.85 & 289.49 & 104.24 & 105.28 \\
synthetic ($S=40$, $A=10$) & 3.30 & 145.45 & 144.37 & 466.93 & 11.27 & 8.59 & 481.52 & 183.38 & 164.79 \\
synthetic ($S=50$, $A=10$) & 4.79 & 247.83 & 230.86 & 747.41 & 25.78 & 16.47 & 781.01 & 287.33 & 240.68 \\
synthetic ($S=60$, $A=10$) & 7.23 & 389.06 & 377.18 & 1,049.32 & 52.54 & 30.06 & 1,027.18 & 411.73 & 325.34 \\
synthetic ($S=70$, $A=10$) & 10.48 & 564.72 & 570.39 & 1,416.99 & 91.73 & 49.63 & 1,370.74 & 560.45 & 446.30 \\
synthetic ($S=80$, $A=10$) & 13.57 & 752.75 & 687.27 & 1,831.23 & 164.10 & 75.59 & 1,712.59 & 737.30 & 561.46 \\
synthetic ($S=90$, $A=10$) & 18.96 & 1,038.82 & 825.52 & 2,318.79 & 247.93 & 114.80 & 2,307.91 & 956.41 & 686.20 \\
synthetic ($S=100$, $A=10$) & 23.19 & 1,272.28 & 1,081.04 & 2,934.97 & 379.10 & 161.46 & 2,858.56 & 1,282.02 & 833.61 \\
\hline
synthetic ($S=20$, $A=20$) & 1.51 & 63.95 & 74.18 & 250.28 & 2.14 & 2.53 & 252.47 & 99.92 & 99.42 \\
synthetic ($S=30$, $A=30$) & 3.89 & 225.97 & 218.44 & 780.74 & 12.51 & 7.68 & 1,399.57 & 322.27 & 288.52 \\
synthetic ($S=40$, $A=40$) & 8.31 & 625.65 & 617.02 & 1,944.94 & 45.78 & 18.53 & 3,806.36 & 734.32 & 637.73 \\
synthetic ($S=50$, $A=50$) & 14.78 & 1,197.84 & 1,198.48 & 3,866.64 & 130.71 & 38.87 & 10,428.00 & 1,428.66 & 1,274.13 \\
synthetic ($S=60$, $A=60$) & 24.48 & 2,238.33 & 2,580.78 & 6,219.14 & 329.31 & 77.01 & 23,031.25 & 2,512.82 & 2,106.52 \\
synthetic ($S=70$, $A=70$) & 37.92 & 3,891.67 & 4,978.35 & 10,261.05 & 659.18 & 138.96 & 38,795.90 & 4,040.86 & 3,296.62 \\
synthetic ($S=80$, $A=80$) & 53.06 & 6,404.92 & 8,614.56 & 16,014.75 & 1,247.10 & 234.71 & 68,192.10 & 6,027.53 & 4,890.48 \\
synthetic ($S=90$, $A=90$) & 73.20 & 10,370.05 & 16,797.80 & 23,302.25 & 2,220.68 & 384.67 & --- & 8,658.69 & 6,981.41 \\
synthetic ($S=100$, $A=100$) & 103.71 & 16,941.10 & 32,160.25 & 33,363.20 & 4,240.86 & 589.63 & --- & 12,097.85 & 10,013.65 \\
\hline
blackjack & 77.78 & 1,653.75 & 1,061.72 & 3,014.06 & 30.41 & 249.26 & 4,471.39 & 2,615.19 & 1,570.60 \\
capacity50 & 0.26 & 5.88 & 6.67 & 15.92 & 0.12 & 0.57 & --- & 9.36 & 7.72 \\
chain10 & 0.11 & 3.41 & 3.94 & 9.61 & 0.01 & 0.13 & 7.68 & 3.74 & 10.36 \\
cliffwalking & 0.79 & 46.95 & 79.89 & 191.04 & 0.13 & 1.06 & 230.63 & 85.74 & 95.27 \\
forest50 & 1.24 & 35.77 & 32.69 & 111.07 & 0.11 & 2.23 & 118.71 & 63.13 & 66.53 \\
frozenlake4x4 & 0.02 & 0.77 & 0.59 & 2.62 & 0.00 & 0.02 & 1.03 & 0.51 & 0.82 \\
frozenlake8x8 & 0.30 & 11.16 & 9.58 & 35.00 & 0.08 & 0.19 & 12.03 & 6.58 & 5.64 \\
gridworld25 & 0.02 & 0.34 & 0.35 & 1.29 & 0.00 & 0.02 & 1.68 & 0.90 & 1.06 \\
inventory50 & 5.10 & 107.97 & 108.56 & 320.69 & 2.56 & 10.59 & 421.55 & 162.61 & 162.18 \\
machine20 & 0.17 & 7.99 & 6.20 & 23.06 & 0.02 & 0.20 & 24.86 & 10.60 & 18.21 \\
openspiel$\underline{~}$grid16 & 0.02 & 0.39 & 0.34 & 1.20 & 0.00 & 0.03 & 5.23 & 4.77 & 1.00 \\
perishable50 & 3.29 & 116.33 & 106.19 & 357.29 & 2.55 & 4.38 & 365.81 & 174.50 & 160.55 \\
riverswim20 & 0.09 & 2.68 & 2.35 & 9.62 & 0.02 & 0.12 & 7.58 & 2.98 & 5.34 \\
riverswim6 & 0.02 & 0.74 & 0.51 & 2.22 & 0.00 & 0.03 & 1.55 & 0.74 & 1.90 \\
taxi & 560.31 & 7,869.75 & 10,059.20 & 33,971.10 & 31.23 & 626.01 & --- & 26,030.40 & 16,071.50 \\ \hline \hline
\end{tabular}
\end{table}

\begin{table}[t]
\centering
\caption{KL divergence and Burg entropy robust value iteration runtimes
on synthetic and benchmark instances (in seconds). Missing entries
correspond to instances where MOSEK encountered numerical issues. \label{tab:vi-kl-burg}}
\begin{tabular}{l|rr|rr}
Instance & KL Fast & KL Mosek & Burg Fast & Burg Mosek \\
\hline
synthetic ($S=10$, $A=10$) & 1.28 & 68.01 & 0.67 & 48.31 \\
synthetic ($S=20$, $A=10$) & 5.29 & 309.12 & 2.33 & 185.36 \\
synthetic ($S=30$, $A=10$) & 12.57 & 747.70 & 5.01 & 420.42 \\
synthetic ($S=40$, $A=10$) & 24.06 & 1,406.86 & 8.76 & 738.22 \\
synthetic ($S=50$, $A=10$) & 34.26 & 2,350.78 & 14.28 & 1,158.16 \\
synthetic ($S=60$, $A=10$) & 56.11 & 3,324.76 & 21.76 & 1,674.24 \\
synthetic ($S=70$, $A=10$) & 81.33 & 5,525.52 & 27.26 & 2,204.15 \\
synthetic ($S=80$, $A=10$) & 99.80 & 7,088.89 & 34.64 & 2,831.80 \\
synthetic ($S=90$, $A=10$) & 135.35 & 9,177.22 & 47.38 & 3,593.31 \\
synthetic ($S=100$, $A=10$) & 159.59 & 11,913.70 & 56.27 & --- \\
\hline
synthetic ($S=20$, $A=20$) & 6.97 & 778.27 & 3.82 & 388.29 \\
synthetic ($S=30$, $A=30$) & 20.41 & --- & 11.60 & --- \\
synthetic ($S=40$, $A=40$) & 48.33 & --- & 27.32 & --- \\
synthetic ($S=50$, $A=50$) & 76.27 & --- & 54.15 & --- \\
synthetic ($S=60$, $A=60$) & 132.45 & --- & 87.27 & --- \\
synthetic ($S=70$, $A=70$) & 211.10 & --- & 125.27 & --- \\
synthetic ($S=80$, $A=80$) & 274.85 & --- & 186.08 & --- \\
synthetic ($S=90$, $A=90$) & 394.54 & --- & 260.07 & --- \\
synthetic ($S=100$, $A=100$) & 546.70 & --- & 369.63 & --- \\
\hline
blackjack & 672.13 & 9,801.14 & 132.99 & --- \\
capacity50 & 76.75 & 1,151.15 & 18.96 & 848.49 \\
chain10 & 1.05 & 14.69 & 0.31 & 14.72 \\
cliffwalking & 47.87 & 637.52 & 12.33 & 466.57 \\
forest50 & 18.39 & 293.37 & 5.43 & 265.44 \\
frozenlake4x4 & 1.98 & 29.22 & 0.54 & 25.58 \\
frozenlake8x8 & 24.80 & 374.63 & 5.68 & 289.12 \\
gridworld25 & 6.33 & 127.55 & 1.74 & 84.91 \\
inventory50 & 36.67 & 1,020.96 & 11.35 & 700.36 \\
machine20 & 3.12 & 46.26 & 0.87 & 44.78 \\
openspiel$\underline{~}$grid16 & 7.57 & 143.52 & 2.37 & 116.93 \\
perishable50 & 26.96 & 938.97 & 9.55 & 634.93 \\
riverswim20 & 2.04 & 25.25 & 0.49 & 23.74 \\
riverswim6 & 0.32 & 4.12 & 0.09 & 4.77 \\
taxi & 4,476.32 & --- & 1,297.72 & --- \\ \hline \hline
\end{tabular}
\end{table}

\paragraph{Use of Large Language Models.}
We used large language models (LLMs) as a support tool in the preparation of this manuscript. In particular, LLMs were employed to double-check our hand-derived proofs and to assist with the implementation of the algorithms. The code was subsequently checked manually. Any remaining errors or omissions are solely the responsibility of the authors.

\clearpage

\linespread{1}
\small
    
\bibliographystyle{plainnat}
\bibliography{bibliography}

\linespread{1.5}
\normalsize

\newpage
	
\section*{Appendix A: Proofs}

\noindent \textbf{Proof of Theorem~\ref{thm:overall_complexity:exact_subproblem}.} $\;$
	We compute an $\epsilon$-optimal solution $\bm{v}'$ to the robust Bellman iteration $\mathfrak{B} (\bm{v})$	component-wise. To this end, consider any component $v'_s$, $s \in \mathcal{S}$. Since $\Delta_A$ and $\mathcal P_s$ are nonempty, convex and compact, and the objective is bilinear in $(\bm\pi_s,\bm p_s)$, Sion's minimax theorem applies. We can thus equivalently reformulate the right-hand side of~\eqref{eq:rob_value_it} as the optimal value of the optimization problem
	\begin{equation}\label{eq:rob_value_it:reformulated}
	\begin{array}{l@{\quad}l}
	\text{minimize} & \displaystyle \max_{a \in \mathcal{A}} \; \left\{ \bm{p}_{sa}{}^\top (\bm{r}_{sa} + \lambda \bm{v}) \right\} \\
	\text{subject to} & \displaystyle \sum_{a \in \mathcal{A}} d_a (\bm{p}_{sa}, \overline{\bm{p}}_{sa}) \leq \kappa \\
	& \displaystyle \bm{p}_s \in (\Delta_S)^A.
	\end{array}
	\tag{\ref{eq:rob_value_it}'}
	\end{equation}
	In this reformulation, we have replaced the inner maximization over $\bm{\pi}_s \in \Delta_A$ with the maximization over the extreme points of $\Delta_A$, which is allowed since the objective function is linear in $\bm{\pi}_s$.
	
	We obtain $v'_s$ via bisection on the value of $[\mathfrak{B} (\bm{v})]_s$. We start the bisection with the lower and upper bounds $\underline{v}_{s0} = 0$ and $\overline{v}_{s0} = \overline{R}$, respectively, and halve the length of the interval $[\underline{v}_{si}, \overline{v}_{si}]$ in each iteration $i = 0, 1, \ldots$ by verifying whether or not $[\mathfrak{B} (\bm{v})]_s \leq (\underline{v}_{si} + \overline{v}_{si})/2$ and updating either the upper or the lower interval bound accordingly. The length of the interval no longer exceeds the accuracy $\epsilon$ once the iteration number $i$ satisfies
	\begin{equation*}
	2^{-i} \cdot (\overline{v}_{s0} - \underline{v}_{s0}) \leq \epsilon
	\quad \Longleftrightarrow \quad
	i \geq \log_2 \overline{R} - \log_2 \epsilon,
	\end{equation*}
	that is, after $\mathcal{O} (\log [\overline{R} / \epsilon])$ iterations.
	
	To verify whether $[\mathfrak{B} (\bm{v})]_s \leq \theta$ for some $\theta \in \mathbb{R}$, we check whether the optimal value of problem~\eqref{eq:rob_value_it:reformulated} does not exceed $\theta$. This is the case if and only if
	\begin{align*}
	& \exists \bm{p}_s \in (\Delta_S)^A \; : \;
	\sum_{a \in \mathcal{A}} d_a (\bm{p}_{sa}, \overline{\bm{p}}_{sa}) \leq \kappa
	\text{ and }
	\bm{p}_{sa}{}^\top (\bm{r}_{sa} + \lambda \bm{v}) \leq \theta \;\; \forall a \in \mathcal{A} \\
	\Longleftrightarrow \quad
	& \sum_{a \in \mathcal{A}} \min \left\{ d_a (\bm{p}_{sa}, \overline{\bm{p}}_{sa}) \, : \, \bm{p}_{sa} \in \Delta_S, \;\; \bm{p}_{sa}{}^\top (\bm{r}_{sa} + \lambda \bm{v}) \leq \theta \right\} \leq \kappa \\
	\Longleftrightarrow \quad
	& \sum_{a \in \mathcal{A}} \mathfrak{P} (\overline{\bm{p}}_{sa}; \bm{r}_{sa} + \lambda \bm{v}, \theta) \leq \kappa,
	\end{align*}
	which in turn can be verified by solving $A$ generalized $d_a$-projection problems.

	In conclusion, for each of the $S$ components of $\bm{v}'$, we need to execute $\mathcal{O} (\log [\overline{R} / \epsilon])$ bisection iterations of complexity $\mathcal{O} (A \cdot h (S))$ each, which proves the statement.
\qed

To analyze the quantitative stability of the robust Bellman operator $\mathfrak{B} (\bm{v})$, we study the dual of problem~\eqref{eq:rob_value_it:reformulated}, which by the proof of Theorem~\ref{thm:overall_complexity:exact_subproblem} is equivalent to the robust Bellman operator~\eqref{eq:rob_value_it}:
	\begin{equation}\label{eq:rob_value_it:dual}
	\begin{array}{l@{\quad}l}
	\text{maximize} & \displaystyle  -\kappa \omega + \mathbf{e}^\top \bm{\gamma} - \omega \sum_{a \in \mathcal{A}} d_a^{\star} \left( \frac{1}{\omega} \left[\bm{\theta}_a + \gamma_a \mathbf{e} - \alpha_a (\bm{r}_{sa} + \lambda \bm{v}) \right], \overline{\bm{p}}_{sa} \right)  \\
	\text{subject to} & \displaystyle \bm{\alpha} \in \Delta_A, \;\; \omega \in \mathbb{R}_+, \;\; \bm{\gamma} \in \mathbb{R}^A, \;\; \bm{\theta} \in \mathbb{R}_+^{AS}
	\end{array}
	\end{equation}
	Here, $d_a^\star (\bm{x}, \overline{\bm{p}}_{sa}) := \sup \, \{ \bm{p}_{sa}{}^\top \bm{x} - d_a  (\bm{p}_{sa}, \overline{\bm{p}}_{sa}) \, : \, \bm{p}_{sa} \in \mathbb{R}^S \}$ denotes the conjugate of the deviation function $d_a (\cdot, \overline{\bm{p}}_{sa})$ in the definition~\eqref{eq:additive_s_rect} of the ambiguity set $\mathcal{P}_s$, and the perspective function in~\eqref{eq:rob_value_it:dual} extends to $\omega = 0$ in the usual way \cite[Corollary~8.5.2]{R97:convex_analysis}. Note also that strong duality holds between~\eqref{eq:rob_value_it:reformulated} and~\eqref{eq:rob_value_it:dual} since problem~\eqref{eq:rob_value_it:reformulated} affords a Slater point by assumption (\textbf{K}).
	
	\begin{lem}\label{lem:bound_it_like_beckham}
	For any primal-dual pair $\bm{p}^\star_s \in \mathbb{R}^{AS}$ and $(\bm{\alpha}^\star, \omega^\star, \bm{\gamma}^\star, \bm{\theta}^\star) \in \mathbb{R}^A \times \mathbb{R} \times \mathbb{R}^A \times \mathbb{R}^{AS}$ satisfying the Karush-Kuhn-Tucker conditions for the problems~\eqref{eq:rob_value_it:reformulated} and~\eqref{eq:rob_value_it:dual}, we have that
	\begin{equation*}
	\omega^\star 
    \;\leq \; 
    \frac{\displaystyle \max_{a \in \mathcal{A}} \, \Vert \bm{r}_{sa} + \lambda \bm{v} \Vert_{\infty}}{\displaystyle \sum_{a \in \mathcal{A}} d_a (\bm{p}_{sa}^\star, \overline{\bm{p}}_{sa})},
	\end{equation*}
	where the right-hand side is interpreted as $+ \infty$ whenever the denominator is zero.
	\end{lem}

	\begin{proof}
	Using the notational shorthand $\bm{b}_{sa} = \bm{r}_{sa}+\lambda \bm{v}$,	the KKT conditions for~\eqref{eq:rob_value_it:reformulated} and~\eqref{eq:rob_value_it:dual} are:
	\begin{equation*}
	\begin{array}{l@{\qquad}r}
	\displaystyle \alpha_a\bm{b}_{sa}-\gamma_a\mathbf{e}-\bm{\theta}_a +\omega \nabla_{\bm{p}_{sa}}d_a(\bm{p}_{sa},\overline{\bm{p}}_{sa}) = \mathbf{0} \;\; \forall a\in\mathcal{A} & \text{(Stationarity)} \\
	\displaystyle \mathbf{e}^\top \bm{\alpha} = 1 & \text{(Stationarity)} \\
	\displaystyle \sum_{a \in \mathcal{A}} d_a (\bm{p}_{sa}, \overline{\bm{p}}_{sa}) \leq \kappa,  \;\; \bm{p}_{sa}{}^\top\bm{b}_{sa}\leq B \;\; \forall a\in\mathcal{A}, \;\; \bm{p}_s \in (\Delta_S)^A, \;\; B\in\mathbb{R} & \text{(Primal Feasibility)} \\
	\displaystyle \bm{\alpha} \in \mathbb{R}^A_+, \;\; \omega \in \mathbb{R}_+, \;\; \bm{\gamma} \in \mathbb{R}^A, \;\; \bm{\theta} \in \mathbb{R}_+^{AS} & \text{(Dual Feasibility)} \\
	\alpha_a(\bm{p}_{sa}{}^\top\bm{b}_{sa} - B) = 0 \;\; \forall a\in\mathcal{A} & \text{(Complementary Slackness)} \\
	\displaystyle\omega\left( \sum_{a \in \mathcal{A}} d_a (\bm{p}_{sa}, \overline{\bm{p}}_{sa}) - \kappa \right) = 0 & \text{(Complementary Slackness)} \\
	\theta_{as'}p_{sas'} = 0 \;\; \forall a\in\mathcal{A}, \, s'\in\mathcal{S} & \text{(Complementary Slackness)}
	\end{array}
	\end{equation*}
	Here, $B \in \mathbb{R}$ denotes the epigraphical variable used to linearize the objective function in~\eqref{eq:rob_value_it:reformulated}. The proof is split into two parts. We first show that for every $a \in \mathcal{A}$ there is $s' \in \mathcal{S}$ such that
	\begin{equation}\label{eq:inexact_subproblem:bound1}
	d_a (\bm{p}_{sa}, \overline{\bm{p}}_{sa})
	\; \leq \;
	\bm{p}_{sa}{}^\top \nabla_{\bm{p}_{sa}} d_a(\bm{p}_{sa},\overline{\bm{p}}_{sa}) - [\nabla_{\bm{p}_{sa}} d_a(\bm{p}_{sa},\overline{\bm{p}}_{sa})]_{s'}.
	\end{equation}
	We next prove that for all $s' \in \mathcal{S}$ and $a \in \mathcal{A}$, we have
	\begin{equation}\label{eq:inexact_subproblem:bound2}
	\omega \left( \bm{p}_{sa}{}^\top \nabla_{\bm{p}_{sa}} d_a(\bm{p}_{sa},\overline{\bm{p}}_{sa}) - [\nabla_{\bm{p}_{sa}} d_a(\bm{p}_{sa},\overline{\bm{p}}_{sa})]_{s'} \right)
	\; \leq \;
	\alpha_a b_{sas'}.
	\end{equation}
	Since $\omega \in \mathbb{R}_+$ by the dual feasibility condition,~\eqref{eq:inexact_subproblem:bound1} and~\eqref{eq:inexact_subproblem:bound2} imply that for every $a \in \mathcal{A}$ there is $s' \in \mathcal{S}$ such that $\omega d_a (\bm{p}_{sa}, \overline{\bm{p}}_{sa}) \leq \alpha_a b_{sas'}$. From this we obtain that
	\begin{equation*}
	\omega \sum_{a \in \mathcal{A}} d_a (\bm{p}_{sa}, \overline{\bm{p}}_{sa})
	\; \leq \;
	\sum_{a \in \mathcal{A}} \alpha_a \max_{s' \in \mathcal{S}} \left\{ b_{sas'} \right\}
	\; \leq \;
	\max_{a \in \mathcal{A}, s' \in \mathcal{S}} \left\{ b_{sas'} \right\},
	\end{equation*}
	where the last inequality holds since $\mathbf{e}^\top \bm{\alpha} = 1$ by the second stationarity condition. This proves the statement of the lemma.
	
	To show~\eqref{eq:inexact_subproblem:bound1}, we note that
	\begin{equation*}
	d_a (\bm{p}_{sa}, \overline{\bm{p}}_{sa}) + \nabla_{\bm{p}_{sa}} d_a(\bm{p}_{sa},\overline{\bm{p}}_{sa})^\top (\overline{\bm{p}}_{sa} - \bm{p}_{sa})
	\; \leq \;
	d_a (\overline{\bm{p}}_{sa}, \overline{\bm{p}}_{sa})
	\; = \;
	0
	\end{equation*}
	since $d_a$ is convex by assumption (\textbf{C}) and $d_a (\overline{\bm{p}}_{sa}, \overline{\bm{p}}_{sa}) = 0$ by assumption (\textbf{D}). We thus have
	\begin{equation*}
	d_a (\bm{p}_{sa}, \overline{\bm{p}}_{sa})
	\; \leq \;
	\bm{p}_{sa}{}^\top \nabla_{\bm{p}_{sa}} d_a(\bm{p}_{sa},\overline{\bm{p}}_{sa}) - \overline{\bm{p}}_{sa}{}^\top \nabla_{\bm{p}_{sa}} d_a(\bm{p}_{sa},\overline{\bm{p}}_{sa}),
	\end{equation*}
	and the fact that $\overline{\bm{p}}_{sa} \in \Delta_S$ implies that
	\begin{equation*}
	d_a (\bm{p}_{sa}, \overline{\bm{p}}_{sa})
	\; \leq \;
	\bm{p}_{sa}{}^\top \nabla_{\bm{p}_{sa}} d_a(\bm{p}_{sa},\overline{\bm{p}}_{sa}) - \min_{s' \in \mathcal{S}} \; \left[ \nabla_{\bm{p}_{sa}} d_a(\bm{p}_{sa},\overline{\bm{p}}_{sa}) \right]_{s'},
	\end{equation*}
	which is equivalent to~\eqref{eq:inexact_subproblem:bound1}. 
	
	We now prove~\eqref{eq:inexact_subproblem:bound2}. Aggregating the equations in the first stationarity condition according to the weights $\bm{p}_{sa} \in \Delta_S$ shows that for all $a \in \mathcal{A}$, we have
	\begin{align}
	& \alpha_a \bm{p}_{sa}{}^\top \bm{b}_{sa} - \gamma_a \mathbf{e}^\top \bm{p}_{sa} - \bm{p}_{sa}{}^\top \bm{\theta}_a + \omega \bm{p}_{sa}{}^\top \nabla_{\bm{p}_{sa}}d_a(\bm{p}_{sa},\overline{\bm{p}}_{sa}) = 0 \nonumber \\
	\Longleftrightarrow \quad
	& \gamma_a  = \alpha_a \bm{p}_{sa}{}^\top \bm{b}_{sa} + \omega \bm{p}_{sa}{}^\top \nabla_{\bm{p}_{sa}}d_a(\bm{p}_{sa},\overline{\bm{p}}_{sa}) \label{eq:inexact_subproblem:bound2:aux1}
	\end{align}
	since the primal feasibility condition guarantees that $\mathbf{e}^\top \bm{p}_{sa} = 1$ and the last complementary slackness condition ensures that $\bm{p}_{sa}{}^\top \bm{\theta}_a = 0$. However, the first stationarity condition also implies
	\begin{align}
	& \alpha_a b_{sas'} - \gamma_a - \theta_{as'} + \omega \left[ \nabla_{\bm{p}_{sa}} d_a(\bm{p}_{sa},\overline{\bm{p}}_{sa}) \right]_{s'} = 0 \quad \forall a \in \mathcal{A}, \, s' \in \mathcal{S} \nonumber \\
	\Longleftrightarrow \quad
	& \gamma_a \leq \alpha_a b_{sas'} + \omega \left[ \nabla_{\bm{p}_{sa}} d_a(\bm{p}_{sa},\overline{\bm{p}}_{sa}) \right]_{s'} \mspace{95mu} \forall a \in \mathcal{A}, \, s' \in \mathcal{S} \label{eq:inexact_subproblem:bound2:aux2}
	\end{align}
	since $\theta_{as'} \geq 0$ due to the dual feasibility condition. Combining~\eqref{eq:inexact_subproblem:bound2:aux1} and~\eqref{eq:inexact_subproblem:bound2:aux2}, finally, yields
	\begin{align*}
	& \alpha_a \bm{p}_{sa}{}^\top \bm{b}_{sa} + \omega \bm{p}_{sa}{}^\top \nabla_{\bm{p}_{sa}}d_a(\bm{p}_{sa},\overline{\bm{p}}_{sa})
	\; \leq \;
	\alpha_a b_{sas'} + \omega \left[ \nabla_{\bm{p}_{sa}} d_a(\bm{p}_{sa},\overline{\bm{p}}_{sa}) \right]_{s'} \mspace{25mu} \forall a \in \mathcal{A}, s' \in \mathcal{S} \\
	\Longleftrightarrow \quad
	& \omega \left( \bm{p}_{sa}{}^\top \nabla_{\bm{p}_{sa}}d_a(\bm{p}_{sa},\overline{\bm{p}}_{sa}) - \left[ \nabla_{\bm{p}_{sa}} d_a(\bm{p}_{sa},\overline{\bm{p}}_{sa}) \right]_{s'} \right)
	\; \leq \; \alpha_a b_{sas'} - \alpha_a \bm{p}_{sa}{}^\top \bm{b}_{sa} \quad \forall a \in \mathcal{A}, s' \in \mathcal{S},
	\end{align*}
	which implies~\eqref{eq:inexact_subproblem:bound2} since $\alpha_a \bm{p}_{sa}{}^\top \bm{b}_{sa} \geq 0$ as $\alpha_a \geq 0$ by the dual feasibility condition, $\bm{p}_{sa} \geq \bm{0}$ by the primal feasibility condition and $\bm{b}_{sa} \geq \bm{0}$ by assumption.
	\end{proof}

	\begin{lem}\label{lem:the_grand_unification_of_everything}
	Let $\mathfrak{B} (\bm{v}; \kappa)$ be the robust Bellman iterate~\eqref{eq:rob_value_it} with the budget $\kappa > 0$ in the ambiguity set $\mathcal{P}$. For any $\kappa' \geq \kappa$ and any primal-dual pair $\bm{p}^\star_s \in \mathbb{R}^{AS}$ and $(\bm{\alpha}^\star, \omega^\star, \bm{\gamma}^\star, \bm{\theta}^\star) \in \mathbb{R}^A \times \mathbb{R} \times \mathbb{R}^A \times \mathbb{R}^{AS}$ satisfying the Karush-Kuhn-Tucker conditions for~\eqref{eq:rob_value_it:reformulated} and~\eqref{eq:rob_value_it:dual} with budget $\kappa$, we have 
	\begin{equation*}
	\left \lVert \mathfrak{B} (\bm{v}; \kappa) - \mathfrak{B} (\bm{v}; \kappa') \right \rVert_\infty \leq \frac{\displaystyle (\kappa' - \kappa) \max_{a \in \mathcal{A}} \, \Vert \bm{r}_{sa} + \lambda \bm{v} \Vert_{\infty}}{\displaystyle  \sum_{a \in \mathcal{A}} d_a (\bm{p}_{sa}^\star, \overline{\bm{p}}_{sa})},
	\end{equation*}
	where the right-hand side is interpreted as $+ \infty$ whenever the denominator is zero.
	\end{lem}

    Intuitively, Lemma~\ref{lem:the_grand_unification_of_everything} states that for a fixed budget $\kappa$, the worst-case expected total reward lost by loosening the budget to $\kappa'$ can be bounded from above by a function that is linear in the additional budget $\kappa' - \kappa$.

	\begin{proof}[Proof of Lemma~\ref{lem:the_grand_unification_of_everything}]
	Since $\kappa' \geq \kappa$, we have for fixed $s \in \mathcal{S}$ that
	\begin{equation*}
	\left \lvert [ \mathfrak{B} (\bm{v}; \kappa) ]_s - [ \mathfrak{B} (\bm{v}; \kappa') ]_s \right \rvert
	\; = \;
	[ \mathfrak{B} (\bm{v}; \kappa) ]_s - [ \mathfrak{B} (\bm{v}; \kappa') ]_s
	\; \leq \;
	\omega^\star (\kappa' - \kappa),
	\end{equation*}
	where $\omega^\star$ belongs to any primal-dual pair $\bm{p}_s^\star \in \mathbb{R}^{AS}$ and $(\bm{\alpha}^\star, \omega^\star, \bm{\gamma}, \bm{\theta}^\star) \in \mathbb{R}^A \times \mathbb{R} \times \mathbb{R}^A \times \mathbb{R}^{AS}$ satisfying the KKT conditions of problems~\eqref{eq:rob_value_it:reformulated} and~\eqref{eq:rob_value_it:dual}. Indeed, since only the first term in the objective function in~\eqref{eq:rob_value_it:dual} depends on $\kappa$, the solution $(\bm{\alpha}^\star, \omega^\star, \bm{\gamma}, \bm{\theta}^\star)$ for the dual problem with budget $\kappa$ remains feasible (but is typically not optimal) for the dual problem with budget $\kappa'$, and its objective value decreases by precisely $\omega^\star (\kappa' - \kappa)$. The result now follows from Lemma~\ref{lem:bound_it_like_beckham}.
\end{proof}

\noindent \textbf{Proof of Theorem~\ref{thm:overall_complexity:inexact_subproblem}.} $\;$
	As in the proof of Theorem~\ref{thm:overall_complexity:exact_subproblem}, we consider the equivalent reformulation~\eqref{eq:rob_value_it:reformulated} of problem~\eqref{eq:rob_value_it} and compute each component $v'_s$ of $\bm{v}'$, $s \in \mathcal{S}$, individually. We compute $v'_s$ through a bisection on the optimal value of problem~\eqref{eq:rob_value_it:reformulated}. To this end, we set $\delta = \epsilon \kappa / [2A\overline{R} + A \epsilon]$. We start the bisection with the lower and upper bounds $\underline{v}_{s0} = \underline{R}_s (\bm{v})$ and $\overline{v}_{s0} = \overline{R}$, respectively. Note that the lower bound $\underline{R}_s (\bm{v}) = \max_{a \in \mathcal{A}} \; \min_{s' \in \mathcal{S}} \, \{ r_{sas'} + \lambda v_{s'} \}$ is justified since the projection subproblem~\eqref{eq:gen_projection} is infeasible if $\theta = \beta < \min \{ \bm{b} \} $ which is set to be $\min \{ \bm{r}_{sa} + \lambda \bm{v} \} $ for each $a\in\mathcal{A}$. In each iteration $i = 0, 1, \ldots$, we consider the midpoint $\theta = (\underline{v}_{si} + \overline{v}_{si}) / 2$ and compute the generalized $d_a$-projections $\mathfrak{P} (\overline{\bm{p}}_{sa}; \bm{r}_{sa} + \lambda \bm{v}, \theta)$, $a \in \mathcal{A}$, to $\delta$-accuracy, resulting in the action-wise lower and upper bounds $(\underline{d}_a, \overline{d}_a)$, respectively. We then update the interval bounds as follows:
	\begin{equation*}
	\begin{cases}
	(\underline{v}^{i+1}_s, \, \overline{v}^{i+1}_s) \leftarrow (\underline{v}_{si}, \, \theta) & \displaystyle \text{if } \sum_{a \in \mathcal{A}} \overline{d}_a \leq \kappa, \\
	(\underline{v}^{i+1}_s, \, \overline{v}^{i+1}_s) \leftarrow (\theta, \, \overline{v}_{si}) & \displaystyle \text{if } \sum_{a \in \mathcal{A}} \underline{d}_a > \kappa
	\end{cases}
	\end{equation*}
	We terminate the bisection once \emph{(i)} $\overline{v}_{si} - \underline{v}_{si} \leq \epsilon$ or \emph{(ii)} $\kappa \in \big[ \sum_{a \in \mathcal{A}} \underline{d}_a, \, \sum_{a \in \mathcal{A}} \overline{d}_a \big)$, whichever condition holds first. Note that both interval updates ensure that $\underline{v}^{i+1}_s$ and $\overline{v}^{i+1}_s$ remain valid bounds since
	\begin{equation*}
	\sum_{a \in \mathcal{A}} \overline{d}_a \leq \kappa
	\quad \Longrightarrow \quad
	\sum_{a \in \mathcal{A}} \mathfrak{P} (\overline{\bm{p}}_{sa}; \bm{r}_{sa} + \lambda \bm{v}, \theta) \leq \kappa
	\quad \Longrightarrow \quad
	[\mathfrak{B} (\bm{v})]_s \leq \theta
	\end{equation*}
	as well as
	\begin{equation*}
	\sum_{a \in \mathcal{A}} \underline{d}_a > \kappa
	\quad \Longrightarrow \quad
	\sum_{a \in \mathcal{A}} \mathfrak{P} (\overline{\bm{p}}_{sa}; \bm{r}_{sa} + \lambda \bm{v}, \theta) > \kappa
	\quad \Longrightarrow \quad
	[\mathfrak{B} (\bm{v})]_s > \theta,
	\end{equation*}
	where the respective second implications follow from the proof of Theorem~\ref{thm:overall_complexity:exact_subproblem}. At termination, in case \emph{(i)} we have $\overline{v}_{si} - \underline{v}_{si} \leq \epsilon$, which implies that $\theta = (\underline{v}_{si} + \overline{v}_{si}) / 2$ is an $\epsilon$-optimal solution to $[\mathfrak{B} (\bm{v})]_s$. If case \emph{(ii)} is satisfied at termination, on the other hand, then
	\begin{equation*}
	\sum_{a \in \mathcal{A}} \underline{d}_a \leq \kappa < \sum_{a \in \mathcal{A}} \overline{d}_a
	\quad \Longrightarrow \quad
	\sum_{a \in \mathcal{A}} \mathfrak{P} (\overline{\bm{p}}_{sa}; \bm{r}_{sa} + \lambda \bm{v}, \theta) - A \delta
	\leq \kappa <
	\sum_{a \in \mathcal{A}} \mathfrak{P} (\overline{\bm{p}}_{sa}; \bm{r}_{sa} + \lambda \bm{v}, \theta) + A \delta,
	\end{equation*}
	in which case $\theta = (\underline{v}_{si} + \overline{v}_{si}) / 2$ is an \emph{exact} optimal solution to the variant $[\mathfrak{B} (\bm{v}; \kappa')]_s$ of the robust value iteration~\eqref{eq:rob_value_it} where the budget $\kappa$ in the ambiguity set is replaced with some $\kappa' \in [\kappa - A \delta, \, \kappa + A \delta]$. In this case, we have that
	\begin{equation*}
	\big| \theta - [\mathfrak{B} (\bm{v})]_s \big|
	\;\; \leq \;\;
	[\mathfrak{B} (\bm{v}; \kappa - A \delta)]_s - [\mathfrak{B} (\bm{v}; \kappa + A \delta)]_s
	\;\; \leq \;\;
	\epsilon,
	\end{equation*}
	where the first inequality follows from the monotonicity of $\mathfrak{B} (\bm{v}; \cdot)$ in its second argument, and the second inequality holds because of the following argument. If the constraint $\sum_{a \in \mathcal{A}} d_a (\bm{p}_{sa}, \overline{\bm{p}}_{sa}) \leq \kappa - A \delta$ in problem~\eqref{eq:rob_value_it:reformulated} is not binding at optimality, then $[\mathfrak{B} (\bm{v}; \kappa - A \delta)]_s - [\mathfrak{B} (\bm{v}; \kappa + A \delta)]_s = 0 < \epsilon$. On the other hand, if the constraint $\sum_{a \in \mathcal{A}} d_a (\bm{p}_{sa}, \overline{\bm{p}}_{sa}) \leq \kappa - A \delta$ in problem~\eqref{eq:rob_value_it:reformulated} is binding at optimality, then by applying Lemma~\ref{lem:the_grand_unification_of_everything} in the appendix and using our definition of $\delta$ and the fact that $\Vert \bm{r}_{sa} + \lambda \bm{v} \Vert_{\infty} \leq \overline{R}$, we have
	\begin{equation*}
	[\mathfrak{B} (\bm{v}; \kappa - A \delta)]_s - [\mathfrak{B} (\bm{v}; \kappa + A \delta)]_s
	\leq 
	\frac{\displaystyle 2A\delta \max_{a \in \mathcal{A}} \, \Vert \bm{r}_{sa} + \lambda \bm{v} \Vert_{\infty}}{\kappa - A \delta} \leq \epsilon .
	\end{equation*}

	A similar reasoning as in the proof of Theorem~\ref{thm:overall_complexity:exact_subproblem} shows that at most $\mathcal{O} (\log [\overline{R} / \epsilon])$ iterations of complexity  $\mathcal{O} (A \cdot h (S, \delta))$ are executed in each of the $S$ bisections, which concludes the proof.
\qed

We next prove how Assumption~\ref{assm:2norm:simple_algo} simplifies the design and analysis of Algorithm~\ref{alg:2norm:gamma_form}.

\begin{lem}[Consequences of Assumption~\ref{assm:2norm:simple_algo}]\label{lem:consequences_of_assm_2norm}
    Consider iteration $t$ of Algorithm~\ref{alg:2norm:gamma_form}. Assumption~\ref{assm:2norm:simple_algo}~\emph{(ii)} implies that there is no $s \in \mathcal{S}$ such that $-b_s \alpha + \gamma_t (\alpha) + c_s = 0$ on an interval of positive width. Moreover, Assumption~\ref{assm:2norm:simple_algo}~\emph{(iii)} ensures that there are no $s, s' \in \mathcal{S}$, $s \neq s'$, such that $-b_s \alpha + \gamma_t (\alpha) + c_s = 0$ and $-b_{s'} \alpha + \gamma_t (\alpha) + c_{s'} = 0$ simultaneously at any $\alpha \in \mathbb{R}$.
\end{lem}

\begin{proof}
    In view of the first statement, assume to the contrary that there is an iteration $k \in \mathbb{N}$, bounds $\underline{\alpha}, \overline{\alpha} \in \mathbb{R}$ with $\overline{\alpha} > \underline{\alpha}$, and a state $s \in \mathcal{S}$ such that
    \begin{align*}
        -b_s \alpha + \gamma_k (\alpha) + c_s \;
        = \; 0
        \quad &\Longleftrightarrow \quad
        -b_s \alpha + \frac{\rho + \sum_{s' \in \mathcal{I}_k} a_{s'} (b_{s'} \alpha - c_{s'})}{\sum_{s' \in \mathcal{I}_k} a_{s'}} + c_s \;
        = \; 0 \\
        &\Longleftrightarrow \quad
        \rho + \sum_{s' \in \mathcal{I}_k} a_{s'} ([b_{s'} - b_s] \alpha - c_{s'} + c_s) \;
        = \; 0
    \end{align*}
    for all $\alpha \in [\underline{\alpha}, \overline{\alpha}]$.
    For this affine expression to vanish on an interval of positive width, its constant term must be zero. However, Assumption~\ref{assm:2norm:simple_algo}~\emph{(ii)} implies that $\rho + \sum_{s' \in \mathcal{I}_k} a_{s'} (c_s - c_{s'}) \neq 0$, which is a contradiction.

    As for the second statement, assume to the contrary that there is an iteration $k \in \mathbb{N}$, a value $\alpha \in \mathbb{R}$ and states $s, t \in \mathcal{S}$, $s \neq t$, such that
    \begin{align*}
        -b_s \alpha + \gamma_k (\alpha) + c_s
        \;
        = 0 \; = \;
        -b_t \alpha + \gamma_k (\alpha) + c_t
        \quad &\Longleftrightarrow \quad
        \gamma_k (\alpha) \;
        = \; b_s \alpha - c_s \; = \; b_t \alpha - c_t \\
        &\Longleftrightarrow \quad
        \gamma_k (\alpha_{s,t}) \;
        = \; b_s \alpha_{s,t} - c_s \; = \; b_t \alpha_{s,t} - c_t,
    \end{align*}
    where the second equivalence holds by construction of $\alpha_{s,t}$.
    The first equality on the right-hand side of the last equivalence implies that
    \begin{align*}
        \mspace{-30mu}
        \frac{\rho + \sum_{s' \in \mathcal{I}_k} a_{s'} (b_{s'} \alpha_{s,t} - c_{s'})}{\sum_{s' \in \mathcal{I}_k} a_{s'}} \;
        = \; b_s \alpha_{s,t} - c_s
        \quad &\Longleftrightarrow \quad
        \rho + \sum_{s' \in \mathcal{I}_k} a_{s'} ( (b_{s'}-b_s ) \alpha_{s,t} - c_{s'} + c_s) \;
        = \; 0 \\
        &\Longleftrightarrow \quad
        \rho + \sum_{s' \in \mathcal{I}_k} a_{s'} ( (b_{s'}-b_s ) \alpha_{s,t} + (b_s - b_{s'}) \alpha_{s,s'} ) \;
        = \; 0 \\
        &\Longleftrightarrow \quad
        \rho + \sum_{s' \in \mathcal{I}_k} a_{s'} (b_{s'}-b_s ) ( \alpha_{s,t} - \alpha_{s,s'} ) \;
        = \; 0,
    \end{align*}
    where the second equivalence follows from the definition of $\alpha_{s,s'}$.
    Since the last equality violates Assumption~\ref{assm:2norm:simple_algo}~\emph{(iii)}, we obtain the desired contradiction.
\end{proof}

\section*{Appendix B: Extended Numerical Results}

We report extended numerical results that complement the figures and summary tables from Section~\ref{sec:numericals}. In particular, we provide detailed runtimes for the projection subproblems and the robust Bellman operator.

\begin{table}[t]
\centering
\caption{$\ell_1$- and $\ell_2$-norm projection runtimes (in $\mu$s) on synthetic and benchmark instances. Missing entries correspond to instances where CPLEX encountered numerical issues.}
~\hspace{-2.8cm}
\begin{tabular}{l|rrrr|rrrr}
Instance & L1 Fast & L1 CPLEX & L1 Gurobi & L1 Mosek & L2 Fast & L2 CPLEX & L2 Gurobi & L2 Mosek \\
\hline
synthetic ($S=10$, $A=10$) & 0.43 & 277.59 & 193.15 & 638.59 & 0.54 & 241.63 & 129.99 & 660.29 \\
synthetic ($S=20$, $A=10$) & 0.81 & 319.16 & 271.34 & 739.09 & 1.20 & 261.11 & 161.04 & 691.62 \\
synthetic ($S=30$, $A=10$) & 1.13 & 370.52 & 343.64 & 860.49 & 1.82 & 294.28 & 196.08 & 746.80 \\
synthetic ($S=40$, $A=10$) & 1.29 & 457.92 & 406.99 & 980.02 & 2.71 & 313.52 & 227.15 & 805.21 \\
synthetic ($S=50$, $A=10$) & 1.62 & 481.39 & 507.00 & 1,142.74 & 3.87 & 342.27 & 254.56 & 873.96 \\
synthetic ($S=60$, $A=10$) & 2.11 & 549.95 & 619.12 & 1,243.26 & 4.59 & 361.08 & 291.93 & 934.62 \\
synthetic ($S=70$, $A=10$) & 2.25 & 615.89 & 658.89 & 1,344.07 & 5.86 & 391.20 & 310.00 & 994.19 \\
synthetic ($S=80$, $A=10$) & 2.56 & 680.18 & 760.77 & 1,466.68 & 7.28 & 393.38 & 337.44 & 1,059.37 \\
synthetic ($S=90$, $A=10$) & 3.04 & 731.95 & 882.69 & 1,576.70 & 9.03 & 466.13 & 368.01 & 1,168.46 \\
synthetic ($S=100$, $A=10$) & 3.38 & 818.88 & 912.20 & 1,764.80 & 10.61 & 436.03 & 386.86 & 1,182.20 \\
\hline
synthetic ($S=20$, $A=20$) & 0.81 & 317.89 & 268.33 & 734.18 & 1.17 & 261.40 & 165.46 & 691.39 \\
synthetic ($S=30$, $A=30$) & 1.08 & 380.07 & 336.40 & 860.16 & 1.85 & 292.61 & 217.60 & 741.01 \\
synthetic ($S=40$, $A=40$) & 1.26 & 435.89 & 402.55 & 992.04 & 2.58 & 301.91 & 224.27 & 805.96 \\
synthetic ($S=50$, $A=50$) & 1.51 & 479.39 & 606.59 & 1,118.60 & 3.52 & 329.61 & 272.13 & 873.07 \\
synthetic ($S=60$, $A=60$) & 1.94 & 558.86 & 627.20 & 1,241.78 & 4.55 & 352.32 & 303.07 & 933.11 \\
synthetic ($S=70$, $A=70$) & 2.14 & 602.65 & 685.44 & 1,345.93 & 6.04 & 389.05 & 341.21 & 993.16 \\
synthetic ($S=80$, $A=80$) & 2.49 & 712.91 & 752.92 & 1,459.28 & 7.22 & 394.72 & 336.62 & 1,041.12 \\
synthetic ($S=90$, $A=90$) & 3.07 & 726.02 & 821.62 & 1,573.33 & 8.94 & 433.39 & 397.96 & 1,123.05 \\
synthetic ($S=100$, $A=100$) & 3.30 & 816.77 & 907.96 & 1,717.12 & 10.59 & 435.26 & 386.02 & 1,224.11 \\
\hline
blackjack & 17.76 & -- & 2,084.10 & 2,303.33 & 25.31 & 12,065.50 & 1,140.58 & 2,949.52 \\
capacity50 & 1.66 & 1,929.49 & 423.12 & 917.25 & 2.78 & 1,180.03 & 254.67 & 926.80 \\
chain10 & 0.45 & 425.79 & 189.64 & 634.76 & 0.50 & 362.13 & 145.81 & 661.86 \\
cliffwalking & 0.14 & 994.67 & 316.24 & 429.76 & 0.27 & 960.56 & 239.59 & 680.07 \\
forest50 & 1.67 & 933.68 & 415.18 & 778.30 & 2.46 & 752.32 & 245.11 & 949.91 \\
frozenlake4x4 & 0.57 & 487.42 & 226.56 & 664.48 & 1.09 & 404.72 & 152.72 & 732.47 \\
frozenlake8x8 & 2.25 & 1,880.41 & 510.34 & 896.61 & 4.03 & 1,420.82 & 286.13 & 1,243.89 \\
gridworld25 & 0.89 & 586.27 & 260.46 & 624.14 & 1.20 & 559.33 & 172.06 & 759.22 \\
inventory50 & 0.13 & 1,502.19 & 406.90 & 755.11 & 0.29 & 1,004.22 & 253.52 & 742.31 \\
machine20 & 0.09 & 363.89 & 179.83 & 351.85 & 0.21 & 418.45 & 164.25 & 616.27 \\
openspiel$\underline{~}$grid16 & 0.54 & 498.32 & 208.19 & 586.60 & 1.02 & 408.15 & 146.03 & 742.79 \\
perishable50 & 0.13 & 1,458.00 & 404.76 & 755.58 & 0.29 & 997.68 & 257.23 & 736.45 \\
riverswim20 & 0.72 & 491.31 & 257.31 & 1,015.01 & 1.29 & 409.26 & 163.66 & 678.85 \\
riverswim6 & 0.33 & 348.92 & 165.53 & 729.57 & 0.54 & 319.25 & 133.06 & 615.11 \\
taxi & 32.05 & -- & 2,997.11 & 3,157.03 & 34.49 & -- & 1,441.25 & 3,240.99 \\ \hline \hline
\end{tabular}
\end{table}

\begin{table}[t]
\centering
\caption{KL divergence and Burg entropy projection runtimes (in $\mu$s) on synthetic and benchmark instances.}
\begin{tabular}{l|rr|rr}
Instance & KL Fast & KL Mosek & Burg Fast & Burg Mosek \\
\hline
synthetic ($S=10$, $A=10$) & 5.61 & 1,091.76 & 1.32 & 968.72 \\
synthetic ($S=20$, $A=10$) & 10.71 & 1,485.03 & 2.20 & 1,459.05 \\
synthetic ($S=30$, $A=10$) & 15.85 & 1,942.33 & 3.10 & 1,946.30 \\
synthetic ($S=40$, $A=10$) & 21.14 & 2,448.60 & 4.29 & 2,458.39 \\
synthetic ($S=50$, $A=10$) & 25.76 & 2,835.97 & 5.19 & 2,957.32 \\
synthetic ($S=60$, $A=10$) & 30.77 & 3,289.00 & 5.85 & 3,456.49 \\
synthetic ($S=70$, $A=10$) & 35.93 & 3,849.30 & 6.76 & 4,019.76 \\
synthetic ($S=80$, $A=10$) & 41.11 & 4,246.58 & 7.67 & 4,589.93 \\
synthetic ($S=90$, $A=10$) & 46.00 & 4,844.35 & 8.58 & 5,014.32 \\
synthetic ($S=100$, $A=10$) & 51.20 & 5,194.77 & 9.57 & 5,564.42 \\
\hline
synthetic ($S=20$, $A=20$) & 10.62 & 1,485.60 & 2.20 & 1,474.46 \\
synthetic ($S=30$, $A=30$) & 15.93 & 1,996.88 & 3.21 & 1,937.97 \\
synthetic ($S=40$, $A=40$) & 20.67 & 2,379.92 & 4.20 & 2,458.76 \\
synthetic ($S=50$, $A=50$) & 25.77 & 2,827.25 & 4.94 & 2,946.57 \\
synthetic ($S=60$, $A=60$) & 30.78 & 3,397.15 & 5.85 & 3,458.24 \\
synthetic ($S=70$, $A=70$) & 36.02 & 3,878.38 & 6.75 & 4,014.76 \\
synthetic ($S=80$, $A=80$) & 40.95 & 4,219.85 & 7.66 & 4,624.34 \\
synthetic ($S=90$, $A=90$) & 46.01 & 4,884.22 & 8.58 & 5,022.81 \\
synthetic ($S=100$, $A=100$) & 51.29 & 5,210.92 & 9.60 & 5,713.97 \\
\hline
blackjack & 212.66 & 18,469.40 & 34.15 & 20,674.80 \\
capacity50 & 28.86 & 2,877.01 & 5.15 & 2,890.49 \\
chain10 & 7.66 & 1,086.38 & 1.70 & 969.93 \\
cliffwalking & 1.04 & 2,419.60 & 0.88 & 2,593.19 \\
forest50 & 29.52 & 2,856.43 & 5.16 & 2,794.84 \\
frozenlake4x4 & 10.35 & 1,337.28 & 2.01 & 1,218.13 \\
frozenlake8x8 & 34.05 & 3,551.01 & 6.65 & 3,510.66 \\
gridworld25 & 15.02 & 1,792.84 & 2.65 & 1,620.61 \\
inventory50 & 1.12 & 2,420.09 & 1.09 & 2,490.72 \\
machine20 & 0.42 & 1,417.57 & 0.41 & 1,393.86 \\
openspiel$\underline{~}$grid16 & 10.63 & 1,350.78 & 1.95 & 1,188.22 \\
perishable50 & 0.84 & 2,436.55 & 0.84 & 2,500.77 \\
riverswim20 & 11.57 & 1,513.22 & 2.26 & 1,431.42 \\
riverswim6 & 4.78 & 873.35 & 1.51 & 821.30 \\
taxi & 301.23 & 32,121.90 & 46.41 & 25,453.50 \\ \hline \hline
\end{tabular}
\end{table}

\begin{table}[t]
\caption{$\ell_1$- and $\ell_2$-norm robust Bellman operator runtimes (in ms) on synthetic and benchmark instances.}
~\hspace{-2.5cm}\footnotesize
\begin{tabular}{l|rrrrr|rrrr}
Instance & L1 Fast & L1 CPLEX & L1 Gurobi & L1 Mosek & L1 Homotopy & L2 Fast & L2 CPLEX & L2 Gurobi & L2 Mosek \\
\hline
synthetic ($S=10$, $A=10$) & 0.05 & 0.97 & 1.10 & 3.68 & 0.01 & 0.08 & 4.16 & 2.98 & 2.66 \\
synthetic ($S=20$, $A=10$) & 0.06 & 2.30 & 2.01 & 6.60 & 0.05 & 0.12 & 9.66 & 5.31 & 3.55 \\
synthetic ($S=30$, $A=10$) & 0.09 & 3.57 & 3.15 & 9.63 & 0.17 & 0.20 & 14.65 & 6.30 & 4.34 \\
synthetic ($S=40$, $A=10$) & 0.11 & 5.29 & 4.23 & 12.49 & 0.31 & 0.31 & 17.99 & 8.19 & 5.18 \\
synthetic ($S=50$, $A=10$) & 0.13 & 6.33 & 5.49 & 15.97 & 0.57 & 0.48 & 23.14 & 10.41 & 6.28 \\
synthetic ($S=60$, $A=10$) & 0.16 & 8.42 & 7.58 & 18.85 & 0.97 & 0.75 & 25.29 & 11.50 & 7.22 \\
synthetic ($S=70$, $A=10$) & 0.19 & 10.01 & 8.11 & 22.29 & 1.51 & 1.04 & 30.04 & 12.89 & 7.83 \\
synthetic ($S=80$, $A=10$) & 0.22 & 12.94 & 9.15 & 25.68 & 2.26 & 1.45 & 34.31 & 13.62 & 8.51 \\
synthetic ($S=90$, $A=10$) & 0.26 & 15.05 & 10.82 & 30.59 & 3.17 & 1.98 & 39.48 & 15.64 & 9.53 \\
synthetic ($S=100$, $A=10$) & 0.31 & 17.30 & 12.12 & 32.79 & 4.20 & 2.60 & 44.01 & 16.97 & 10.24 \\
\hline
synthetic ($S=20$, $A=20$) & 0.11 & 4.47 & 3.99 & 13.49 & 0.12 & 0.20 & 17.94 & 11.07 & 6.35 \\
synthetic ($S=30$, $A=30$) & 0.19 & 12.74 & 9.33 & 30.35 & 0.46 & 0.43 & 73.05 & 22.03 & 11.97 \\
synthetic ($S=40$, $A=40$) & 0.28 & 28.85 & 18.60 & 51.60 & 1.24 & 0.75 & 157.17 & 39.69 & 20.15 \\
synthetic ($S=50$, $A=50$) & 0.39 & 46.57 & 32.01 & 85.65 & 2.86 & 1.35 & 346.00 & 61.45 & 33.08 \\
synthetic ($S=60$, $A=60$) & 0.57 & 177.89 & 54.92 & 128.22 & 5.78 & 2.10 & 554.11 & 85.63 & 48.94 \\
synthetic ($S=70$, $A=70$) & 0.77 & 328.20 & 95.22 & 179.85 & 10.85 & 3.59 & 863.58 & 119.87 & 65.95 \\
synthetic ($S=80$, $A=80$) & 0.95 & 112.32 & 153.34 & 235.74 & 18.10 & 5.00 & 1,217.77 & 140.16 & 83.40 \\
synthetic ($S=90$, $A=90$) & 1.24 & 180.93 & 248.09 & 307.14 & 28.72 & 7.64 & 1,803.32 & 186.48 & 108.17 \\
synthetic ($S=100$, $A=100$) & 1.58 & 251.41 & 439.89 & 400.82 & 42.00 & 10.14 & 2,538.84 & 240.46 & 132.77 \\
\hline
blackjack & 0.45 & 7.37 & 4.94 & 17.15 & 0.20 & 1.64 & 23.08 & 10.64 & 7.78 \\
capacity50 & 0.12 & 3.15 & 2.49 & 7.06 & 0.06 & 0.37 & 8.80 & 4.74 & 3.42 \\
chain10 & 0.02 & 0.41 & 0.34 & 1.07 & 0.00 & 0.03 & 1.22 & 0.51 & 1.14 \\
cliffwalking & 0.02 & 1.05 & 1.30 & 4.28 & 0.00 & 0.03 & 3.13 & 1.54 & 1.86 \\
forest50 & 0.05 & 0.75 & 0.87 & 2.75 & 0.00 & 0.07 & 3.80 & 1.77 & 1.87 \\
frozenlake4x4 & 0.04 & 0.76 & 0.67 & 2.22 & 0.01 & 0.06 & 2.78 & 1.21 & 1.83 \\
frozenlake8x8 & 0.12 & 2.77 & 1.82 & 6.46 & 0.02 & 0.20 & 9.46 & 3.97 & 3.70 \\
gridworld25 & 0.04 & 0.71 & 0.81 & 2.76 & 0.00 & 0.05 & 4.04 & 2.08 & 2.08 \\
inventory50 & 0.04 & 1.41 & 1.93 & 6.31 & 0.05 & 0.06 & 6.74 & 3.79 & 3.43 \\
machine20 & 0.01 & 0.43 & 0.38 & 1.30 & 0.00 & 0.01 & 1.13 & 0.47 & 1.15 \\
openspiel$\underline{~}$grid16 & 0.05 & 0.62 & 0.73 & 2.80 & 0.00 & 0.08 & 4.45 & 2.18 & 1.95 \\
perishable50 & 0.04 & 1.43 & 1.92 & 5.97 & 0.04 & 0.06 & 7.00 & 4.02 & 3.10 \\
riverswim20 & 0.03 & 0.46 & 0.46 & 1.46 & 0.00 & 0.03 & 2.46 & 0.70 & 1.26 \\
riverswim6 & 0.01 & 0.34 & 0.26 & 0.89 & 0.00 & 0.03 & 0.85 & 0.39 & 1.01 \\
taxi & 1.09 & 16.51 & 19.09 & 62.81 & 0.05 & 1.77 & 140.36 & 60.95 & 35.32 \\ \hline \hline
\end{tabular}
\end{table}

\begin{table}[t]
\centering
\caption{KL divergence and Burg entropy robust Bellman operator runtimes (in ms) on synthetic and benchmark instances.}
\begin{tabular}{l|rr|rr}
Instance & KL Fast & KL Mosek & Burg Fast & Burg Mosek \\
\hline
synthetic ($S=10$, $A=10$) & 0.25 & 10.08 & 0.10 & 7.08 \\
synthetic ($S=20$, $A=10$) & 0.43 & 19.11 & 0.18 & 14.61 \\
synthetic ($S=30$, $A=10$) & 0.68 & 27.79 & 0.25 & 20.53 \\
synthetic ($S=40$, $A=10$) & 0.84 & 37.48 & 0.32 & 27.21 \\
synthetic ($S=50$, $A=10$) & 1.06 & 45.64 & 0.39 & 38.46 \\
synthetic ($S=60$, $A=10$) & 1.56 & 58.21 & 0.53 & 46.54 \\
synthetic ($S=70$, $A=10$) & 1.56 & 70.63 & 0.57 & 56.58 \\
synthetic ($S=80$, $A=10$) & 1.80 & 76.71 & 0.62 & 62.12 \\
synthetic ($S=90$, $A=10$) & 2.77 & 93.38 & 0.78 & 70.33 \\
synthetic ($S=100$, $A=10$) & 2.96 & 96.83 & 0.88 & 78.24 \\
\hline
synthetic ($S=20$, $A=20$) & 0.62 & 40.39 & 0.31 & 30.34 \\
synthetic ($S=30$, $A=30$) & 1.36 & 99.82 & 0.65 & 77.02 \\
synthetic ($S=40$, $A=40$) & 2.02 & 181.45 & 1.03 & 145.24 \\
synthetic ($S=50$, $A=50$) & 2.99 & 369.29 & 1.60 & 302.96 \\
synthetic ($S=60$, $A=60$) & 4.19 & 577.15 & 2.21 & 649.81 \\
synthetic ($S=70$, $A=70$) & 5.80 & 937.38 & 3.04 & 875.56 \\
synthetic ($S=80$, $A=80$) & 7.27 & 1,417.59 & 3.81 & 1,330.96 \\
synthetic ($S=90$, $A=90$) & 9.58 & 2,530.39 & 4.79 & 2,248.03 \\
synthetic ($S=100$, $A=100$) & 10.83 & 3,245.97 & 5.99 & 3,593.97 \\
\hline
blackjack & 4.67 & 58.52 & 1.02 & 70.96 \\
capacity50 & 1.09 & 24.95 & 0.30 & 17.13 \\
chain10 & 0.14 & 2.10 & 0.04 & 1.97 \\
cliffwalking & 1.58 & 12.00 & 0.32 & 9.12 \\
forest50 & 0.63 & 7.11 & 0.15 & 6.95 \\
frozenlake4x4 & 0.31 & 5.34 & 0.09 & 4.51 \\
frozenlake8x8 & 1.14 & 24.01 & 0.30 & 17.05 \\
gridworld25 & 0.37 & 10.45 & 0.12 & 7.13 \\
inventory50 & 0.56 & 21.38 & 0.17 & 12.58 \\
machine20 & 0.17 & 2.51 & 0.05 & 2.42 \\
openspiel$\underline{~}$grid16 & 0.35 & 8.17 & 0.12 & 6.39 \\
perishable50 & 0.58 & 19.86 & 0.17 & 12.99 \\
riverswim20 & 0.27 & 3.15 & 0.07 & 3.20 \\
riverswim6 & 0.09 & 1.41 & 0.03 & 1.33 \\
taxi & 8.89 & 379.12 & 2.75 & 383.60 \\ \hline \hline
\end{tabular}
\end{table}

\end{document}